%% file: pipedreams-draft4.tex
\documentclass[11pt]{article}
\input{preamble.tex}

\usepackage{fullpage}
\usepackage[nomarkers,figuresonly,nofiglist]{endfloat}

\begin{document}

\title{Involution pipe dreams}

\author{
Zachary Hamaker  \\ University of Florida \\ {\tt zachary.hamaker@gmail.com}
\and 
Eric Marberg \\ HKUST \\ {\tt eric.marberg@gmail.com}
\and
Brendan Pawlowski \\ University of Southern California \\ {\tt br.pawlowski@gmail.com}
}

\date{}

\maketitle

\begin{abstract}
Involution Schubert polynomials represent cohomology classes of $K$-orbit closures in the complete flag variety, where $K$ is the orthogonal or symplectic group. We show they also represent $\Torus$-equivariant cohomology classes of subvarieties defined by upper-left rank conditions in the spaces of symmetric or skew-symmetric matrices. This geometry implies that these polynomials are positive combinations of monomials in the variables $x_i + x_j$, and we give explicit formulas of this kind as sums over new objects called involution pipe dreams. Our formulas are analogues of the Billey-Jockusch-Stanley formula for Schubert polynomials. In Knutson and Miller's approach to matrix Schubert varieties, pipe dream formulas reflect Gr\"obner degenerations of the ideals of those varieties, and we conjecturally identify analogous degenerations in our setting.
\end{abstract}


\setcounter{tocdepth}{2}

\section{Introduction}

One can identify the equivariant cohomology rings for the spaces of symmetric and skew-symmetric complex matrices with multivariate polynomial rings.
Under this identification, we show 
that the classes of certain natural subvarieties of (skew-)symmetric matrices are given by the \emph{involution Schubert polynomials} introduced by Wyser and Yong in \cite{wyser-yong-orthogonal-symplectic}.
These classes of varieties generalize various others studied in the settings of degeneracy loci and combinatorial commutative algebra, for instance the (skew-)symmetric determinantal varieties studied by Harris and Tu~\cite{harris-tu}.

Involution Schubert polynomials have a combinatorial formula for their monomial expansion~\cite{HMP1}.
As a consequence of our geometric results, they must also expand as sums of products of binomials $x_i + x_j$.
We give a combinatorial description of these expansions, which is a new analogue of the classic Billey-Jockusch-Stanley expansion for ordinary Schubert polynomials~\cite{BJK}.
This description is far more compact than the monomial expansion.
Our formulas involve novel objects that we call \emph{involution pipe dreams}.
Involution pipe dreams appear to be the fundamental objects necessary to replicate Knutson and Miller's program~\cite{knutson-miller} to understand our varieties from a commutative algebra perspective.

\subsection{Three flavors of matrix Schubert varieties}\label{intro1-sect}

Fix a positive integer $n$.
Let $\GL_n$ denote the general linear group of complex $n\times n$ invertible matrices, and write $\Borel$ and $\Borel^+$ for the Borel subgroups of lower- and upper-triangular matrices in $\GL_n$. Our work aims to extend what is known about the geometry of the $\Borel$-orbits on matrix space to symmetric and skew-symmetric matrix spaces.

We begin with some classical background.
Consider the \emph{type A flag variety} $\Fl{n} = \Borel \backslash \GL_n$.
The subgroup $\Borel^+$ acts on  $\Fl{n} $ with finitely many orbits, which are naturally indexed by permutations $w$ in the symmetric group  $S_n$
of permutations of $\{1,2,\dots,n\}$.
These orbits afford a CW decomposition of $\Fl{n}$, so the cohomology classes of their closures $X_w$, the \emph{Schubert varieties}, form a basis for the integral singular cohomology ring $H^*(\Fl{n})$. 
Borel's isomorphism explicitly identifies $H^*(\Fl{n})$ with a quotient of the 
polynomial ring $\ZZ[x_1,\dots,x_n]$, and the \emph{Schubert polynomials} $\fkS_w \in \ZZ[x_1,\dots,x_n]$ are (non-unique) representatives for the \emph{Schubert classes} $[X_w] \in H^*(\Fl{n})$.

The maximal torus $\Torus$ of diagonal matrices in $\GL_n$ also acts on $\Fl{n}$, so we can instead consider the equivariant cohomology ring $H^*_\Torus(\Fl{n})$.
Via an extension of Borel's isomorphism, this ring 
is isomorphic to a quotient of $\ZZ[x_1, \ldots, x_n, y_1, \ldots, y_n]$.
Lascoux and Sch\"utzenberger \cite{lascoux1982structure} introduced the \emph{double Schubert polynomials} $\fkS_w(x,y)$ to represent the 
equivariant classes $[X_w]_\Torus\in H^*_\Torus(\Fl{n})$.
These representatives are distinguished in the following sense.
 
Let $\Mat_n$ be the set of $n\times n$ complex matrices and write $\iota:\GL_n \hookrightarrow \Mat_n$ for the obvious inclusion.
The product group $ \Torus \times \Torus$ acts on $A\in \Mat_n$ by $(t_1, t_2) \cdot A = t_1 A t_2^{-1}$.
The \emph{matrix Schubert variety} of a permutation $w \in S_n$ is  $\MX_w = \overline{\iota(X_w)}$.
Since $M_n$ is $\Torus \times \Torus$-equivariantly contractible, $H^*_{\Torus \times \Torus}(\Mat_n) \cong H^*_{\Torus \times \Torus}(\text{point}) \cong \ZZ[x_1,\dots,x_n,y_1,\dots,y_n]$.
The launching point for Knutson and Miller's program is the following theorem:

\begin{thm}[\cite{knutson-miller}]
\label{t:fulton}
	For all $w \in S_n$, we have $\fkS_w(x,y) = [\MX_w] \in H^*_{\Torus \times \Torus}(\Mat_n)$.
\end{thm}

As mentioned in the historical notes at the end of \cite[Chpt. 15]{miller2004combinatorial}, Theorem~\ref{t:fulton} is equivalent to Fulton's characterization of each $\fkS_w(x,y)$ as the class of a certain degeneracy locus for vector bundle morphisms~\cite{fulton1992flags}.

Our results are related to the geometry of certain spherical varieties studied by Richardson and Springer in \cite{richardson-springer}.
Specifically, define the \emph{orthogonal group} $\O_n$ as the subgroup of $\GL_n$ preserving
a fixed nondegenerate symmetric bilinear form on $\CC^n$, and when $n$ is even define the \emph{symplectic group} $\Sp_n$ as the subgroup of $\GL_n$
preserving a fixed nondegenerate skew-symmetric bilinear form.

We consider the actions of $\O_n$ and $\Sp_n$ (when $n$ is even) on $\Fl{n}$.
The associated orbit closures $\iX_y$ and $\iXfpf_z$ are indexed by arbitrary involutions $y$ and fixed-point-free involutions $z$ in $S_n$.
Let $\kappa(y)$ denote the number of 2-cycles in an involution $y=y^{-1} \in S_n$.
Wyser and Yong \cite{wyser-yong-orthogonal-symplectic} constructed certain polynomials $\iS_y, \iSfpf_z \in \ZZ[x_1,\dots,x_n]$
and showed that 
the classes $[\iX_y]$ and $[\iXfpf_z]$ are represented in $H^*(\Fl{n})$ by $2^{\kappa(y)}\iS_y$ and $\iSfpf_z$.
We refer to $\iS_y$ and $\iSfpf_z$ as \emph{involution Schubert polynomials}; for their precise definitions, see Section~\ref{schub-sect}.

Write $\SM_n$ and $\SSM_n$ for the sets of symmetric and skew-symmetric $n\times n$ complex matrices.
Let $t \in \Torus$ act on these spaces by $t\cdot A = tAt$.
One can identify  the $\Torus$-equivariant cohomology rings of both spaces with $\ZZ[x_1, \ldots, x_n]$; see the discussion in Section~\ref{subsec:eq-cohom}.
For each involution $y \in S_n$, let $\iMX_y = \MX_y \cap \SM_n$.
Similarly, for each fixed-point-free involution $z \in S_n$, let $\iMXfpf_z = \MX_z \cap \SSM_n$.
Our first main result is a (skew-)symmetric analogue of Theorem~\ref{t:fulton}:

\begin{thm}
\label{thm:inv-schubert-classes} 
For all involutions $y$ and fixed-point-free involution $z$ in $S_n$, we have
\[
2^{\kappa(y)}\iS_y = [\iMX_y]   \in   H_\Torus^*(\SM_n) \quand \iSfpf_z = [\iMXfpf_z]  \in  H_\Torus^*(\SSM_n).
\]
Thus, involution Schubert polynomials are also equivariant cohomology representatives for symmetric and skew-symmetric matrix varieties.
\end{thm}

Our proof of this theorem appears in Section~\ref{sec:inv-matrix-schubert}.
An extension of Theorem~\ref{thm:inv-schubert-classes} to complex $K$-theory appears in~\cite{MP2019}.
Theorem~\ref{thm:inv-schubert-classes} was first  announced in a conference 
proceedings before the appearance of the preprint version of \cite{MP2019}, which precedes the
preprint version of this article. The proof of Theorem~\ref{thm:inv-schubert-classes} is a special
case of results of \cite{MP2019}.


\begin{rmk}
Another family of varieties in $\SM_n$ indexed by permutations in $S_n$ has been studied by Fink, Rajchgot and Sullivant~\cite{fink2016matrix}.
However, their varieties are cut out by northeast rank conditions, while 
$\iMX_y$ and $\iMXfpf_z$ are cut out by northwest rank conditions (see \eqref{imxy-eq} and \eqref{imxz-eq} in Section~\ref{ss:classes}).
The varieties in~\cite{fink2016matrix} are closely related to type C Schubert calculus and generally do not coincide with our $\iMX_y$ varieties.
\end{rmk}

\subsection{Three flavors of pipe dreams}

If $Z$ is a closed subvariety of $\SM_n$ or $\SSM_n$, then its $\Torus$-equivariant cohomology class is a positive integer combination of products of binomials $x_i + x_j$ (see Corollary~\ref{cor:positivity}). 
Our second main result gives a combinatorial description of such an expansion for $\iS_y$ and $\iSfpf_z$.

Let $[n] = \{1, 2, \ldots, n\}$
and $\rtriang_n =   \{(i,j) \in [n] \times [n] : i+j \leq n\}$.
Consider a subset $D\subseteq \rtriang_n$.
One associates to $D$ a \emph{wiring diagram}
by replacing the cells $(i,j) \in \rtriang_n$
by tiles of two types, given either by 
a crossing of two paths (drawn as a
$\raisebox{-1pt}{\tikz[scale=0.7]{ \draw (-0.2,0) -- (0.2,0); \draw (0,-0.2) -- (0,0.2);}}$
tile) if $(i,j) \in D$
or 
by two paths bending away from each other (drawn as a 
$\raisebox{-1pt}{\tikz[scale=0.7]{ \draw (-0.2,0) to[bend right] (0,0.2); \draw (0,-0.2) to[bend left] (0.2,0); }}$
tile)
if $(i,j) \notin D$.
Connecting the endpoints of adjacent tiles yields a union of $n$ continuously differentiable paths,
which we refer to as ``pipes.''
For example:
\begin{equation}\label{first-wiring-eq}
    D=\{(1,3),(2,1)\} 
    \qquad\text{corresponds to}
    \qquad
    \begin{tikzpicture}[scale=0.3,baseline=(c.base)]
        \node at (-0.5,3.5) {$\scriptstyle 1$};
        \node (c) at (-0.5,2.5) {$\scriptstyle 2$};
        \node at (-0.5,1.5) {$\scriptstyle 3$};
        \node at (-0.5,0.5) {$\scriptstyle 4$};
        \node at (0.5,4.5) {$\scriptstyle 1$};
        \node at (1.5,4.5) {$\scriptstyle 2$};
        \node at (2.5,4.5) {$\scriptstyle 3$};
        \node at (3.5,4.5) {$\scriptstyle 4$};
        \filldraw (0.5, 3.5) circle[radius=.6mm]; 
        \draw 
        (0,3.5) to[bend right] (0.5,4) 
        (0.5,3) to[bend left] (1,3.5); 
        \filldraw (1.5, 3.5) circle[radius=.6mm]; 
        \draw 
        (1,3.5) to[bend right] (1.5,4) 
        (1.5,3) to[bend left] (2,3.5); 
        \filldraw (2.5, 3.5) circle[radius=.6mm]; 
        \draw 
        (2,3.5) to (3,3.5) 
        (2.5,3) to (2.5,4); 
        \filldraw (3.5, 3.5) circle[radius=.6mm]; 
        \draw 
        (3,3.5) to[bend right] (3.5,4); 
        \filldraw (0.5, 2.5) circle[radius=.6mm]; 
        \draw 
        (0,2.5) to (1,2.5) 
        (0.5,2) to (0.5,3); 
        \filldraw (1.5, 2.5) circle[radius=.6mm]; 
        \draw 
        (1,2.5) to[bend right] (1.5,3) 
        (1.5,2) to[bend left] (2,2.5); 
        \filldraw (2.5, 2.5) circle[radius=.6mm]; 
        \draw 
        (2,2.5) to[bend right] (2.5,3); 
        \filldraw (0.5, 1.5) circle[radius=.6mm]; 
        \draw 
        (0,1.5) to[bend right] (0.5,2) 
        (0.5,1) to[bend left] (1,1.5); 
        \filldraw (1.5, 1.5) circle[radius=.6mm]; 
        \draw 
        (1,1.5) to[bend right] (1.5,2); 
        \filldraw (0.5, 0.5) circle[radius=.6mm]; 
        \draw 
        (0,0.5) to[bend right] (0.5,1); 
    \end{tikzpicture}
\end{equation}
\begin{defn}
A subset $D\subseteq \rtriang_n$ is a \emph{reduced pipe dream} if no two pipes in the associated wiring diagram
cross more than once. %
\end{defn}

This condition holds in the example \eqref{first-wiring-eq}.
Pipe dreams as described here were introduced by Bergeron and Billey \cite{bergeron-billey}, inspired by related diagrams of Fomin and Kirillov \cite{fomin-kirillov-yang-baxter}. Bergeron and Billey originally referred to pipe dreams as \emph{reduced-word compatible sequence graphs} or \emph{rc-graphs} for short.

A reduced pipe dream $D$ determines a permutation $w \in S_n$
 in the following way.
Label the left endpoints of the pipes in $D$'s wiring diagram by $1, 2, \ldots, n$ from top to bottom, and the top endpoints by $1, 2, \ldots, n$ from left to right. Then the associated permutation
$w \in S_n$ is the element such that the pipe with left endpoint $i$ has top endpoint $w(i)$. For instance, the permutation of  $D=\{(1,3),(2,1)\}$ is $w 
= 1423 \in S_4$.
Let $\RP(w)$ denote the set of all reduced pipe dreams associated to $w \in S_n$.

Pipe dreams are of interest for their role in formulas 
for $\fkS_w$ and $\fkS_w(x,y)$.
Lascoux and Sch\"utzenberger's original definition of  these Schubert polynomials  in \cite{lascoux1983symmetry}
 is recursive in terms of \emph{divided difference operators}.
However, by results of Fomin and Stanley \cite[\S4]{fomin1994schubert} we also have 
\begin{equation} \label{eq:double-schubert-def}
    \fkS_w = \sum_{D \in \RP(w)} \prod_{(i,j) \in D} x_i \qquand  \fkS_w(x,\yvar) = \sum_{D \in \RP(w)} \prod_{(i,j) \in D} (x_i - \yvar_j).
\end{equation}
The first identity is the Billey-Jockusch-Stanley formula for Schubert polynomials from \cite{BJK}.

There are analogues of this formula for the involution Schubert polynomials $\iS_y$ and $\iSfpf_z$,
which involve the following new classes of pipe dreams.
A reduced pipe dream $D\subseteq \rtriang_n$
is \emph{symmetric} if $(i,j) \in D$ implies $(j,i) \in D$,
and \emph{almost-symmetric} if both of the following properties hold:
\begin{itemize}
\item If $(i, j) \in D$ where $i<j$ then  $(j,i) \in D$.
\item If $(j, i) \in D$ where $i<j$ but $(i, j) \notin D$, then the pipes crossing at $(j, i)$
in the wiring diagram of $D$ are also the pipes that avoid each other at $(i, j)$.
\end{itemize}
Equivalently, 
 $D$ is almost-symmetric if it is as symmetric as possible while respecting the condition that no two pipes cross twice, and any violation of symmetry forced by this condition takes the form of a crossing $(j,i)$ below the diagonal rather than at the transposed position $(i,j)$.

Let $\I_n=\{ w \in S_n : w=w^{-1}\}$ and write $\Ifpf_n$ for the subset of fixed-point-free elements of $\I_n$.
Note that $n$ must be even for $\Ifpf_n$ to be non-empty.
Also let
\[
\ltriangeq_n = \{(j,i) \in [n] \times [n] : i \leq j\}
\qquand
\ltriangneq_n = \{(j,i) \in [n] \times [n] : i < j\}.
\]

\begin{defn} \label{defn:inv-pipe-dreams-1} 
The set of \emph{involution pipe dreams} for $y \in \I_n$ is
\begin{equation*}
\IP(y) = \{D \cap \ltriangeq_n : D \in \RP(y) \text{ is almost-symmetric}\}.
\end{equation*}
The set of \emph{fpf-involution pipe dreams} for $z \in \Ifpf_{n}$ is
\begin{equation*}
\FP(z) = \{D \cap \ltriangneq_n : D \in \RP(z) \text{ is symmetric}\}.
\end{equation*}
By convention, (fpf-)involution pipe dreams are always instances of reduced pipe dreams.
It would be more precise to call our objects ``reduced involution pipe dreams,'' 
but since we will never consider any pipe dreams that are unreduced, we opt for 
more concise terminology.
\end{defn}

We can now state our second main result, which will reappear
as Theorems~\ref{inv-pipe-schubert-thm} and \ref{fpf-pipe-schubert-thm}.

\begin{thm} \label{thm:inv-pipe-dream-formula} If $y \in \I_n$ and $z \in \Ifpf_n$ then
\[
 \iS_y = \sum_{D \in \IRP(y)} \prod_{(i,j) \in D} 2^{-\delta_{ij}}(x_i + x_j) 
 \quand 
 \iSfpf_z = \sum_{D \in \FP(z)} \prod_{(i,j) \in D} (x_i + x_j)
\]
where $\delta_{ij}$ denotes the usual Kronecker delta function.
\end{thm}

\begin{ex}
    The involution $y = 1432 = (2,4) \in \I_4$ has five reduced pipe dreams:
    
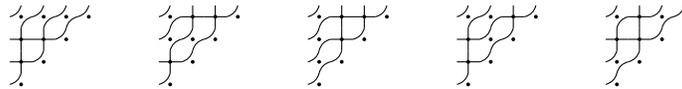
\begin{center}
        \begin{tikzpicture}[scale=0.3]
        \filldraw (0.5, 3.5) circle[radius=.6mm]; 
        \draw 
        (0,3.5) to[bend right] (0.5,4) 
        (0.5,3) to[bend left] (1,3.5); 
        \filldraw (1.5, 3.5) circle[radius=.6mm]; 
        \draw 
        (1,3.5) to[bend right] (1.5,4) 
        (1.5,3) to[bend left] (2,3.5); 
        \filldraw (2.5, 3.5) circle[radius=.6mm]; 
        \draw 
        (2,3.5) to[bend right] (2.5,4) 
        (2.5,3) to[bend left] (3,3.5); 
        \filldraw (3.5, 3.5) circle[radius=.6mm]; 
        \draw 
        (3,3.5) to[bend right] (3.5,4); 
        \filldraw (0.5, 2.5) circle[radius=.6mm]; 
        \draw 
        (0,2.5) to (1,2.5) 
        (0.5,2) to (0.5,3); 
        \filldraw (1.5, 2.5) circle[radius=.6mm]; 
        \draw 
        (1,2.5) to (2,2.5) 
        (1.5,2) to (1.5,3); 
        \filldraw (2.5, 2.5) circle[radius=.6mm]; 
        \draw 
        (2,2.5) to[bend right] (2.5,3); 
        \filldraw (0.5, 1.5) circle[radius=.6mm]; 
        \draw 
        (0,1.5) to (1,1.5) 
        (0.5,1) to (0.5,2); 
        \filldraw (1.5, 1.5) circle[radius=.6mm]; 
        \draw 
        (1,1.5) to[bend right] (1.5,2); 
        \filldraw (0.5, 0.5) circle[radius=.6mm]; 
        \draw 
        (0,0.5) to[bend right] (0.5,1); 
        \end{tikzpicture} \qquad  
       \begin{tikzpicture}[scale=0.3]
        \filldraw (0.5, 3.5) circle[radius=.6mm]; 
        \draw 
        (0,3.5) to[bend right] (0.5,4) 
        (0.5,3) to[bend left] (1,3.5); 
        \filldraw (1.5, 3.5) circle[radius=.6mm]; 
        \draw 
        (1,3.5) to (2,3.5) 
        (1.5,3) to (1.5,4); 
        \filldraw (2.5, 3.5) circle[radius=.6mm]; 
        \draw 
        (2,3.5) to (3,3.5) 
        (2.5,3) to (2.5,4); 
        \filldraw (3.5, 3.5) circle[radius=.6mm]; 
        \draw 
        (3,3.5) to[bend right] (3.5,4); 
        \filldraw (0.5, 2.5) circle[radius=.6mm]; 
        \draw 
        (0,2.5) to[bend right] (0.5,3) 
        (0.5,2) to[bend left] (1,2.5); 
        \filldraw (1.5, 2.5) circle[radius=.6mm]; 
        \draw 
        (1,2.5) to[bend right] (1.5,3) 
        (1.5,2) to[bend left] (2,2.5); 
        \filldraw (2.5, 2.5) circle[radius=.6mm]; 
        \draw 
        (2,2.5) to[bend right] (2.5,3); 
        \filldraw (0.5, 1.5) circle[radius=.6mm]; 
        \draw 
        (0,1.5) to (1,1.5) 
        (0.5,1) to (0.5,2); 
        \filldraw (1.5, 1.5) circle[radius=.6mm]; 
        \draw 
        (1,1.5) to[bend right] (1.5,2); 
        \filldraw (0.5, 0.5) circle[radius=.6mm]; 
        \draw 
        (0,0.5) to[bend right] (0.5,1); 
        \end{tikzpicture} \qquad 
        \begin{tikzpicture}[scale=0.3]
        \filldraw (0.5, 3.5) circle[radius=.6mm]; 
        \draw 
        (0,3.5) to[bend right] (0.5,4) 
        (0.5,3) to[bend left] (1,3.5); 
        \filldraw (1.5, 3.5) circle[radius=.6mm]; 
        \draw 
        (1,3.5) to (2,3.5) 
        (1.5,3) to (1.5,4); 
        \filldraw (2.5, 3.5) circle[radius=.6mm]; 
        \draw 
        (2,3.5) to (3,3.5) 
        (2.5,3) to (2.5,4); 
        \filldraw (3.5, 3.5) circle[radius=.6mm]; 
        \draw 
        (3,3.5) to[bend right] (3.5,4); 
        \filldraw (0.5, 2.5) circle[radius=.6mm]; 
        \draw 
        (0,2.5) to[bend right] (0.5,3) 
        (0.5,2) to[bend left] (1,2.5); 
        \filldraw (1.5, 2.5) circle[radius=.6mm]; 
        \draw 
        (1,2.5) to (2,2.5) 
        (1.5,2) to (1.5,3); 
        \filldraw (2.5, 2.5) circle[radius=.6mm]; 
        \draw 
        (2,2.5) to[bend right] (2.5,3); 
        \filldraw (0.5, 1.5) circle[radius=.6mm]; 
        \draw 
        (0,1.5) to[bend right] (0.5,2) 
        (0.5,1) to[bend left] (1,1.5); 
        \filldraw (1.5, 1.5) circle[radius=.6mm]; 
        \draw 
        (1,1.5) to[bend right] (1.5,2); 
        \filldraw (0.5, 0.5) circle[radius=.6mm]; 
        \draw 
        (0,0.5) to[bend right] (0.5,1); 
        \end{tikzpicture} \qquad 
        \begin{tikzpicture}[scale=0.3]
        \filldraw (0.5, 3.5) circle[radius=.6mm]; 
        \draw 
        (0,3.5) to[bend right] (0.5,4) 
        (0.5,3) to[bend left] (1,3.5); 
        \filldraw (1.5, 3.5) circle[radius=.6mm]; 
        \draw 
        (1,3.5) to[bend right] (1.5,4) 
        (1.5,3) to[bend left] (2,3.5); 
        \filldraw (2.5, 3.5) circle[radius=.6mm]; 
        \draw 
        (2,3.5) to (3,3.5) 
        (2.5,3) to (2.5,4); 
        \filldraw (3.5, 3.5) circle[radius=.6mm]; 
        \draw 
        (3,3.5) to[bend right] (3.5,4); 
        \filldraw (0.5, 2.5) circle[radius=.6mm]; 
        \draw 
        (0,2.5) to (1,2.5) 
        (0.5,2) to (0.5,3); 
        \filldraw (1.5, 2.5) circle[radius=.6mm]; 
        \draw 
        (1,2.5) to[bend right] (1.5,3) 
        (1.5,2) to[bend left] (2,2.5); 
        \filldraw (2.5, 2.5) circle[radius=.6mm]; 
        \draw 
        (2,2.5) to[bend right] (2.5,3); 
        \filldraw (0.5, 1.5) circle[radius=.6mm]; 
        \draw 
        (0,1.5) to (1,1.5) 
        (0.5,1) to (0.5,2); 
        \filldraw (1.5, 1.5) circle[radius=.6mm]; 
        \draw 
        (1,1.5) to[bend right] (1.5,2); 
        \filldraw (0.5, 0.5) circle[radius=.6mm]; 
        \draw 
        (0,0.5) to[bend right] (0.5,1); 
        \end{tikzpicture} \qquad 
        \begin{tikzpicture}[scale=0.3]
        \filldraw (0.5, 3.5) circle[radius=.6mm]; 
        \draw 
        (0,3.5) to[bend right] (0.5,4) 
        (0.5,3) to[bend left] (1,3.5); 
        \filldraw (1.5, 3.5) circle[radius=.6mm]; 
        \draw 
        (1,3.5) to (2,3.5) 
        (1.5,3) to (1.5,4); 
        \filldraw (2.5, 3.5) circle[radius=.6mm]; 
        \draw 
        (2,3.5) to[bend right] (2.5,4) 
        (2.5,3) to[bend left] (3,3.5); 
        \filldraw (3.5, 3.5) circle[radius=.6mm]; 
        \draw 
        (3,3.5) to[bend right] (3.5,4); 
        \filldraw (0.5, 2.5) circle[radius=.6mm]; 
        \draw 
        (0,2.5) to (1,2.5) 
        (0.5,2) to (0.5,3); 
        \filldraw (1.5, 2.5) circle[radius=.6mm]; 
        \draw 
        (1,2.5) to (2,2.5) 
        (1.5,2) to (1.5,3); 
        \filldraw (2.5, 2.5) circle[radius=.6mm]; 
        \draw 
        (2,2.5) to[bend right] (2.5,3); 
        \filldraw (0.5, 1.5) circle[radius=.6mm]; 
        \draw 
        (0,1.5) to[bend right] (0.5,2) 
        (0.5,1) to[bend left] (1,1.5); 
        \filldraw (1.5, 1.5) circle[radius=.6mm]; 
        \draw 
        (1,1.5) to[bend right] (1.5,2); 
        \filldraw (0.5, 0.5) circle[radius=.6mm]; 
        \draw 
        (0,0.5) to[bend right] (0.5,1); 
        \end{tikzpicture}  
    \end{center}
    Only the last two of these are almost-symmetric, so 
    $|\IP(y)| = 2$
    and     Theorem~\ref{thm:inv-pipe-dream-formula} reduces to the formula  $\iS_{y} = (x_2+x_1)(x_3+x_1) + (x_2+x_1)(x_2 + x_2)/2 = (x_2 + x_1)(x_3 + x_1 +x_2)$.
    The monomial expansion has six terms, as opposed to two.
    In general, the expansion in Theorem~\ref{thm:inv-pipe-dream-formula} uses roughly a factor of $2^{\deg \iS_y}$ fewer terms.
    
\end{ex}

%

\begin{rmk}
There is an alternate path towards establishing the fact that the class of a matrix Schubert variety is represented by the weighted sum of 
reduced pipe dreams.
The defining ideal of $\MX_w$ has a simple set of generators due to Fulton~\cite{fulton1992flags}.
Knutson and Miller showed that Fulton's generators form a Gr\"obner basis with respect to any anti-diagonal term order~\cite{knutson-miller}.
The Gr\"obner degeneration of this ideal decomposes into a union of coordinate subspaces indexed by reduced pipe dreams.
Our hope is that a similar program can be implemented in the (skew-)symmetric setting, which would give a geometric proof of Theorem~\ref{thm:inv-pipe-dream-formula}. 
We discuss this in greater detail in Section~\ref{sec:ideals}.
\end{rmk}

In addition to Theorem~\ref{thm:inv-pipe-dream-formula}, 
we also prove a number of results about the properties of involution pipe dreams.
An outline of the rest of this article is as follows.

Section~\ref{prelim-sect} contains some preliminaries on involution Schubert polynomials 
along with a proof of Theorem~\ref{thm:inv-schubert-classes}.
In Section~\ref{sec:reduced-words}, we give several equivalent characterizations of $\IP(y)$ and $\FP(z)$ in terms of {reduced words} for 
permutations.
Section~\ref{sec:inv-polynom} contains our proof of Theorem~\ref{thm:inv-pipe-dream-formula},
which uses ideas from recent work of Knutson~\cite{Knutson} along with certain transition equations for $\iS_y$ and $\iSfpf_z$ given in \cite{HMP3}.
In Section~\ref{sec:generating-pipe-dreams} we show that 
both families of involution pipe dreams are obtained from distinguished ``bottom'' elements by repeatedly applying certain simple transformations.
These transformations are extensions of the \emph{ladder moves} for pipe dreams described by Bergeron and Billey in \cite{bergeron-billey}.
In Section~\ref{sec:future}, finally, we describe several related open problems and conjectures.

\subsection*{Acknowledgements}

The second author was partially supported by Hong Kong RGC Grant ECS 26305218.
We thank Allen Knutson for explaining to us his proof of \eqref{eq:double-schubert-def} prior to the appearance of~\cite{Knutson}.

\section{Schubert polynomials and matrix varieties}\label{prelim-sect}

Everywhere in this paper,
$n$ denotes a fixed positive integer. 
For convenience, we realize the symmetric group $S_n$
as the group of permutations of $\PP = \{1,2,3,\dots\}$ fixing all $i>n$,
so that there is an automatic inclusion $S_n \subset S_{n+1}$.
In this section, we present some relevant background on involution Schubert polynomials
 and equivariant cohomology, and then prove Theorem~\ref{thm:inv-schubert-classes}.

\subsection{Involution Schubert polynomials}\label{schub-sect}

To start, we provide a succinct definition of $\iS_y$ and $\iSfpf_z$ in terms of the ordinary Schubert polynomials $\fkS_w$ given by \eqref{eq:double-schubert-def}.
Let $s_i = (i,i+1) \in S_n$ for each $i \in [n-1]$. 
A \emph{reduced word} for $w \in S_n$ is a minimal-length sequence $a_1a_2 \cdots a_l$ such that $w=s_{a_1}s_{a_2} \cdots s_{a_l}$. Let $\R(w)$ denote the set of reduced words for $w$. 
The \emph{length} $\ell(w)$ of $w \in S_n$ is the  length of any word in $\R(w)$. 
One has $\ell(ws_i)=\ell(w)+1>\ell(w)$ if and only if $w(i)<w(i+1)$.

\begin{prop}[{\cite[Thm. 7.1]{Humphreys}}]
\label{demazure-prop}
There is a unique associative operation $\circ : S_n \times S _n \to S_n$,
called the \emph{Demazure product},
with $s_i \circ s_i = s_i$ for all $i \in [n-1]$ and $v\circ w = vw$
for all $v,w \in S_n$ with $\ell(vw) = \ell(v)+\ell(w)$.
\end{prop}

An \emph{involution word} for $y \in \I_n = \{ w \in S_n : w=w^{-1}\}$ is a minimal-length word $a_1a_2 \cdots a_l$ with 
\be\label{iii-eq}
y = s_{a_l} \circ \cdots \circ s_{a_2}\circ  s_{a_1} \circ 1 \circ s_{a_1} \circ s_{a_2} \circ \cdots \circ s_{a_l}.
\ee
Note that we could replace $s_{a_1}\circ 1 \circ s_{a_1}$ in this 
expression by $s_{a_1}  = s_{a_1}\circ s_{a_1}= s_{a_1}\circ 1 \circ s_{a_1}$.
An \emph{atom} for $y \in \I_n$ is a minimal-length permutation $w \in S_n$ with
$
y =  w^{-1} \circ w.
$
Let $\iR(y)$ be the set of  involution words for $y \in \I_n$ and 
let $\A(y)$ be the set of atoms for $y$.
The associativity of the Demazure product implies that 
$
\label{bigsqcup-eq}
 \iR(y) = \bigsqcup_{w \in \A(y)} \R(w)
$.  

\begin{ex}
If $y = 1432 $ then $\iR(y) = \{  23, 32\}$ and $\A(y) = \{1342, 1423\}$.
\end{ex}

One can show that 
$ \I_n = \{ w^{-1} \circ w : w \in S_n \}
$, so  $\iR(y)$ and $\A(y)$ are nonempty for all $y \in \I_n$.
 Involution words are a special case of a more general construction of Richardson and Springer \cite{richardson-springer}, and have been studied by various authors \cite{can-joyce-wyser, hansson-hultman, hu-zhang, hultman-twisted-involutions}. Our notation follows \cite{HMP2,HMP1}.

\begin{defn}\label{inv-sch-def}
The \emph{involution Schubert polynomial} of $y \in \I_n$ is 
$
\iS_y = \sum_{w \in \A(y)} \fkS_w.
$
\end{defn}

Wyser and Yong \cite{wyser-yong-orthogonal-symplectic}
originally defined these polynomials 
  recursively using divided difference operators; work of Brion \cite{brion} implies that our definition agrees with theirs.
For a detailed explanation of the equivalence among these definitions, see \cite{HMP1}.

\begin{ex} If $z = 1432 \in \I_4$ then $\A(z) = \{1342, 1423\}$ and
\begin{equation*}
\iS_{z} = \fkS_{1342} + \fkS_{1423} = (x_2 x_3 + x_1 x_3 + x_1 x_2) + (x_2^2 + x_1 x_2 + x_1^2).
\end{equation*}
\end{ex}

Assume $n$ is even, so that $\Ifpf_n = \{ z \in \I_n: i \neq z(i)\text{ for all $i \in [n]$}\}$ is nonempty, and let 
\[\idfpf_n = 2143\dots  n\ n{-}1
= s_1s_3\cdots s_{n-1} \in \Ifpf_n.\]
An \emph{fpf-involution word} for $z \in \Ifpf_{n}$ is a minimal-length word $a_1a_2 \cdots a_l$ with
\[
z = s_{a_l}  \cdots   s_{a_2}   s_{a_1}  \idfpf_n   s_{a_1}   s_{a_2} \cdots  s_{a_l}
.\]
This formulation avoids the Demazure product, but there is an equivalent definition that more closely parallels \eqref{iii-eq}.
Namely, by  \cite[Cor. 2.6]{HMP2},
an {fpf-involution word} for $z \in \Ifpf_{n}$ is also a minimal-length word $a_1a_2 \cdots a_l$ with
$
z = s_{a_l}  \circ \cdots  \circ s_{a_2}   \circ s_{a_1} \circ   \idfpf_n  \circ  s_{a_1}  \circ  s_{a_2} \circ  \cdots  \circ  s_{a_l}
$.

An \emph{fpf-atom} for $z \in \Ifpf_n$ is 
a minimal length permutation  $w \in S_n$ with 
$z = w^{-1}  \idfpf_n  w.$
Let $\Afpf(z)$ be the set of fpf-atoms for $z$,
and let 
$
 \iRfpf(z) $ be the set of fpf-involution words for $z$. 
The basic properties of reduced words imply that 
$
 \iRfpf(z) = \bigsqcup_{w \in \Afpf(z)} \R(w).
$

\begin{ex}
If $z = 4321 $ then $\iRfpf(z) = \{  23, 21\}$ and $\Afpf(z) = \{1342, 3124\}$.
\end{ex}

Note that $a_1a_2\cdots a_l $ belongs to $\iRfpf(z)$  
if and only if $135\cdots(n-1)a_1a_2\cdots a_l $ belongs to $\iR(z)$.
Moreover, if $z \in \Ifpf_n$ then 
$\iRfpf(z) = \iRfpf(zs_{n+1}) $ and $ \Afpf(z) = \Afpf(zs_{n+1}).$

Fpf-involution words are special cases of reduced words for \emph{quasiparabolic sets}
 \cite{RainsVazirani}.
Since $\Ifpf_n$ is a single $S_n$-conjugacy class,
each $z \in \Ifpf_n$
has at least one fpf-involution word and fpf-atom.

\begin{defn}\label{fpf-inv-sch-def}
The \emph{fpf-involution Schubert polynomial} of $z \in \Ifpf_n$ is 
$
\iSfpf_z = \sum_{w \in \Afpf(z)} \fkS_w.
$
\end{defn}

These polynomials were also introduced in \cite{wyser-yong-orthogonal-symplectic}. 
Note that $\iSfpf_{z} =\iSfpf_{zs_{n+1}}$ for all $z \in \Ifpf_n$.

\begin{ex} If $z = 532614 \in \Ifpf_6$ then $\Afpf(z) = \{13452, 31254\}$ and
\begin{equation*}
\ba
\iSfpf_{z} &= \fkS_{13452} + \fkS_{31254} 
= (x_2 x_3 x_4 + x_1 x_3 x_4 + x_1 x_2 x_4 + x_1 x_2 x_3) + (x_1^2 x_4 + x_1^2 x_3 + x_1^2 x_2 + x_1^3).
\ea
\end{equation*}
\end{ex}

\subsection{Torus-equivariant cohomology}
\label{subsec:eq-cohom}

Suppose $V$ is a finite-dimensional rational representation of a torus $\Torus \simeq (\CC^\times)^n$. A character $\lambda \in \Hom(\Torus, \CC^\times)$ is a \emph{weight} of $V$ if the \emph{weight space} $V_\lambda = \{v \in V : \text{$tv = \lambda(t)v$ for all $t \in \Torus$}\}$ is nonzero. Any nonzero $v \in V_{\lambda}$ is a \emph{weight vector}, and $V$ has a basis of weight vectors. Let $\wt(V)$ denote the set of weights of $V$. After fixing an isomorphism $\Torus \simeq (\CC^\times)^n$, we identify the character $(t_1, \ldots, t_n) \mapsto t_1^{a_1} \cdots t_n^{a_n}$ with the linear polynomial $a_1 x_1 + \cdots + a_n x_n \in \ZZ[x_1, \ldots, x_n]$.

The \emph{equivariant cohomology ring} $H_\Torus(V)$ is isomorphic to $\ZZ[x_1, \ldots, x_n]$, an identification we make without comment from now on. Each $\Torus$-invariant subscheme $X \subseteq V$ has an associated class $[X] \in H_\Torus(V)$, which we describe following \cite[Chpt. 8]{miller2004combinatorial}. 

First, if $X$ is a linear subspace then we define $[X] = \prod_{\lambda \in \wt(X)} \lambda$, where we identify each character $\lambda$ with a linear polynomial as above. More generally, fix a basis of weight vectors of $V$, and let $z_1, \ldots, z_n \in V^*$ be the dual basis; this determines an isomorphism $\CC[V] = \Sym(V^*) \simeq \CC[z_1, \ldots, z_n]$. 

Choose a term order on monomials in $z_1, \ldots, z_n$, and let $\init(I)$ denote the ideal generated by the leading terms of all members of a given set $I \subseteq \CC[V]$. Given that $\init(I)$ is a monomial ideal, one can show that each of its associated primes $\mathfrak{p}$ is also a monomial ideal, and hence of the form $\langle z_{i_1}, \ldots, z_{i_r} \rangle$. The corresponding subscheme $Z(\mathfrak{p})$ is a $\Torus$-invariant linear subspace of $V$. Now define
\begin{equation} \label{eq:equivariant-class}
    [X] = \sum_{\mathfrak{p}} \mult_{\mathfrak{p}}(\init I(X)) [Z(\mathfrak{p})]
\end{equation}
where $I(X)$ is the ideal of $X$ and $\mathfrak{p}$ runs over the associated primes of $\init I(X)$.

\subsection{Classes of involution matrix Schubert varieties} \label{sec:inv-matrix-schubert}
\label{ss:classes}

The matrix Schubert varieties in Theorem~\ref{t:fulton} can be described in terms of rank conditions, namely:
\[
\MX_w = \{A \in \Mat_n: \text{$\rank A_{[i][j]} \leq \rank w_{[i][j]}$ for $i,j \in [n]$}\},
\]
where $\Mat_n$ is the variety of $n\times n$ matrices, $A_{[i][j]}$ denotes the upper-left $i \times j$ corner of $A \in \Mat_n$, and we identify $w \in S_n$ with the $n\times n$ permutation matrix having $1$'s in positions $(i,w(i))$.

The varieties $\iMX_y$ and $\iMXfpf_z$ from Theorem~\ref{thm:inv-schubert-classes} can be reformulated in a similar way.
Specifically,
we define the \emph{involution matrix Schubert variety} of $y \in \I_n$ by
\begin{equation}\label{imxy-eq}
    \iMX_y = \MX_y \cap \SM_n = \{A \in \SM_n : \text{$\rank A_{[i][j]} \leq \rank y_{[i][j]}$ for $i,j \in [n]$}\},
\end{equation}
where $\SM_n$ is the subvariety of symmetric matrices in $\Mat_n$.
When $n$ is even, we define the \emph{fpf-involution matrix Schubert variety} of $z \in \Ifpf_n$ by 
\begin{equation}\label{imxz-eq}
    \iMXfpf_z = \MX_z \cap \SSM_n = \{A \in \SSM_n : \text{$\rank A_{[i][j]} \leq \rank z_{[i][j]}$ for $i,j \in [n]$}\},
\end{equation}
where $\SSM_n$ is the subvariety of  skew-symmetric matrices in $\Mat_n$.

\begin{ex}
Suppose $y = 132 =\left[\begin{smallmatrix} 1 & 0 & 0 \\ 0 & 0 & 1 \\ 0 & 1 & 0 \end{smallmatrix}\right]\in \I_3$. Setting $R_{ij} = \rank y_{[i][j]}$, we have $R = \left[\begin{smallmatrix} 1 & 1 & 1 \\ 1 & 1 & 2\\ 1 & 2 & 3 \end{smallmatrix}\right]$. The conditions $\rank A_{[i][j]} \leq R_{ij}$ for $i,j \in [3]$ defining $\iMX_y$ are all implied by the single condition $\rank A_{[2][2]} \leq R_{22} = 1$. Thus, $\iMX_y = \left\{\left[\begin{smallmatrix} z_{11} & z_{21} & z_{31} \\ z_{21} & z_{22} & z_{32} \\ z_{31} & z_{32} & z_{33} \end{smallmatrix} \right] : z_{11}z_{22} - z_{21}^2 = 0\right\}$.
\end{ex}

Let $\Torus \subseteq \GL_n$ be the usual torus of invertible diagonal matrices. Recall that $\kappa(y) = |\{ i : y(i) < i\}|$ for $y \in \I_n$, and that $ \Torus$ acts on matrices in $ \Mat_n$ by $t\cdot A = tA$ and on symmetric matrices in $ \SM_n$ by $t \cdot A = tAt$.
We can now prove Theorem~\ref{thm:inv-schubert-classes}, which states that 
if $y \in \I_n$ and $z \in \Ifpf_n$ then
$2^{\kappa(y)}\iS_y = [\iMX_y] \in H_\Torus^*(\SM_n)$  while $\iSfpf_z = [\iMXfpf_z] \in H_\Torus^*(\SSM_n)$.


\begin{rmk}
\label{r:K-theory}
It is possible, though a little cumbersome, to derive Theorem~\ref{thm:inv-schubert-classes} from \cite[Thm. 2.17 and Lem. 3.1]{MP2019}, which provide a similar statement in complex $K$-theory. 
We originally announced Theorem~\ref{thm:inv-schubert-classes} in an extended abstract for this paper which
preceded the appearance of \cite{MP2019}.
However, as the argument below is similar to the proofs of the results in \cite{MP2019}, we will be somewhat curt here in our presentation of the details.
\end{rmk}

For $w = w_1 \dots w_n \in S_n$, let $w \times 1^k = w_1 \dots w_n\ n{+}1 \dots n{+}k \in S_{n+k}$.
Similarly, for $n$ even define $w \times (21)^k = w \times 1^{2k} \cdot (\idfpf_n \cdot \idfpf_{n+2k})$.
Our proof of Theorem~\ref{thm:inv-schubert-classes} relies on the following characterizations of $\iS_y$ and $\iSfpf_z$:

\begin{thm}[{\cite[Thm. 2]{wyser-yong-orthogonal-symplectic}}]
\label{t:stable}
	If $y \in \I_n$ and $z \in \Ifpf_n$, then $2^{\kappa(y)}\iS_y$ and $\iSfpf_z$ are the unique representatives for $[\iX_y]$ and $[\iXfpf_z]$ with $2^{\kappa(y)}\iS_y = 2^{\kappa(y)}\iS_{y\times 1^k}$ and $\iSfpf_z = \iSfpf_{z \times (21)^k}$ for all $k \geq 1$.
\end{thm}

\begin{proof}[Proof of Theorem~\ref{thm:inv-schubert-classes}]
    If $X$ and $Y$ are complex varieties with $\Torus$-actions, and $f : X \to Y$ is a $\Torus$-equivariant morphism, then there is a pullback homomorphism $f^* : H_\Torus^*(Y) \to H_\Torus^*(X)$. If $f$ is a flat morphism (e.g., an inclusion of an open subset, a projection of a fiber bundle, or a composition of flat morphisms), then $f^*([Z]) = [f^{-1}(Z)]$ for any subscheme $Z \subseteq Y$.

 Because $\Torus$ acts freely on $\GL_n$ and since $\Torus\backslash \GL_n \twoheadrightarrow \Borel \backslash \GL_n \simeq \Fl{n}$ is a homotopy equivalence (see, e.g., \cite[\S8.1]{McGovern}), one has $H_\Torus^*(\GL_n) \simeq H^*(\Torus \backslash \GL_n) \simeq H^*(\Fl{n})$. 
 If $Z \subseteq \GL_n$ is a $\Borel$-invariant subvariety, then $[Z] \in H_\Torus^*(\GL_n)$ corresponds to the class of $\Borel\backslash Z = \{\Borel g : g \in Z\}$ in $H^*(\Fl{n})$.
    Fix $y \in \I_n$ and define $\sigma : \GL_n \to \SM_n$ by $\sigma(g) = gg^T$. Let $\iota : \GL_n \hookrightarrow M_n$ be the obvious inclusion and consider the diagram 
    \begin{equation}\label{cd-eq}
    \begin{CD}
        H_\Torus^*(\SM_n) @>>\sigma^*> H_\Torus^*(\GL_n) @>>\sim> H^*(\Fl{n})\\
        @AA\rotatebox{90}{$\sim$}A                     @AA\iota^*A\\
        \ZZ[x_1, \ldots, x_n] @>>\sim> H_\Torus^*(M_n)
    \end{CD}
    \end{equation}
    Realize $\O_n$ as the group $\{g \in \GL_n : gg^T = 1\}$. The map $\sigma$ is flat because it is the composition $\GL_n \twoheadrightarrow \GL_n/\O_n \hookrightarrow \SM_n$, where the second map sends $g\O_n \mapsto gg^T$ and may be identified with the open inclusion $\GL_n \cap \SM_n \hookrightarrow \SM_n$. For fixed $i \in [n]$, one checks using the prescription of $\S \ref{subsec:eq-cohom}$ that $2x_i$ represents both the class of $Z = \{A \in \SM_n : A_{ii} = 0\}$ in $H_\Torus^*(\SM_n)$ and the class of $Z' = \{A \in M_n : (AA^T)_{ii} = 0\}$ in $H_\Torus^*(M_n)$. Since $\sigma^*[Z] = [\sigma^{-1}(Z)] = [\iota^{-1}(Z')] = \iota^*[Z']$, this calculation implies that \eqref{cd-eq} commutes.

    Now set
    $        \iX_y = \Borel \backslash \sigma^{-1}(\iMX_y) = \{\Borel g  \in \Fl{n} : \rank (gg^T)_{[i][j]} \leq \rank y_{[i][j]} \text{ for $i,j \in [n]$}\},$
    so that the path through the upper-left corner of \eqref{cd-eq} sends the polynomial $[\iMX_y]$ to $[\iX_y]$. The variety $\iX_y$ is the closure of an $\O_n$-orbit on $\Fl{n}$ \cite[\S2.1.2]{wyser-degeneracy-loci}.
    The path through the lower-right corner of \eqref{cd-eq} is simply the classical Borel map $\ZZ[x_1, \ldots, x_n] \to H^*(\Fl{n})$.
    We claim $[\iMX_{y \times 1^m}]$ is constant for fixed $y$ and varying $m$.
    Since $[\iMX_{y}]$ is a representative for $[\iX_y]$, the result then follows by Theorem~\ref{t:stable}.

    For $y \neq 1 \in S_n$, define $\maxdes(y) = \max \{i \in \NN : y(i) > y(i+1)\}$. 
    Replacing $[n]$ in the definition \eqref{imxy-eq} by $[\maxdes(y)]$ yields exactly the same variety $\iMX_y$.
    Since  $\maxdes(y \times 1^m)$ is independent of $m$, as is $\rank(y \times 1^m)_{[i][j]}$ for $i, j \in [\maxdes(y)]$, it follows that the ideals of $\iMX_{y \times 1^m}$ for fixed $y$ and varying $m$ have a common generating set.
    It is clear from \S \ref{subsec:eq-cohom} that this means that the polynomial $[\iMX_{y \times 1^m}]$ is independent of $m$.

    The proof for the skew-symmetric case is the same, replacing $\O_n$ by $\Sp_n$ and the map $\sigma : g \mapsto gg^T$ by $g \mapsto g\Omega g^T$, where $\Omega \in \GL_n$ is the 
    nondegenerate skew-symmetric form preserved by $\Sp_n$.
\end{proof}

\begin{cor} \label{cor:positivity} The polynomial $2^{\kappa(y)}\iS_y$ (respectively, $\iSfpf_z$) is a positive integer linear combination of products of terms $x_i+x_j$ for $1 \leq i \leq j \leq n$ (respectively, $1 \leq i < j \leq n$). 
\end{cor}

\begin{proof} The weights of $\Torus$ acting on $\SM_n$ are $x_i+x_j$ for $1 \leq i \leq j \leq n$, while the weights of $\SSM_n$ are the same with the added restriction $i < j$. The expression \eqref{eq:equivariant-class} makes clear that the classes $[\iMX_y]$ and $[\iMXfpf_z]$ are positive integer linear combinations of products of these weights. \end{proof}

    \begin{rmk}
        Let $\STorus$ be a maximal torus in $\O_n$. Let $\Torus \times \STorus$ act on $\GL_n$ by $(t,s) \cdot g = tgs^{-1}$ and on $\SM_n$ by $(t,s) \cdot A = tAt$. The map $\sigma : \GL_n \to \SM_n$, $g \mapsto gg^T$ considered above is then $\Torus \times \STorus$-equivariant. Since the second factor of $\Torus \times \STorus$ acts trivially on $\SM_n$, the polynomial $2^{\kappa(y)}\iS_y$ still represents the class $[\iMX_y] \in H_{\Torus \times \STorus}(\SM_n)$. It follows as in the proof of Theorem~\ref{thm:inv-schubert-classes} that $2^{\kappa(y)}\iS_y$ also represents the class $[\iX_y]_\STorus \in H_{\STorus}(\Fl{n})$. The latter fact was proven by Wyser and Yong \cite{wyser-yong-orthogonal-symplectic}, but our approach gives an explanation for the surprising existence of a representative for $[\iX_y]_\STorus$ not involving the $\STorus$-weights. Similar remarks apply in the skew-symmetric case.
    \end{rmk}

\section{Characterizing pipe dreams} \label{sec:reduced-words}

The rest of this article is focused on the combinatorial properties of involution pipe dreams and 
their role in the formulas in Theorem~\ref{thm:inv-pipe-dream-formula}
that manifest Corollary~\ref{cor:positivity}.
In the introduction, we defined (fpf-)involution pipe dreams
via simple symmetry conditions.
In this section, we give an equivalent characterization
in terms of ``compatible sequences'' related to involution words.

\subsection{Reading words}

For $p \in \ZZ$, the \emph{$p\textsuperscript{th}$ antidiagonal} 
(respectively, \emph{$p\textsuperscript{th}$ diagonal}) 
in $\PP \times \PP$ is the set 
\[\{(i,j) \in \PP\times \PP : i+j-1 = p\}
\qquad(\text{respectively, }\{(i,j)  \in \PP\times \PP: j-i = p\}).\]
Labeling the elements of $\{1,2,3\}\times\{1,2,3\}$ by 
their respective antidiagonal and diagonal  gives
\begin{equation*}
    \begin{array}{ccc}
        1 & 2 & 3\\
        2 & 3 & 4\\
        3 & 4 & 5\\
    \end{array} \qquad \text{and} \qquad
    \begin{array}{rrr}
        0 & -1 & -2\\
        1 & 0 & -1\\
        2 & 1 & 0
    \end{array}
\end{equation*} 
Let $\adiag: \PP\times \PP \to \PP$ be the map sending  $(i,j) \mapsto i+j-1$. 

\begin{defn}
The \emph{standard reading word} of $D\subseteq [n]\times [n]$ is the sequence 
\[ \word(D) = \adiag(\alpha_1) \adiag(\alpha_2) \cdots \adiag(\alpha_{|D|})\]
where $\alpha_1,\alpha_2,\dots,\alpha_{|D|}$ are the positions of $D$ 
read row-by-row from right to left, starting with the top row.
\end{defn}

If one also records the row indices of the positions $\alpha_i$ as a second word,
then the resulting words uniquely determine $D$ and are the same data as a \emph{compatible sequence} for $\word(D)$
(see \cite[(1)]{BJK}).

\begin{ex}
The subset $D= \{ (1,3),(1,2),(2,3),(2,2),(3,2)\}$ has $\word(D) = 32434$.
\end{ex}

We introduce a more general class of reading words.
Suppose $\omega : [n] \times [n] \to [n^2]$ is a bijection.  
For a subset $D \subseteq [n] \times [n]$ with $\omega(D) = \{ i_1 < i_2< \dots < i_{m}\}$, let 
\[\word(D,\omega) =\adiag(\omega^{-1}(i_1))\adiag(\omega^{-1}(i_2))\cdots \adiag(\omega^{-1}(i_m)).\]
The standard reading word of $D\subseteq [n]\times [n]$ 
corresponds to 
the bijection $\omega : (i,j) \mapsto ni - j + 1$.

\begin{ex} If $n=2$ and $\omega$ is such that
$
\left[ \begin{array}{rr}   \omega(1,1) &   \omega(1,2) \\    \omega(2,1) &   \omega(2,2) \end{array}\right] = \left[ 
 \begin{array}{cc}   3 &   1 \\    4 &   2 \end{array}
 \right]
 $
 then we would have $\word([n]\times [n],\omega) = 2312$, while
 if 
 $D 
 = \{(1,1),(2,2)\}$
 then $\word(D,\omega) = 31$.
 \end{ex}

For us, a \emph{linear extension} of a finite poset $(P,\preceq)$ with size $m=|P|$
is a bijection $\omega : P \to [m]$ such that $\omega(s) < \omega(t)$ whenever $s \prec t$ in $P$.

\begin{defn} A \emph{reading order} on $[n]\times[n]$ is a linear extension  of 
the partial order $\NEleq$ on $[n] \times [n]$ that has $(i,j) \NEleq (i',j')$ if and only if both $i \leq i'$ and $j \geq j'$.
If $\omega$ is a reading order, then  we refer to $\word(D,\omega)$
as a \emph{reading word} of $D\subseteq [n]\times [n]$. \end{defn}

The \emph{Coxeter commutation class} of a finite sequence of integers is its equivalence class
under the relation that lets adjacent letters commute if their positive difference is at least two.
For example, $\{1324, 3124,1342,3142, 3412\}$ is a single Coxeter commutation class.
Fix a set $D \subseteq [n] \times [n]$.

\begin{lem}\label{commutation-lem}
 All reading words of $D$
are in the same Coxeter commutation class.
 \end{lem}

This result can be derived using Viennot's theory of heaps of pieces; see \cite[Lem. 3.3]{Viennot}.

\begin{proof}
Let $s_p \in S_{n^2}$ be the simple transposition interchanging $p$ and $p+1$,
and choose a reading order $\omega$ on $[n]\times [n]$.
    The sequence $\word(D, s_p \omega)$ is equal to $\word(D, \omega)$ when $\{p,p+1\}\not\subset \omega(D)$,
    and otherwise is obtained by interchanging two adjacent letters in $\word(D,\omega)$. In the latter case, if $\omega^{-1}(p) = (i,j)$ and $\omega^{-1}(p+1) = (i',j')$ are not in adjacent antidiagonals, 
    then $\word(D,\omega)$ and $\word(D, s_p \omega)$ are in the same Coxeter commutation class.
    
    Now suppose $\upsilon$ is a second reading order on $[n]\times [n]$. 
    We claim that one can pass from $\omega$ to $\upsilon$ by composing $\omega$ with a sequence of simple transpositions obeying the condition just described. To check this, we induct on the number of inversions in the permutation $\upsilon  \omega^{-1} \in S_{n^2}$. If $\upsilon \omega^{-1}$ is not the identity, then there exists $p$ with $\upsilon(\omega^{-1}(p)) > \upsilon(\omega^{-1}(p+1))$. 
Since $\upsilon$ and $\omega$ are both linear extensions of $\NEleq$,
we can have neither $\omega^{-1}(p) \NEleq \omega^{-1}(p+1)$ nor $\omega^{-1}(p+1) \NEleq \omega^{-1}(p)$, so the cells $\omega^{-1}(p)$ and $\omega^{-1}(p+1)$
are not in adjacent antidiagonals.    
    Therefore $\word(D,\omega)$ and $\word(D, s_p  \omega)$ are in the same Coxeter commutation class, which by induction also includes $\word(D,\upsilon)$.
\end{proof}

Each diagonal is an antichain for $\NEleq$, so if $\omega$ first lists the elements on diagonal $-(n-1)$ in any order, then lists the elements on diagonal $-(n-2)$, and so on, then $\omega$ is a reading order.

\begin{defn}\label{unimodal-def}
The \emph{unimodal-diagonal reading order} on $[n]\times [n]$
is the reading order that lists the elements of the $p\textsuperscript{th}$ diagonal from bottom to top if $p < 0$, and from top to bottom if $p \geq 0$.     The \emph{unimodal-diagonal reading word} of $D\subseteq [n]\times[n]$, denoted $\udiag(D)$, is
    the associated reading word. 
    \end{defn}
    
    The unimodal-diagonal reading order 
    on $\{1,2,3,4\} \times \{1,2,3,4\}$ has values
\begin{equation*}
    \begin{array}{ccccc}
        7 & 6 & 3 & 1 \\
       11 & 8 & 5 & 2 \\
       14 & 12& 9 & 4 \\
       16& 15& 13& 10 \\
      \end{array}
    \end{equation*}
and if
$D = \{1,2,3,4\} \times \{1,2,3,4\}$
then $\udiag(D) = 4536421357246354$.

\subsection{Pipe dreams}

Recall the definitions of the sets of reduced words $\R(w)$, involution words $\iR(y)$, and fpf-involution words $\iRfpf(z)$
for $w \in S_n$, $y \in \I_n$, and $z \in \Ifpf_n$ from Section~\ref{schub-sect}.
For the standard reading word, the following theorem is 
well-known from \cite{bergeron-billey}. The main new results of this section are versions of this theorem
 for involution pipe dreams and fpf-involution pipe dreams.

 \begin{thm}\label{pipe-thm}
A subset 
$D \subseteq [n]\times [n]$
is a reduced pipe dream for $w \in S_n$
if and only if some (equivalently, every) reading word of $D$
is a reduced word for $w$.
\end{thm}

\begin{proof}
Fix $D \subseteq [n]\times [n]$ and $w \in S_n$.
The set $\R(w)$ is a union of Coxeter commutation classes, so
$\word(D)\in\R(w)$ if and only every
reading word of $D$ belongs to $\R(w)$ by Lemma~\ref{commutation-lem}.
Saying that $D$ is a reduced pipe dream for $w$
if and only if $\word(D) \in \R(w)$ is
Bergeron and Billey's original definition of an \emph{rc-graph}
in \cite[\S3]{bergeron-billey}, and it is clear from the basic properties of permutation wiring diagrams
that this is equivalent to the definition of a reduced pipe dream in the introduction.
\end{proof}

\begin{cor}[{\cite[Lem. 3.2]{bergeron-billey}}]
\label{transpose-cor}
If $D$ is a reduced pipe dream for $w \in S_n$ then $D^T$ is a reduced pipe dream for $w^{-1}$.
\end{cor}

Recall that the set  $\IP(z)$ of
 \emph{involution pipe dreams}
for $z \in \I_n$ consists of all intersections $D\cap \ltriang_n$
where $D$ is a reduced pipe dream for $z$ that is
almost-symmetric
and $\ltriang_n = \{ (j,i) \in [n]\times [n] : i \leq j\}$.

\begin{thm}\label{thm:almost-symmetric}
Suppose $z \in \I_n$ and $D \subseteq [n]\times [n]$. The following are equivalent:
\ben 
\item[(a)] Some reading word of $D$ is an involution word for $z$.
\item[(b)] Every reading word of $D$ is an involution word for $z$.
\item[(c)] The set $D$ is a reduced pipe dream for some atom of $z$.
\een
Moreover, if $D \subseteq \ltriang_n$ then $D\in \IP(z)$ if and only if these equivalent conditions hold.
\end{thm}

\begin{rmk}
Although this theorem implies that $\IP(z) \subseteq \bigsqcup_{w \in \A(z)} \RP(w)$,
 it is possible for an atom  $w \in \A(z)$  to have no reduced pipe dreams contained in $\ltriang_n$,
in which case $\IP(z)$ and $\RP(w)$ are disjoint. See Example~\ref{3.10ex} for an illustration of this.
\end{rmk}

\begin{proof}
Recall that $\iR(z)$ is the disjoint union of the sets $\R(w)$, running over all atoms $w \in \A(z)$.
The equivalences (a) $\Leftrightarrow$ (b) $\Leftrightarrow$ (c) are clear from Lemma~\ref{commutation-lem} and Theorem~\ref{pipe-thm}. 
Assume $D \subseteq \ltriang_n$. 
To prove the final assertion,
it suffices to show that $D\in \IP(z)$
if and only if the unimodal-diagonal reading word of $D$ from Definition~\ref{unimodal-def}
is an involution word of $z$.

Suppose $|D| = m$ and $\udiag(D)  = a_1a_2\cdots a_m$.
We construct a sequence $w_0,w_1,w_2,\dots,w_m$ of involutions as follows:
start by setting $w_0 = 1$, and for each $i \in [m]$ define  $ w_i = s_{a_i} w_{i-1} s_{a_i}$  if we have $w_{i-1} s_{a_i} \neq s_{a_i} w_{i-1}$, or else set $w_i = w_{i-1} s_{a_i} = s_{a_i} w_{i-1}$. For example, if $m=5$ and $a_1a_2a_3a_4a_5  =13235$ then this sequence has
\[w_1 =s_1,
\quad
w_2= s_1s_3,
\quad
w_3 =  s_2s_1s_3s_2,
\quad
w_4 = s_3 s_2s_1s_3s_2 s_3,
\quand
w_5 =s_3 s_2s_1s_3s_2 s_3 s_5.
\]
Let $b_l \cdots b_2 b_1$ be the subword of $a_m \cdots a_2 a_1$ which contains $a_i$ if and
only if $w_i = s_{a_i} w_{i-1} s_{a_i}$.
In our example with $m=5$ and $a_1a_2a_3a_4a_5= 13235$, we have
$l=2$ and $ b_2 b_1 = a_4a_3= 32$.
Let
$(p_1,q_1)$, $(p_2,q_2)$, \dots, $(p_m,q_m)$
be the cells in $D$ listed in the unimodal-diagonal reading order 
and define $E = D\sqcup \{  (q_i,p_i) : w_i = s_{a_i} w_{i-1} s_{a_i}\}.$
If $\udiag(D) = 13235$ then we could have
\[D = \left\{\begin{smallmatrix} 
+ & \cdot & \cdot & \cdot \\ 
+ & +  & \cdot & \cdot \\
+ &  \cdot & \cdot & \cdot\\
\cdot & + & \cdot & \cdot
\end{smallmatrix}\right\}
\qquad\text{then}\qquad
E = \left\{\begin{smallmatrix} 
+ & + & + & \cdot \\ 
+ & + & \cdot & \cdot \\
+ & \cdot & \cdot & \cdot \\
\cdot & + & \cdot & \cdot
\end{smallmatrix}\right\}.
\]
By construction $\udiag(E) = b_l \cdots b_2 b_1 a_1a_2\cdots a_m$ is a reduced word for $z$.
It follows that $E$ is almost-symmetric since each $b_i$ has a corresponding $a_j$ and the associated cells are transposes of each other.

The exchange principle (see, e.g., \cite[Lem. 3.4]{hultman-twisted-involutions}) implies that
if $w \in \I_n$, $i \in [n-1]$, and $w(i)<w(i+1)$, then
either
 $s_i ws_i =w\neq ws_i =s_i w = s_i \circ w \circ s_i$ or $ s_iw s_i = s_i \circ w \circ s_i \neq w$.
From this, it is straightforward to show that $\udiag(D) \in \iR(z) $ if and only if 
$\udiag(E) \in \R(z)$; this also follows from the results in \cite[\S2]{HMP2}.
Given the previous paragraph, we conclude that $\udiag(D)\in \R(z)$ if and only if $D = E \cap \ltriang_n$ is an involution pipe dream for $z$.
\end{proof}

\begin{ex}\label{3.10ex}
Let $z = 1432 = \in \I_4$. Since
$
z = s_3 \circ s_2 \circ 1 \circ  s_2 \circ s_3 = s_2 \circ s_3 \circ 1 \circ s_3 \circ s_2,$
we have $ 23 \in \iR(z)$ and $ 32\in \iR(z)$.
These  are the standard reading words of the involution pipe dreams
$\{(2,1),(3,1)\}$ and $\{(2,1),(2,2)\}$, which may be drawn as 
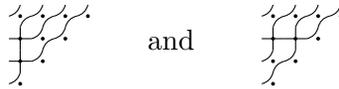
\[
\begin{tikzpicture}[scale=0.3,baseline=(b.base)]
\node (b) at (0,2) {}; 
\filldraw (0.5, 3.5) circle[radius=.6mm]; 
\draw 
(0,3.5) to[bend right] (0.5,4) 
(0.5,3) to[bend left] (1,3.5); 
\filldraw (1.5, 3.5) circle[radius=.6mm]; 
\draw 
(1,3.5) to[bend right] (1.5,4) 
(1.5,3) to[bend left] (2,3.5); 
\filldraw (2.5, 3.5) circle[radius=.6mm]; 
\draw 
(2,3.5) to[bend right] (2.5,4) 
(2.5,3) to[bend left] (3,3.5); 
\filldraw (3.5, 3.5) circle[radius=.6mm]; 
\draw 
(3,3.5) to[bend right] (3.5,4); 
\filldraw (0.5, 2.5) circle[radius=.6mm]; 
\draw 
(0,2.5) to (1,2.5) 
(0.5,2) to (0.5,3); 
\filldraw (1.5, 2.5) circle[radius=.6mm]; 
\draw 
(1,2.5) to[bend right] (1.5,3) 
(1.5,2) to[bend left] (2,2.5); 
\filldraw (2.5, 2.5) circle[radius=.6mm]; 
\draw 
(2,2.5) to[bend right] (2.5,3); 
\filldraw (0.5, 1.5) circle[radius=.6mm]; 
\draw 
(0,1.5) to (1,1.5) 
(0.5,1) to (0.5,2); 
\filldraw (1.5, 1.5) circle[radius=.6mm]; 
\draw 
(1,1.5) to[bend right] (1.5,2); 
\filldraw (0.5, 0.5) circle[radius=.6mm]; 
\draw 
(0,0.5) to[bend right] (0.5,1); 
\end{tikzpicture} 
\qquand
\begin{tikzpicture}[scale=0.3,baseline=(b.base)]
\node (b) at (0,2) {}; 
\filldraw (0.5, 3.5) circle[radius=.6mm]; 
\draw 
(0,3.5) to[bend right] (0.5,4) 
(0.5,3) to[bend left] (1,3.5); 
\filldraw (1.5, 3.5) circle[radius=.6mm]; 
\draw 
(1,3.5) to[bend right] (1.5,4) 
(1.5,3) to[bend left] (2,3.5); 
\filldraw (2.5, 3.5) circle[radius=.6mm]; 
\draw 
(2,3.5) to[bend right] (2.5,4) 
(2.5,3) to[bend left] (3,3.5); 
\filldraw (3.5, 3.5) circle[radius=.6mm]; 
\draw 
(3,3.5) to[bend right] (3.5,4); 
\filldraw (0.5, 2.5) circle[radius=.6mm]; 
\draw 
(0,2.5) to (1,2.5) 
(0.5,2) to (0.5,3); 
\filldraw (1.5, 2.5) circle[radius=.6mm]; 
\draw 
(1,2.5) to (2,2.5) 
(1.5,2) to (1.5,3); 
\filldraw (2.5, 2.5) circle[radius=.6mm]; 
\draw 
(2,2.5) to[bend right] (2.5,3); 
\filldraw (0.5, 1.5) circle[radius=.6mm]; 
\draw 
(0,1.5) to[bend right] (0.5,2) 
(0.5,1) to[bend left] (1,1.5); 
\filldraw (1.5, 1.5) circle[radius=.6mm]; 
\draw 
(1,1.5) to[bend right] (1.5,2); 
\filldraw (0.5, 0.5) circle[radius=.6mm]; 
\draw 
(0,0.5) to[bend right] (0.5,1); 
\end{tikzpicture}
\]
The only involution pipe dream for $y = 321 \in \I_3$  is  $\{(1,1),(2,1)\}$ which has standard reading word $12$.
Although $\iR(y) =\{12,21\}$, there is no involution pipe dream with standard reading word $21$.
\end{ex}

We turn to the fixed-point-free case.

\begin{lem}\label{diag-sym-lem}
Assume $n$ is even. Suppose $z \in \I_n$ is an involution with a symmetric reduced pipe dream $D=D^T$.
Then $z \in \Ifpf_n$ if and only if  
$\{(i,i) : i \in [n/2]\}\subseteq D$.
\end{lem}

\begin{proof}
In fact, a stronger statement holds: for symmetric $D$ and $i \in [n/2]$, the pipes in cell $(i,i)$ of the wiring diagram
of $D$
are labeled by fixed points of $z$ if and only if
$(i,i) \notin D$.
Let $a$ and $b$ be the labels for the pipes entering $(i,i)$ from the left and below, respectively.
Since $D$ is symmetric, if $(i,i) \in D$ then $z(a) = b$ (hence $z(b) = a$), and if $(i,i) \notin D$ then $z(a) = a$ and $z(b) = b$.
\end{proof}

Recall that the set $\FP(z)$ of
 \emph{fpf-involution pipe dreams}
for $z \in \Ifpf_n$ consists of all intersections $D\cap \ltriangneq_n$
where $D$ is a reduced pipe dream for $z$ 
that is symmetric
and $\ltriangneq_n = \{ (j,i) \in [n]\times [n] : i < j\}$.

\begin{thm}\label{thm:fpf-almost-symmetric}
Suppose $n$ is even, $z \in \Ifpf_n$, and $D \subseteq [n]\times [n]$. The following are equivalent:
\ben
\item[(a)] Some reading word of $D$ is an fpf-involution word for $z$.
\item[(b)] Every reading word of $D$ is an fpf-involution word for $z$.
\item[(c)] The set $D$ is a reduced pipe dream for some fpf-atom of $z$.
\een
Moreover, if $D \subseteq \ltriangneq_n$ then $D\in \FP(z)$ if and only if these equivalent conditions hold.

\end{thm}

\begin{proof}
Recall that $\iRfpf(z)$ is the disjoint union of the sets $\R(w)$, running over all fpf-atoms $w \in \Afpf(z)$.
Properties (a), (b), and (c) are again equivalent 
by Lemma~\ref{commutation-lem} and Theorem~\ref{pipe-thm}. 
Assume $D \subseteq \ltriangneq_n$. To prove the final assertion, it suffices to check that
$D$ is an fpf-involution pipe dream for $z$ if and only if $\udiag(D) \in \iRfpf(z)$.

To this end, first suppose $D = E \cap \ltriangneq_n$ where $E=E^T \in \RP(z)$.
Then $E$ is also almost-symmetric, 
so
Theorem~\ref{thm:almost-symmetric} implies that $E\cap \ltriang_n\in \IP(z)$.
This combined with
Lemma~\ref{diag-sym-lem} implies that
$\udiag(E\cap \ltriang_n) = 1 3 5 \cdots (n-1)\udiag(D) \in \iR(z)$,
so $\udiag(D)  \in \iRfpf(z)$. 

Conversely, suppose every reading word of $D$ is an fpf-involution word for $z$,
so that $\udiag(D) \in \iRfpf(z)$.
The set
$D' = D \sqcup \{(i,i) : i \in [n-1]\}$ then has $\udiag(D') \in \iR(z)$,
so there exists an almost-symmetric
$D'' \in \RP(z)$ with $D'' \cap \ltriang_n = D'$ by Theorem~\ref{thm:almost-symmetric}.
By construction $D=D'' \cap \ltriangneq_n$, and
since $|D''| = \ell(z) =  2|D| + n/2$
it follows that $D''$ is actually symmetric.
Therefore $D\in \FP(z)$.
\end{proof}

  \begin{ex}
  Let $z =216543 \in \Ifpf_6$. Then $\ell(z) = 7$ and
  \[ z = s_3 \cdot s_4 \cdot (s _1 \cdot s_3 \cdot s_5) \cdot s_4\cdot s_3 =s_5 \cdot s_4 \cdot (s_1 \cdot s_3 \cdot s_5) \cdot s_4\cdot s_5,
  \]
  so $3413543$ and $5413545$ are reduced words for $z$. These words are the unimodal-diagonal reading words 
  of the symmetric reduced pipe dreams
  
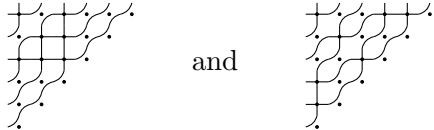
\[
\begin{tikzpicture}[scale=0.3,baseline=(b.base)]
\node (b) at (0,3) {}; 
\filldraw (0.5, 5.5) circle[radius=.6mm]; 
\draw 
(0,5.5) to (1,5.5) 
(0.5,5) to (0.5,6); 
\filldraw (1.5, 5.5) circle[radius=.6mm]; 
\draw 
(1,5.5) to[bend right] (1.5,6) 
(1.5,5) to[bend left] (2,5.5); 
\filldraw (2.5, 5.5) circle[radius=.6mm]; 
\draw 
(2,5.5) to (3,5.5) 
(2.5,5) to (2.5,6); 
\filldraw (3.5, 5.5) circle[radius=.6mm]; 
\draw 
(3,5.5) to[bend right] (3.5,6) 
(3.5,5) to[bend left] (4,5.5); 
\filldraw (4.5, 5.5) circle[radius=.6mm]; 
\draw 
(4,5.5) to[bend right] (4.5,6) 
(4.5,5) to[bend left] (5,5.5); 
\filldraw (5.5, 5.5) circle[radius=.6mm]; 
\draw 
(5,5.5) to[bend right] (5.5,6); 
\filldraw (0.5, 4.5) circle[radius=.6mm]; 
\draw 
(0,4.5) to[bend right] (0.5,5) 
(0.5,4) to[bend left] (1,4.5); 
\filldraw (1.5, 4.5) circle[radius=.6mm]; 
\draw 
(1,4.5) to (2,4.5) 
(1.5,4) to (1.5,5); 
\filldraw (2.5, 4.5) circle[radius=.6mm]; 
\draw 
(2,4.5) to (3,4.5) 
(2.5,4) to (2.5,5); 
\filldraw (3.5, 4.5) circle[radius=.6mm]; 
\draw 
(3,4.5) to[bend right] (3.5,5) 
(3.5,4) to[bend left] (4,4.5); 
\filldraw (4.5, 4.5) circle[radius=.6mm]; 
\draw 
(4,4.5) to[bend right] (4.5,5); 
\filldraw (0.5, 3.5) circle[radius=.6mm]; 
\draw 
(0,3.5) to (1,3.5) 
(0.5,3) to (0.5,4); 
\filldraw (1.5, 3.5) circle[radius=.6mm]; 
\draw 
(1,3.5) to (2,3.5) 
(1.5,3) to (1.5,4); 
\filldraw (2.5, 3.5) circle[radius=.6mm]; 
\draw 
(2,3.5) to (3,3.5) 
(2.5,3) to (2.5,4); 
\filldraw (3.5, 3.5) circle[radius=.6mm]; 
\draw 
(3,3.5) to[bend right] (3.5,4); 
\filldraw (0.5, 2.5) circle[radius=.6mm]; 
\draw 
(0,2.5) to[bend right] (0.5,3) 
(0.5,2) to[bend left] (1,2.5); 
\filldraw (1.5, 2.5) circle[radius=.6mm]; 
\draw 
(1,2.5) to[bend right] (1.5,3) 
(1.5,2) to[bend left] (2,2.5); 
\filldraw (2.5, 2.5) circle[radius=.6mm]; 
\draw 
(2,2.5) to[bend right] (2.5,3); 
\filldraw (0.5, 1.5) circle[radius=.6mm]; 
\draw 
(0,1.5) to[bend right] (0.5,2) 
(0.5,1) to[bend left] (1,1.5); 
\filldraw (1.5, 1.5) circle[radius=.6mm]; 
\draw 
(1,1.5) to[bend right] (1.5,2); 
\filldraw (0.5, 0.5) circle[radius=.6mm]; 
\draw 
(0,0.5) to[bend right] (0.5,1); 
\end{tikzpicture}
\qquand
\begin{tikzpicture}[scale=0.3,baseline=(b.base)]
\node (b) at (0,3) {}; 
\filldraw (0.5, 5.5) circle[radius=.6mm]; 
\draw 
(0,5.5) to (1,5.5) 
(0.5,5) to (0.5,6); 
\filldraw (1.5, 5.5) circle[radius=.6mm]; 
\draw 
(1,5.5) to[bend right] (1.5,6) 
(1.5,5) to[bend left] (2,5.5); 
\filldraw (2.5, 5.5) circle[radius=.6mm]; 
\draw 
(2,5.5) to[bend right] (2.5,6) 
(2.5,5) to[bend left] (3,5.5); 
\filldraw (3.5, 5.5) circle[radius=.6mm]; 
\draw 
(3,5.5) to (4,5.5) 
(3.5,5) to (3.5,6); 
\filldraw (4.5, 5.5) circle[radius=.6mm]; 
\draw 
(4,5.5) to (5,5.5) 
(4.5,5) to (4.5,6); 
\filldraw (5.5, 5.5) circle[radius=.6mm]; 
\draw 
(5,5.5) to[bend right] (5.5,6); 
\filldraw (0.5, 4.5) circle[radius=.6mm]; 
\draw 
(0,4.5) to[bend right] (0.5,5) 
(0.5,4) to[bend left] (1,4.5); 
\filldraw (1.5, 4.5) circle[radius=.6mm]; 
\draw 
(1,4.5) to (2,4.5) 
(1.5,4) to (1.5,5); 
\filldraw (2.5, 4.5) circle[radius=.6mm]; 
\draw 
(2,4.5) to[bend right] (2.5,5) 
(2.5,4) to[bend left] (3,4.5); 
\filldraw (3.5, 4.5) circle[radius=.6mm]; 
\draw 
(3,4.5) to[bend right] (3.5,5) 
(3.5,4) to[bend left] (4,4.5); 
\filldraw (4.5, 4.5) circle[radius=.6mm]; 
\draw 
(4,4.5) to[bend right] (4.5,5); 
\filldraw (0.5, 3.5) circle[radius=.6mm]; 
\draw 
(0,3.5) to[bend right] (0.5,4) 
(0.5,3) to[bend left] (1,3.5); 
\filldraw (1.5, 3.5) circle[radius=.6mm]; 
\draw 
(1,3.5) to[bend right] (1.5,4) 
(1.5,3) to[bend left] (2,3.5); 
\filldraw (2.5, 3.5) circle[radius=.6mm]; 
\draw 
(2,3.5) to (3,3.5) 
(2.5,3) to (2.5,4); 
\filldraw (3.5, 3.5) circle[radius=.6mm]; 
\draw 
(3,3.5) to[bend right] (3.5,4); 
\filldraw (0.5, 2.5) circle[radius=.6mm]; 
\draw 
(0,2.5) to (1,2.5) 
(0.5,2) to (0.5,3); 
\filldraw (1.5, 2.5) circle[radius=.6mm]; 
\draw 
(1,2.5) to[bend right] (1.5,3) 
(1.5,2) to[bend left] (2,2.5); 
\filldraw (2.5, 2.5) circle[radius=.6mm]; 
\draw 
(2,2.5) to[bend right] (2.5,3); 
\filldraw (0.5, 1.5) circle[radius=.6mm]; 
\draw 
(0,1.5) to (1,1.5) 
(0.5,1) to (0.5,2); 
\filldraw (1.5, 1.5) circle[radius=.6mm]; 
\draw 
(1,1.5) to[bend right] (1.5,2); 
\filldraw (0.5, 0.5) circle[radius=.6mm]; 
\draw 
(0,0.5) to[bend right] (0.5,1); 
\end{tikzpicture}
\]
so  
$\{(3,1),(3,2)\}$ and $\{(4,1),(5,1)\}$ are fpf-involution pipe dreams for $z$,
and their standard reading words $43$ and $45$ are fpf-involution words for $z$.
  \end{ex}

\section{Pipe dreams and Schubert polynomials}\label{sec:inv-polynom}

In this section, we derive the pipe dream formulas for involution Schubert polynomials
given in Theorem~\ref{thm:inv-pipe-dream-formula}.
Our arguments are inspired by a new proof due to Knutson \cite{Knutson} of the
classical pipe dream formula \eqref{eq:double-schubert-def}.
Knutson's approach is inductive.
The key step in his argument is to show that the right side of \eqref{eq:double-schubert-def}
satisfies certain recurrences that also apply to double Schubert polynomials \cite[\S4]{kohnert1997using}.


Similar recurrences for $\iS_y$ and $\iSfpf_z$ appear in~\cite{HMP3}.
Adapting Knutson's strategy to our setting requires us to 
show that the right hand expressions in Theorem~\ref{thm:inv-pipe-dream-formula} satisfy
the same family of identities.
This is accomplished in Theorems~\ref{thm:inv-dominant-transition} and \ref{thm:fpf-dominant-transition}.
Proving these results involves a detailed analysis of the 
maximal (shifted) Ferrers diagram contained in a reduced pipe dream,
which we refer to as the \emph{(shifted) dominant component}.
We gradually develop the technical properties of these components over the course of this section.

\subsection{Dominant components of permutations}

The results in this subsection are all straightforward consequences of known results, with the possible exception of Lemma~\ref{outer-corner-lem}; see in particular~\cite[\S3]{Knutson}.
However, we are unaware of an explicit description of Definition~\ref{d:dominant-component} in the literature.
Since this definition is central to our construction, we give a self-contained treatment of its properties.

A \emph{lower set} in a poset $(P,<)$ is a subset $L \subset P$
such that if $x \in P$, $y \in L$, and $x < y$, then $x \in L$.
Let $\NWleq$ be the partial order on $\PP \times \PP$ with $(i,j) \NWleq (i',j')$ if $i \leq i'$ and $j \leq j'$, 
i.e., if $(i,j)$ is northwest of $(i',j')$ in matrix coordinates. 

\begin{defn}\label{d:dominant-component}
The \emph{dominant component} $\dom{D}$ of a set $D \subseteq \PP\times \PP$ is 
the maximal lower set in $(\PP\times\PP,\NWleq)$ contained in $D$.
\end{defn}

Equivalently,  the set $\dom{D}$ consists of all pairs $(i,j) \in D$ such that whenever $(i',j') \in \PP\times \PP$ and $(i',j') \NWleq (i,j)$ 
it holds that $(i',j') \in D$.
If $D$ is finite, 
 then its dominant component $\dom{D}$ is the Ferrers diagram 
$
\D_\lambda = \{(i,j) : 1 \leq i \leq \ell(\lambda),\ 1 \leq j \leq \lambda_i\}
$
of some partition $\lambda$. 
An \emph{outer corner} of $D$ 
is a pair $(i,j) \in (\PP \times \PP) \setminus D$ such that $\dom{D} \sqcup \{(i,j)\}$ is again a Ferrers diagram
of some partition.
For example, $(1,2)$ and $(2,1)$ are the outer corners of $D = \{(1,1),(1,3)\}$, since $\dom{D} = \{(1,1)\}$.

For distinct $i,j \in [n]$, let $t_{ij} \in S_n$  be the transposition interchanging  $i$ and $j$.

\begin{lem}\label{outer-corner-lem}
Suppose $w \in S_n$ and $(i,j)$ is an outer corner of some $D \in \RP(w)$.
Then $w(i) = j$ and $D\sqcup\{(i,j)\}$ is a reduced pipe dream (for a longer permutation).
\end{lem}

\begin{proof}
By hypothesis, $D$ contains every cell above $(i,j)$ in the $j$th column
and every cell to the left of $(i,j)$ in the $i$th row.
This means that in the wiring diagram associated to $D$, the pipe leaving the top of position $(i,j)$ must continue straight up and terminate in column $j$ on the top side of $D$, and after leaving the left of position $(i,j)$, the same pipe 
must continue straight left and terminate in row $i$ on the left side of $D$. 
Thus $w(i) = j$ as claimed.
   Suppose the other pipe at position $(i,j)$ starts at $p$ on the left and ends at $q=w(p)$ on the top. As this pipe leaves  $(i,j)$ rightwards and downwards, we have $p > i$ and $q>j$, and the pipe only intersects $[i] \times [j]$ at $(i,j)$, where it avoids the other pipe.
   Therefore, we have $D\sqcup\{(i,j)\} \in \RP(w')$ for $w' :=wt_{ip} = t_{jq}w\in S_n$,
   and it holds that $\ell(w) < \ell(w')$ as $i<p$ and $w(i)<w(p)$.
\end{proof}

\begin{ex}
Suppose $w=426135 \in S_6$. If  $(i,j) = (2,2)$ and
\begin{equation*}\label{wiring-ex}
     D =  \begin{tikzpicture}[scale=0.3,baseline=(b.base)]
         	  \node (b) at (0,3) {}; 
            \filldraw (0.5, 5.5) circle[radius=.6mm]; 
            \draw 
            (0,5.5) to (1,5.5) 
            (0.5,5) to (0.5,6); 
            \filldraw (1.5, 5.5) circle[radius=.6mm]; 
            \draw 
            (1,5.5) to (2,5.5) 
            (1.5,5) to (1.5,6); 
            \filldraw (2.5, 5.5) circle[radius=.6mm]; 
            \draw 
            (2,5.5) to (3,5.5) 
            (2.5,5) to (2.5,6); 
            \filldraw (3.5, 5.5) circle[radius=.6mm]; 
            \draw 
            (3,5.5) to[bend right] (3.5,6) 
            (3.5,5) to[bend left] (4,5.5); 
            \filldraw (4.5, 5.5) circle[radius=.6mm]; 
            \draw 
            (4,5.5) to (5,5.5) 
            (4.5,5) to (4.5,6); 
            \filldraw (5.5, 5.5) circle[radius=.6mm]; 
            \draw 
            (5,5.5) to[bend right] (5.5,6); 
            \filldraw (0.5, 4.5) circle[radius=.6mm]; 
            \draw 
            (0,4.5) to (1,4.5) 
            (0.5,4) to (0.5,5); 
            \filldraw (1.5, 4.5) circle[radius=.6mm]; 
            \draw 
            (1,4.5) to[bend right] (1.5,5) 
            (1.5,4) to[bend left] (2,4.5); 
            \filldraw (2.5, 4.5) circle[radius=.6mm]; 
            \draw 
            (2,4.5) to (3,4.5) 
            (2.5,4) to (2.5,5); 
            \filldraw (3.5, 4.5) circle[radius=.6mm]; 
            \draw 
            (3,4.5) to[bend right] (3.5,5) 
            (3.5,4) to[bend left] (4,4.5); 
            \filldraw (4.5, 4.5) circle[radius=.6mm]; 
            \draw 
            (4,4.5) to[bend right] (4.5,5); 
            \filldraw (0.5, 3.5) circle[radius=.6mm]; 
            \draw 
            (0,3.5) to (1,3.5) 
            (0.5,3) to (0.5,4); 
            \filldraw (1.5, 3.5) circle[radius=.6mm]; 
            \draw 
            (1,3.5) to[bend right] (1.5,4) 
            (1.5,3) to[bend left] (2,3.5); 
            \filldraw (2.5, 3.5) circle[radius=.6mm]; 
            \draw 
            (2,3.5) to[bend right] (2.5,4) 
            (2.5,3) to[bend left] (3,3.5); 
            \filldraw (3.5, 3.5) circle[radius=.6mm]; 
            \draw 
            (3,3.5) to[bend right] (3.5,4); 
            \filldraw (0.5, 2.5) circle[radius=.6mm]; 
            \draw 
            (0,2.5) to[bend right] (0.5,3) 
            (0.5,2) to[bend left] (1,2.5); 
            \filldraw (1.5, 2.5) circle[radius=.6mm]; 
            \draw 
            (1,2.5) to[bend right] (1.5,3) 
            (1.5,2) to[bend left] (2,2.5); 
            \filldraw (2.5, 2.5) circle[radius=.6mm]; 
            \draw 
            (2,2.5) to[bend right] (2.5,3); 
            \filldraw (0.5, 1.5) circle[radius=.6mm]; 
            \draw 
            (0,1.5) to[bend right] (0.5,2) 
            (0.5,1) to[bend left] (1,1.5); 
            \filldraw (1.5, 1.5) circle[radius=.6mm]; 
            \draw 
            (1,1.5) to[bend right] (1.5,2); 
            \filldraw (0.5, 0.5) circle[radius=.6mm]; 
            \draw 
            (0,0.5) to[bend right] (0.5,1); 
            \draw (-0.3,5.5) node {$\scriptstyle 1$};
            \draw (-0.3,4.5) node {$\scriptstyle 2$};
            \draw (-0.3,3.5) node {$\scriptstyle 3$};
            \draw (-0.3,2.5) node {$\scriptstyle 4$};
            \draw (-0.3,1.5) node {$\scriptstyle 5$};
            \draw (-0.3,0.5) node {$\scriptstyle 6$};

            \draw (0.5,6.3) node {$\scriptstyle 1$};
            \draw (1.5,6.3) node {$\scriptstyle 2$};
            \draw (2.5,6.3) node {$\scriptstyle 3$};
            \draw (3.5,6.3) node {$\scriptstyle 4$};
            \draw (4.5,6.3) node {$\scriptstyle 5$};
            \draw (5.5,6.3) node {$\scriptstyle 6$};
        \end{tikzpicture} \qquad \text{so that} \qquad D\sqcup\{(i,j)\} = 
        \begin{tikzpicture}[scale=0.3,baseline=(b.base)]
         	  \node (b) at (0,3) {}; 
            \filldraw (0.5, 5.5) circle[radius=.6mm]; 
            \draw 
            (0,5.5) to (1,5.5) 
            (0.5,5) to (0.5,6); 
            \filldraw (1.5, 5.5) circle[radius=.6mm]; 
            \draw 
            (1,5.5) to (2,5.5) 
            (1.5,5) to (1.5,6); 
            \filldraw (2.5, 5.5) circle[radius=.6mm]; 
            \draw 
            (2,5.5) to (3,5.5) 
            (2.5,5) to (2.5,6); 
            \filldraw (3.5, 5.5) circle[radius=.6mm]; 
            \draw 
            (3,5.5) to[bend right] (3.5,6) 
            (3.5,5) to[bend left] (4,5.5); 
            \filldraw (4.5, 5.5) circle[radius=.6mm]; 
            \draw 
            (4,5.5) to (5,5.5) 
            (4.5,5) to (4.5,6); 
            \filldraw (5.5, 5.5) circle[radius=.6mm]; 
            \draw 
            (5,5.5) to[bend right] (5.5,6); 
            \filldraw (0.5, 4.5) circle[radius=.6mm]; 
            \draw 
            (0,4.5) to (1,4.5) 
            (0.5,4) to (0.5,5); 
            \filldraw (1.5, 4.5) circle[radius=.6mm]; 
            \draw 
            (1,4.5) to (2,4.5) 
            (1.5,4) to (1.5,5); 
            \filldraw (2.5, 4.5) circle[radius=.6mm]; 
            \draw 
            (2,4.5) to (3,4.5) 
            (2.5,4) to (2.5,5); 
            \filldraw (3.5, 4.5) circle[radius=.6mm]; 
            \draw 
            (3,4.5) to[bend right] (3.5,5) 
            (3.5,4) to[bend left] (4,4.5); 
            \filldraw (4.5, 4.5) circle[radius=.6mm]; 
            \draw 
            (4,4.5) to[bend right] (4.5,5); 
            \filldraw (0.5, 3.5) circle[radius=.6mm]; 
            \draw 
            (0,3.5) to (1,3.5) 
            (0.5,3) to (0.5,4); 
            \filldraw (1.5, 3.5) circle[radius=.6mm]; 
            \draw 
            (1,3.5) to[bend right] (1.5,4) 
            (1.5,3) to[bend left] (2,3.5); 
            \filldraw (2.5, 3.5) circle[radius=.6mm]; 
            \draw 
            (2,3.5) to[bend right] (2.5,4) 
            (2.5,3) to[bend left] (3,3.5); 
            \filldraw (3.5, 3.5) circle[radius=.6mm]; 
            \draw 
            (3,3.5) to[bend right] (3.5,4); 
            \filldraw (0.5, 2.5) circle[radius=.6mm]; 
            \draw 
            (0,2.5) to[bend right] (0.5,3) 
            (0.5,2) to[bend left] (1,2.5); 
            \filldraw (1.5, 2.5) circle[radius=.6mm]; 
            \draw 
            (1,2.5) to[bend right] (1.5,3) 
            (1.5,2) to[bend left] (2,2.5); 
            \filldraw (2.5, 2.5) circle[radius=.6mm]; 
            \draw 
            (2,2.5) to[bend right] (2.5,3); 
            \filldraw (0.5, 1.5) circle[radius=.6mm]; 
            \draw 
            (0,1.5) to[bend right] (0.5,2) 
            (0.5,1) to[bend left] (1,1.5); 
            \filldraw (1.5, 1.5) circle[radius=.6mm]; 
            \draw 
            (1,1.5) to[bend right] (1.5,2); 
            \filldraw (0.5, 0.5) circle[radius=.6mm]; 
            \draw 
            (0,0.5) to[bend right] (0.5,1); 

            \draw (-0.3,5.5) node {$\scriptstyle 1$};
            \draw (-0.3,4.5) node {$\scriptstyle 2$};
            \draw (-0.3,3.5) node {$\scriptstyle 3$};
            \draw (-0.3,2.5) node {$\scriptstyle 4$};
            \draw (-0.3,1.5) node {$\scriptstyle 5$};
            \draw (-0.3,0.5) node {$\scriptstyle 6$};

            \draw (0.5,6.3) node {$\scriptstyle 1$};
            \draw (1.5,6.3) node {$\scriptstyle 2$};
            \draw (2.5,6.3) node {$\scriptstyle 3$};
            \draw (3.5,6.3) node {$\scriptstyle 4$};
            \draw (4.5,6.3) node {$\scriptstyle 5$};
            \draw (5.5,6.3) node {$\scriptstyle 6$};
        \end{tikzpicture}        
 \end{equation*}
then in the notation of the proof, we have $p = 3$, $q=6$, and $w' = 462135$.
\end{ex}

The \emph{Rothe diagram} of  $w \in S_n$ is 
$
D(w) = \{(i,j) \in [n] \times [n] : w(i) > j\text{ and } w^{-1}(j) > i\}.
$
It is often useful to observe that the set $D(w)$ is the complement in $[n]\times [n]$ 
of the union of the hooks 
$\{ (x,w(i)) : i < x \leq n\} \sqcup \{ (i,w(i))\} \sqcup \{ (i,y) : w(i) < y \leq n\}$
for $i \in [n]$.
It is not hard to show that one always has $|D(w)| = \ell(w)$. 

\begin{defn}
The \emph{dominant component}
of a permutation $w \in S_n$
is
$\dom{w} = \dom{D(w)}.$
We say that permutation $w \in S_n$ is \emph{dominant} if $\dom{w} \in \RP(w)$.
\end{defn}

It is more common to define $w$ to be dominant if $D(w)$ is the Ferrers diagram of a partition, or equivalently if $w$ is $132$-avoiding. The following lemma shows that our definition is equivalent.

\begin{lem}\label{dom1-lem}
A permutation $w \in S_n$ is dominant 
if and only if it holds that $\RP(w) = \{ \dom{w}\}$, in which case $\dom{w} = D(w)$.
\end{lem}

\begin{proof}
If $w \in S_n$ is dominant 
then $\RP(w) = \{ \dom{w}\}= \{ D(w) \}$
since all reduced pipe dreams for $w$ have size $\ell(w) = |D(w)|$
and contain $\dom{w} \subseteq D(w)$.
\end{proof}

 \begin{cor}\label{sym-dom-cor}
 Let $w \in S_n$. Then $\dom{w^{-1}} = \dom{w}^T$.
If $w $ is dominant, then $\dom{w} = \dom{w}^T$ if and only if $w=w^{-1}$.
\end{cor}

\begin{proof}
 The first claim holds since $D(w^{-1}) = D(w)^T$.
If $w$ is dominant and $\dom{w} = \dom{w}^T$, then $\dom{w} = D(w)$
by Lemma~\ref{dom1-lem}
 so $D(w) = D(w)^T = D(w^{-1})$ and therefore $w=w^{-1}$.
\end{proof}

We write $\mu \subseteq \lambda$ for partitions $\mu$ and $\lambda$
to indicate that $\D_\mu \subseteq \D_\lambda$.

\begin{prop}
If $\lambda$ is a partition with $\lambda \subseteq (n-1,\dots,3, 2,1)$ then there exists a unique dominant 
permutation $w \in S_n$
with $\dom{w} = \D_\lambda$.
\end{prop}

\begin{proof}
This holds by induction as
adding an outer corner to the reduced pipe dream of a dominant permutation yields 
a reduced pipe dream of a new dominant permutation. 
\end{proof}

Write $\leq$ for the Bruhat order on $S_n$.
Since $v \leq w$ if and only if 
some (equivalently, every) reduced word for $w$ has a subword that is a reduced word for $v$  \cite[\S5.10]{Humphreys},
Theorem~\ref{pipe-thm} implies:

\begin{lem}
If $v,w \in S_n$ then $v \leq w$ if and only if some (equivalently, every) reduced pipe dream for $w$ has a subset that is a reduced pipe dream for $v$. 
\end{lem}

\begin{cor}\label{dom-subset-cor}
Let $v,w \in S_n$ with $v$ dominant. 
Then $v\leq w$ if and only if $\dom{v} \subseteq D$ for some (equivalently, every) $D \in \RP(w)$.
\end{cor}

\begin{proof}
This holds since a dominant permutation has only one reduced pipe dream.
\end{proof}


For each $i \in [n]$ let $c_i(w) =
|\{ j : (i,j) \in D(w)\}|.$
The \emph{code} of $w \in S_n$ is the integer sequence 
$ c(w) = (c_1(w), \ldots, c_n(w))$.
The \emph{bottom pipe dream} of $w \in S_n$ is 
the set 
\be\Dbot(w)=\{(i,j) \in [n]\times [n] :  j \leq c_i(w)\}
\label{eq:dbot}
\ee obtained by left-justifying $D(w)$.
It is not obvious that $\Dbot(w) \in \RP(w)$,
but this holds by results in \cite{bergeron-billey}; see also Theorem~\ref{bb-thm} below.

\begin{ex}
\label{ex:pipedream}
If $w = 35142 \in S_5$, then $D(w)$ is the set of $+$'s below: 
\begin{equation*}
\arraycolsep=1.5pt
\def\arraystretch{0.8}
\begin{array}{ccccc} 
+ & + & 1 & \cdot & \cdot\\
+ & + & \cdot & + & 1\\
1 & \cdot & \cdot & \cdot & \cdot\\
\cdot & + & \cdot & 1 & \cdot\\
\cdot & 1 & \cdot & \cdot & \cdot
\end{array}
\qquad \text{so we have} \qquad
c(w) = (2,3,0,1,0)
\qquand
\Dbot(w) = 
\arraycolsep=1.5pt
\begin{array}{ccccc}
+ & + &  \cdot & \,\cdot\\
+ & + &  + & \,\cdot\\
\cdot & \cdot & \cdot & \,\cdot \\
+ & \cdot & \cdot & \,\cdot \\
\cdot & \cdot & \cdot & \,\cdot 
\end{array}
\end{equation*}
\end{ex}

\begin{prop}\label{prop:dominant-independence}
If $w \in S_n$ and $D \in \RP(w)$ then $\dom{D} = \dom{w}$.
\end{prop}

\begin{proof}
For each $D \in \RP(w)$ there exists a dominant permutation $v \in S_n$
with $\dom{v} = \dom{D}$ and $v \leq w$, in which case 
$\dom{D} \subseteq \dom{E}$ for all $E \in \RP(w)$ by Corollary~\ref{dom-subset-cor}.
This can only hold if $\dom{D} = \dom{E}$ for all $E \in \RP(w)$.

To finish the proof, it suffices to show that $\dom{w} = \dom{\Dbot(w)}$.
It is clear by definition that $\dom{w} \subseteq \dom{\Dbot(w)}$.
Conversely,
each outer corner of $\dom{w}$ has the form $(i,w(i))$ for some $i \in [n]$
but no such cell is in $\dom{\Dbot(w)}$,
so we cannot have $\dom{w} \subsetneq \dom{\Dbot(w)}$.
\end{proof}

In the next sections, we define an \emph{outer corner} of $w \in S_n$ to be an outer corner of $\dom{w}$.

\subsection{Involution pipe dream formulas}

Recall that $\I_n = \{ w \in S_n : w=w^{-1}\}$ and $\ltriang_n = \{ (j,i) \in [n]\times [n] : i\leq j\}$.

\begin{defn}
The  \emph{shifted dominant component} of $z \in \I_n$ is the set 
$ \shdom{z} = \dom{z} \cap \ltriang_n.$
\end{defn}

Fix $z \in \I_n$.
By Proposition~\ref{prop:dominant-independence}, we have 
$\shdom{z} = \dom{D} \cap \ltriang_n$ for all $D \in \RP(z)$.
Recall that the \emph{shifted Ferrers diagram} of a strict partition $\lambda = (\lambda_1>\lambda_2 > \dots > \lambda_k >0)$
is the set \[ \SD_\lambda = \{ (i, i+j-1) : 1 \leq i \leq k,\ 1\leq j \leq \lambda_i\},\]
which is formed from $\D_\lambda$
by moving the boxes in row $i$ to the right by $i-1$ columns.
Since $\dom{z}$ is a Ferrers diagram, 
the set $\shdom{z}$ is the transpose of the shifted Ferrers diagram of some strict partition.
A pair $(j,i) \in \PP\times \PP$ with $i\leq j$ is an outer corner of $z$  
if and only if
 the transpose of 
$\shdom{z} \cup \{(j,i)\}$ is a shifted Ferrers diagram,
in which case $z(j) = i$.

\begin{lem}\label{double-lem}
If $z \in \I_n$ then $\dom{z} = \shdom{z} \cup \shdom{z}^T$.
\end{lem}

\begin{proof}
This holds since $z=z^{-1}$ implies that $\dom{z} = \dom{z}^T$.
\end{proof}

\begin{cor}\label{maximal-lower-cor}
If $z \in \I_n$ then
$\shdom{z}$ is the union of all lower sets of $(\ltriang_n,\NWleq)$ that are contained in some (equivalently, every) $D \in \IP(z)$.
\end{cor}

\begin{proof}
This is clear from Proposition~\ref{prop:dominant-independence} and Lemma~\ref{double-lem}.
\end{proof}

The natural definition of ``involution'' dominance
turns out to be equivalent to the usual notion:

\begin{prop}\label{inv-dom-prop}
Let $z \in \I_n$. The following are equivalent:
\begin{itemize}
\item[(a)] The permutation $z$ is dominant.
\item[(b)] It holds that $\RP(z) = \{ \dom{z}\}$.
\item[(c)] It holds that $\IP(z) = \{\shdom{z}\}$.
\end{itemize}
\end{prop}

\begin{proof}
We have (a) $\Leftrightarrow$ (b) by Lemma~\ref{dom1-lem} and (b) $\Leftrightarrow$ (c) by Lemma~\ref{double-lem}.
\end{proof}

\begin{prop}
If $\lambda$ is a strict partition with $\lambda \subseteq (n-1,n-3,n-5,\dots)$ then there exists a unique dominant 
involution $z \in \I_n$
with $\shdom{z}^T =  \SD_\lambda$.
\end{prop}

\begin{proof}
We have $\D_\mu =  \SD_\lambda \cup ( \SD_\lambda)^T$ for some $\mu$.
Let $z \in S_n$ be dominant with $\dom{z} = \D_\mu$.
Then $z \in \I_n$ by Corollary~\ref{sym-dom-cor} and $\shdom{z}^T = (\D_\mu \cap \ltriang_n)^T =  \SD_\lambda$.
Uniqueness  holds by Lemma~\ref{double-lem}.
\end{proof}

If $y,z \in \I_n$,
then  $y \leq z$ in Bruhat order if and only if 
some (equivalently, every) involution word for $z$ contains a subword that is an
involution word for $y$ (see either \cite[Cor. 8.10]{richardson-springer} with \cite{richardson-springer2}, 
or \cite[Thm. 2.8]{hultman-twisted-involutions2}).
The following is an immediate corollary of this property and Theorem~\ref{thm:almost-symmetric}.

\begin{lem}\label{inv-bruhat-lem}
Let $y,z \in \I_n$. Then $y \leq z$ if and only if 
some (equivalently, every) involution pipe dream for $z$ has a subset that is an
 involution pipe dream for $y$.
\end{lem}

\begin{cor}\label{inv-dom-subset-cor}
Let $y,z \in \I_n$ with $y$ dominant. 
Then $y\leq z$ if and only if $\shdom{y} \subseteq D$ for some (equivalently, every) $D \in \IP(z)$.
\end{cor}

\begin{proof}
This is clear since if $y \in \I_n$ is dominant then $|\IP(y)|=1$.
\end{proof}

We need to mention the following technical property of
the Demazure product from \cite{KM03}.

\begin{lem}[{\cite[Lem. 3.4(1)]{KM03}}]
\label{0-hecke-lem}
If $b_1b_2\cdots b_q$ is a subword of $a_1a_2\dots a_p$
where each $a_i \in [n-1]$,
then
$
s_{b_1} \circ s_{b_2}\circ \dots \circ s_{b_q} \leq s_{a_1} \circ s_{a_2}\circ \dots \circ s_{a_p} \in S_n.
$
\end{lem}

\begin{cor}\label{0-hecke-cor-0}
If $v',w',v,w \in S_n$ and $v \leq w$ and $v'\leq w'$ then $v' \circ v \leq w' \circ w$.
\end{cor}

\begin{proof}
This is clear from Lemma~\ref{0-hecke-lem}
given the subword property of $\leq$.
\end{proof}

\begin{cor}\label{0-hecke-cor}
If $v,w \in S_n$ and $v \leq w$ then $v^{-1}\circ v \leq w^{-1} \circ w$.
\end{cor}

\begin{proof}
Apply Corollary~\ref{0-hecke-cor-0} with $v' = v^{-1}$ and $w'=w^{-1}$.
\end{proof}

Recall $t_{ij} \in S_n$ is a transposition for distinct $i,j \in [n]$.
It is well-known and not hard to check that if $w \in S_n$ then 
$\ell(w t_{ij})  = \ell(w) +1$ if and only if $w(i)<w(j)$
and no  $i<e<j$ has $w(i) < w(e) < w(j)$.

Given $y \in \I_n$ and $1\leq  i <j \leq n$, let
$\A_{ij}(y)= \{ w t_{ij} : w \in \A(y)\text{, }\ell(wt_{ij})=\ell(w)+1\}.$
Each covering relation in $(S_n,\leq)$ arises
as the image of right multiplication by some transposition $t_{ij}$.
The following theorem
characterizes certain operators $\tau_{ij}$
which play an analogous role for $(\I_n,\leq)$.

\begin{thm}[{See \cite[\S3]{HMP3}}]
\label{tau-thm}
For each pair of integers $1\leq i < j \leq n$, there
are unique maps $\tau_{ij} : \I_n \to \I_n$
with the following properties:
\ben
\item[(a)] If $y \in \I_n$ and
$\A_{ij}(y)\cap \A(z)\neq\varnothing $
for some $z \in \I_n$
then $\tau_{ij}(y)= z$.

\item[(b)] If $y \in \I_n$ and
$\A_{ij}(y)\cap \A(z)=\varnothing$
for all $z \in \I_n$
then $\tau_{ij}(y)= y$.
\een
Moreover, if $y \in \I_n$ and $y \neq \tau_{ij}(y)=z$, then $y(i) \neq z(i)$ and $y(j)\neq z(j)$.
\end{thm}

This result has an extension for affine symmetric groups;
see \cite{M2018,MZ2018}.

\begin{rmk}
The operators $\tau_{ij}$, which first appeared in \cite{incitti},
 can be given a more explicit definition; see \cite[Table 1]{HMP3}.
However, our present applications only require the properties 
in the theorem.
\end{rmk}

For $y \in \I_n$, let $\ellhat(y)$ denote the common value of $\ell(w)$ for any $w \in \A(y)$.
This  is also the size of any $D \in \IP(y)$.
By Lemma~\ref{inv-bruhat-lem},
if $y,z \in \I_n$ and $y<z$ then $\ellhat(y) < \ellhat(z)$.
Let
\begin{equation*}
    \iPhi(y,j) = \left\{
    z \in \I_{n+1} : z =\tau_{js}(y)\text{ and }\ellhat(z)=\ellhat(y)+1\text{ for some }s>j
    \right\}
\end{equation*}
for $y \in \I_n$ and $j \in [n]$.
Since $S_n \subset S_{n+1}$ and $\I_n \subset \I_{n+1}$,
this set is well-defined. 

\begin{thm}\label{thm:inv-dominant-transition}
    Let $(j,i) $ be an outer corner of $y \in \I_n$ with $i\leq j$. 
    \ben
    \item[(a)] The map
    $D \mapsto D \sqcup \{(j,i)\}$ is a bijection 
$
        \IP(y) \to \bigsqcup_{z \in \iPhi(y,j)} \IP(z)
$.
\item[(b)] If $i+j \leq n$ then $\iPhi(y,j) \subset \I_n$.
\een
\end{thm}

\begin{proof} 
We have $y(j) = i$ and $y(i) = j$ by Lemma~\ref{outer-corner-lem}.
Suppose $v \in \A(y)$ and $D \in \RP(v) \cap \IP(y)$.
By considering the pipes crossing at position $(j,i)$ in 
the wiring diagram of $D$, as in the proof of Lemma~\ref{outer-corner-lem},
it follows that 
$D\sqcup \{(j,i)\}$ is a reduced pipe dream for a permutation
 $w$ that belongs to $\A_{js}(y)$ for some $j<s \leq n$.
Set $z = w^{-1}\circ w \in \I_n$.
We wish to show that $w \in \A(z)$,
since if this holds then $D\sqcup \{(j,i)\} \in \IP(z)$ and 
Theorem~\ref{tau-thm} implies that $z \in \iPhi(y,j)$.

To this end,
let $\tilde y \in \I_n$ be the dominant involution whose unique involution pipe dream
is  $\shdom{y} \sqcup\{(j,i)\}$
and let $\tilde v \in \A(\tilde y)$ be the (unique) atom with 
 $\IP(\tilde y) \subseteq \RP(\tilde v) $.
Corollaries~\ref{dom-subset-cor} and \ref{inv-dom-subset-cor}
imply that  $\tilde y \not<y$, $\tilde v \leq w $, and $v<w$ 
since $\shdom{\tilde y} \not\subset \shdom{y}$
and $\shdom{\tilde y} \subseteq D\sqcup \{(j,i)\}$.
Hence, we have
$
\tilde y = {\tilde v}^{-1} \circ \tilde v \leq w^{-1} \circ w=z
$
and
$
y = v^{-1} \circ v \leq w^{-1} \circ w=z
$
by Corollary~\ref{0-hecke-cor}.
Putting these relations together gives  
$\tilde y \not < y \leq z$
and
$\tilde y \leq z$, so we must have $y < z$ and $\ell(w) = \ellhat(y)+1  \leq \ellhat(z)$,
and therefore $w \in \A(z)$.

Thus, the map in part (a) at least has the desired codomain and 
 is clearly injective.
To show that it is also surjective, suppose $E \in \IP(z)$ for some $z \in \iPhi(y,j)$.
Lemma~\ref{inv-bruhat-lem} implies some $(l,k) \in E$ 
has $E \setminus \{(l,k)\} \in \IP(y)$. 
Let $E' \in \RP(z)$ be the almost-symmetric reduced pipe dream with $E = E'\cap \ltriang_n$.
If $(j,i) \neq (l,k)$
then, since $\dom{y} = \shdom{y} \cup \shdom{y}^T \subset E'$,
it would follow by considering the wiring diagram of $E'$
that $z(j) = i = y(j)$, 
contradicting the last assertion in Theorem~\ref{tau-thm}.
Thus $(j,i) = (l,k)$ so the map in part (a) is  surjective.
Part (b) holds because an involution belongs to $\I_n$
if any of its involution pipe dreams is contained in $\{ (j,i)  : i\leq j\text{ and }i +j \leq n\}$.
\end{proof}

\begin{ex}\label{inv-dom-ex}
If $y = 35142 
 \in \I_5$
then
 \[
\IP(y) = 
\left\{
\begin{tikzpicture}[scale=0.3,baseline=(b.base)]
\node (b) at (0,2.5) {}; 
\filldraw (0.5, 4.5) circle[radius=.6mm]; 
\draw 
(0,4.5) to (1,4.5) 
(0.5,4) to (0.5,5); 
\filldraw (1.5, 4.5) circle[radius=.6mm]; 
\draw 
(1,4.5) to[bend right] (1.5,5) 
(1.5,4) to[bend left] (2,4.5); 
\filldraw (2.5, 4.5) circle[radius=.6mm]; 
\draw 
(2,4.5) to[bend right] (2.5,5) 
(2.5,4) to[bend left] (3,4.5); 
\filldraw (3.5, 4.5) circle[radius=.6mm]; 
\draw 
(3,4.5) to[bend right] (3.5,5) 
(3.5,4) to[bend left] (4,4.5); 
\filldraw (4.5, 4.5) circle[radius=.6mm]; 
\draw 
(4,4.5) to[bend right] (4.5,5); 
\filldraw (0.5, 3.5) circle[radius=.6mm]; 
\draw 
(0,3.5) to (1,3.5) 
(0.5,3) to (0.5,4); 
\filldraw (1.5, 3.5) circle[radius=.6mm]; 
\draw 
(1,3.5) to (2,3.5) 
(1.5,3) to (1.5,4); 
\filldraw (2.5, 3.5) circle[radius=.6mm]; 
\draw 
(2,3.5) to[bend right] (2.5,4) 
(2.5,3) to[bend left] (3,3.5); 
\filldraw (3.5, 3.5) circle[radius=.6mm]; 
\draw 
(3,3.5) to[bend right] (3.5,4); 
\filldraw (0.5, 2.5) circle[radius=.6mm]; 
\draw 
(0,2.5) to[bend right] (0.5,3) 
(0.5,2) to[bend left] (1,2.5); 
\filldraw (1.5, 2.5) circle[radius=.6mm]; 
\draw 
(1,2.5) to[bend right] (1.5,3) 
(1.5,2) to[bend left] (2,2.5); 
\filldraw (2.5, 2.5) circle[radius=.6mm]; 
\draw 
(2,2.5) to[bend right] (2.5,3); 
\filldraw (0.5, 1.5) circle[radius=.6mm]; 
\draw 
(0,1.5) to (1,1.5) 
(0.5,1) to (0.5,2); 
\filldraw (1.5, 1.5) circle[radius=.6mm]; 
\draw 
(1,1.5) to[bend right] (1.5,2); 
\filldraw (0.5, 0.5) circle[radius=.6mm]; 
\draw 
(0,0.5) to[bend right] (0.5,1); 
\end{tikzpicture}
,\quad
\begin{tikzpicture}[scale=0.3,baseline=(b.base)]
\node (b) at (0,2.5) {}; 
\filldraw (0.5, 4.5) circle[radius=.6mm]; 
\draw 
(0,4.5) to (1,4.5) 
(0.5,4) to (0.5,5); 
\filldraw (1.5, 4.5) circle[radius=.6mm]; 
\draw 
(1,4.5) to[bend right] (1.5,5) 
(1.5,4) to[bend left] (2,4.5); 
\filldraw (2.5, 4.5) circle[radius=.6mm]; 
\draw 
(2,4.5) to[bend right] (2.5,5) 
(2.5,4) to[bend left] (3,4.5); 
\filldraw (3.5, 4.5) circle[radius=.6mm]; 
\draw 
(3,4.5) to[bend right] (3.5,5) 
(3.5,4) to[bend left] (4,4.5); 
\filldraw (4.5, 4.5) circle[radius=.6mm]; 
\draw 
(4,4.5) to[bend right] (4.5,5); 
\filldraw (0.5, 3.5) circle[radius=.6mm]; 
\draw 
(0,3.5) to (1,3.5) 
(0.5,3) to (0.5,4); 
\filldraw (1.5, 3.5) circle[radius=.6mm]; 
\draw 
(1,3.5) to (2,3.5) 
(1.5,3) to (1.5,4); 
\filldraw (2.5, 3.5) circle[radius=.6mm]; 
\draw 
(2,3.5) to[bend right] (2.5,4) 
(2.5,3) to[bend left] (3,3.5); 
\filldraw (3.5, 3.5) circle[radius=.6mm]; 
\draw 
(3,3.5) to[bend right] (3.5,4); 
\filldraw (0.5, 2.5) circle[radius=.6mm]; 
\draw 
(0,2.5) to[bend right] (0.5,3) 
(0.5,2) to[bend left] (1,2.5); 
\filldraw (1.5, 2.5) circle[radius=.6mm]; 
\draw 
(1,2.5) to (2,2.5) 
(1.5,2) to (1.5,3); 
\filldraw (2.5, 2.5) circle[radius=.6mm]; 
\draw 
(2,2.5) to[bend right] (2.5,3); 
\filldraw (0.5, 1.5) circle[radius=.6mm]; 
\draw 
(0,1.5) to[bend right] (0.5,2) 
(0.5,1) to[bend left] (1,1.5); 
\filldraw (1.5, 1.5) circle[radius=.6mm]; 
\draw 
(1,1.5) to[bend right] (1.5,2); 
\filldraw (0.5, 0.5) circle[radius=.6mm]; 
\draw 
(0,0.5) to[bend right] (0.5,1); 
\end{tikzpicture}\
\right\}
\]
so the transpose of $\shdom{y}$ is the shifted Ferrers diagram of the partition $\lambda = (2,1)$,
and $(3,1)$ is an outer corner.
One can show that $\iPhi(y,3) = \{53241,45312\}$,
and as predicted by Theorem~\ref{thm:inv-dominant-transition} with $(j,i) = (3,1)$,
both elements of $\iPhi(y,3)$ are dominant with 
\[
\IP(53241) = 
\left\{
\begin{tikzpicture}[scale=0.3,baseline=(b.base)]
\node (b) at (0,2.5) {}; 
\filldraw (0.5, 4.5) circle[radius=.6mm]; 
\draw 
(0,4.5) to (1,4.5) 
(0.5,4) to (0.5,5); 
\filldraw (1.5, 4.5) circle[radius=.6mm]; 
\draw 
(1,4.5) to[bend right] (1.5,5) 
(1.5,4) to[bend left] (2,4.5); 
\filldraw (2.5, 4.5) circle[radius=.6mm]; 
\draw 
(2,4.5) to[bend right] (2.5,5) 
(2.5,4) to[bend left] (3,4.5); 
\filldraw (3.5, 4.5) circle[radius=.6mm]; 
\draw 
(3,4.5) to[bend right] (3.5,5) 
(3.5,4) to[bend left] (4,4.5); 
\filldraw (4.5, 4.5) circle[radius=.6mm]; 
\draw 
(4,4.5) to[bend right] (4.5,5); 
\filldraw (0.5, 3.5) circle[radius=.6mm]; 
\draw 
(0,3.5) to (1,3.5) 
(0.5,3) to (0.5,4); 
\filldraw (1.5, 3.5) circle[radius=.6mm]; 
\draw 
(1,3.5) to (2,3.5) 
(1.5,3) to (1.5,4); 
\filldraw (2.5, 3.5) circle[radius=.6mm]; 
\draw 
(2,3.5) to[bend right] (2.5,4) 
(2.5,3) to[bend left] (3,3.5); 
\filldraw (3.5, 3.5) circle[radius=.6mm]; 
\draw 
(3,3.5) to[bend right] (3.5,4); 
\filldraw (0.5, 2.5) circle[radius=.6mm]; 
\draw 
(0,2.5) to (1,2.5) 
(0.5,2) to (0.5,3); 
\filldraw (1.5, 2.5) circle[radius=.6mm]; 
\draw 
(1,2.5) to[bend right] (1.5,3) 
(1.5,2) to[bend left] (2,2.5); 
\filldraw (2.5, 2.5) circle[radius=.6mm]; 
\draw 
(2,2.5) to[bend right] (2.5,3); 
\filldraw (0.5, 1.5) circle[radius=.6mm]; 
\draw 
(0,1.5) to (1,1.5) 
(0.5,1) to (0.5,2); 
\filldraw (1.5, 1.5) circle[radius=.6mm]; 
\draw 
(1,1.5) to[bend right] (1.5,2); 
\filldraw (0.5, 0.5) circle[radius=.6mm]; 
\draw 
(0,0.5) to[bend right] (0.5,1); 
\end{tikzpicture}\
\right\}
\quand
\IP(45312) = 
\left\{
\begin{tikzpicture}[scale=0.3,baseline=(b.base)]
\node (b) at (0,2.5) {}; 
\filldraw (0.5, 4.5) circle[radius=.6mm]; 
\draw 
(0,4.5) to (1,4.5) 
(0.5,4) to (0.5,5); 
\filldraw (1.5, 4.5) circle[radius=.6mm]; 
\draw 
(1,4.5) to[bend right] (1.5,5) 
(1.5,4) to[bend left] (2,4.5); 
\filldraw (2.5, 4.5) circle[radius=.6mm]; 
\draw 
(2,4.5) to[bend right] (2.5,5) 
(2.5,4) to[bend left] (3,4.5); 
\filldraw (3.5, 4.5) circle[radius=.6mm]; 
\draw 
(3,4.5) to[bend right] (3.5,5) 
(3.5,4) to[bend left] (4,4.5); 
\filldraw (4.5, 4.5) circle[radius=.6mm]; 
\draw 
(4,4.5) to[bend right] (4.5,5); 
\filldraw (0.5, 3.5) circle[radius=.6mm]; 
\draw 
(0,3.5) to (1,3.5) 
(0.5,3) to (0.5,4); 
\filldraw (1.5, 3.5) circle[radius=.6mm]; 
\draw 
(1,3.5) to (2,3.5) 
(1.5,3) to (1.5,4); 
\filldraw (2.5, 3.5) circle[radius=.6mm]; 
\draw 
(2,3.5) to[bend right] (2.5,4) 
(2.5,3) to[bend left] (3,3.5); 
\filldraw (3.5, 3.5) circle[radius=.6mm]; 
\draw 
(3,3.5) to[bend right] (3.5,4); 
\filldraw (0.5, 2.5) circle[radius=.6mm]; 
\draw 
(0,2.5) to (1,2.5) 
(0.5,2) to (0.5,3); 
\filldraw (1.5, 2.5) circle[radius=.6mm]; 
\draw 
(1,2.5) to (2,2.5) 
(1.5,2) to (1.5,3); 
\filldraw (2.5, 2.5) circle[radius=.6mm]; 
\draw 
(2,2.5) to[bend right] (2.5,3); 
\filldraw (0.5, 1.5) circle[radius=.6mm]; 
\draw 
(0,1.5) to[bend right] (0.5,2) 
(0.5,1) to[bend left] (1,1.5); 
\filldraw (1.5, 1.5) circle[radius=.6mm]; 
\draw 
(1,1.5) to[bend right] (1.5,2); 
\filldraw (0.5, 0.5) circle[radius=.6mm]; 
\draw 
(0,0.5) to[bend right] (0.5,1); 
\end{tikzpicture}\
\right\}.
\]
 \end{ex}

We may finally prove the pipe dream formula
in Theorem~\ref{thm:inv-pipe-dream-formula} for the polynomials $\iS_y$.

\begin{thm}\label{inv-pipe-schubert-thm}
If $z \in \I_n$ then $\iS_z = \sum_{D \in \IP(z)} \prod_{(i,j) \in D} 2^{-\delta_{ij}}(x_i+x_j)$.
\end{thm}

\begin{proof}
We abbreviate by setting $x_{(i,j)} = 2^{-\delta_{ij}}(x_i+x_j)$, so that $x_{(i,j)} = x_i$ if $i=j$ and otherwise $x_{(i,j)} = x_i+x_j$.
It follows from \cite[Thm. 3.30]{HMP3} that 
if $(j,i) \in \ltriang_n$ is an outer corner of $z \in \I_n$ then
\be\label{itransition-eq} 2^{-\delta_{ij}}(x_i+x_j) \iS_z = \sum_{u \in \iPhi(z,j)} \iS_u.\ee
On the other hand,
results of Wyser and Yong \cite{wyser-yong-orthogonal-symplectic} (see \cite[Thm. 1.3]{HMP1}) show that
\be\label{inv-rev-eq}
\iS_{n\cdots 321} = \prod_{1\leq i \leq j \leq n-i} x_{(i,j)}.\ee
Let $\iBJS_z = \sum_{D \in \IP(z)} \prod_{(i,j) \in D} x_{(i,j)}$.
We show that $\iS_z = \iBJS_z$ by downward induction on $\ellhat(z)$. 
If $\ellhat(z) = \max\{\ellhat(y): y \in \I_n\}$ then
$z=n\cdots 321$
and 
the desired identity 
is equivalent to \eqref{inv-rev-eq} since 
$n\cdots 321$ is dominant.
Otherwise, the transpose of $\shdom{z}$ is a proper subset of $\shdom{n\cdots 321}^T=\SD_{(n-1,n-3,n-5,\dots)}$
by Corollary~\ref{inv-dom-subset-cor},
so
$z$ must have an outer corner 
$(j,i)$ with $i\leq j$ and $i+j \leq n$.
In this case we have
$x_{(i,j)} \iS_z = \sum_{u \in \iPhi(z,j)} \iS_u= \sum_{u \in \iPhi(z,j)} \iBJS_u = x_{(i,j)} \iBJS_z$
by \eqref{itransition-eq},
induction,
and Theorem~\ref{thm:inv-dominant-transition}.
Dividing by $x_{(i,j)}$ completes the proof.
\end{proof}

\begin{ex}
Continuing Example~\ref{inv-dom-ex},
we have 
\[
\iS_{53241} = x_1x_2(x_2 + x_1)(x_3 + x_1)(x_4 + x_1)
\quand
\iS_{45312} = x_1x_2(x_2 + x_1)(x_3 + x_1)(x_3 + x_2),
\]
so
$
\iS_{35142} = \tfrac{1}{x_3 + x_1}\left(\iS_{53241} + \iS_{45312}\right)
=x_1x_2(x_2 + x_1)(x_1 + x_2+x_3 + x_4) .
$
 \end{ex}

\subsection{Fixed-point-free involution pipe dream formulas}

In this section, we assume $n$ is even.
Recall that 
$\ltriangneq_n = \{(j,i) \in [n] \times [n] : i<j\}$.
\begin{defn}
The \emph{strictly shifted dominant component} of
 $z \in \Ifpf_n$ is
the set
$ 
\shdomneq{z} = \dom{z} \cap \ltriangneq_n,
$
which is also equal to
$ \dom{D} \cap \ltriangneq_n$ for all $D \in \RP(z)$
by Proposition~\ref{prop:dominant-independence}.
\end{defn}

\begin{cor}\label{fpf-maximal-lower-cor}
If $z \in \Ifpf_n$ then
$\shdomneq{z}$ is the maximal lower set of the poset $(\ltriangneq_n,\NWleq)$ contained in some (equivalently, every) $D \in \FP(z)$.
\end{cor}

\begin{proof}
This is clear from Proposition~\ref{prop:dominant-independence}.
\end{proof}

For subsets $D \subseteq \ZZ\times \ZZ$, define
$D^\uparrow = \{ (i-1,j): (i,j) \in D\}$
and
$D^{\uparrow T}  = (D^{\uparrow })^T.$
For example, if we have $z =465132 \in \Ifpf_6$ then
the Rothe diagram is 
\[
D(z)= \left\{\begin{smallmatrix} 
+ & + & + & 1 & \cdot & \cdot \\ 
+ & +  & + & \cdot & + & 1 \\
+ &  + & + & \cdot & 1 & \cdot \\
1 & \cdot & \cdot & \cdot & \cdot  & \cdot \\
\cdot & + & 1 & \cdot & \cdot & \cdot  \\
\cdot & 1 & \cdot & \cdot & \cdot & \cdot 
\end{smallmatrix}\right\},
\]
with each $1$ indicating a position $(i,j) \in [n]\times [n]$ with $z(i) =j$ and each $+$ indicating a position in $ D(z)$.
The relevant dominant components are 
 \[
\dom{z} = 
\left\{\begin{smallmatrix} 
+ & + & +  \\ 
+ & +  & + \\
+ &  + & + 
\end{smallmatrix}\right\},
\quad
\shdomneq{z} = \left\{\begin{smallmatrix} 
\cdot & \cdot & \cdot  \\ 
+ & \cdot  & \cdot   \\
+ &  + & \cdot   \\
\end{smallmatrix}\right\},
\quand
\shdomneq{z}^{\uparrow T} = \left\{\begin{smallmatrix} 
+ & +  & \cdot\\ 
\cdot & + & \cdot  \\
\cdot & \cdot & \cdot
\end{smallmatrix}\right\}.
\]
As we see in this example, if $z \in \Ifpf_n$ is any fixed-point-free involution, then
$(\shdomneq{z})^{\uparrow T} $  is the shifted Ferrers diagram
of some strict partition.
Moreover, 
a pair $(j,i) \in \ltriangneq_n$ is an outer corner of $z$  
if and only if
$(\shdomneq{D} \sqcup \{(j,i)\})^{\uparrow T}
$ is again a shifted Ferrers diagram,
in which case $z(j) = i$ by Lemma~\ref{outer-corner-lem}.
The unique outer corner of $z = 465132$ in $\ltriangneq_6$ is $(4,1)$.

\begin{defn}
A fixed-point-free involution $z \in \Ifpf_n$ is \emph{fpf-dominant} if $\shdomneq{z} \in \FP(z)$.
\end{defn}

This condition does not imply that $z$ is dominant in the sense of being $132$-avoiding.
For example, $z=s_1s_3\cdots s_{n-1} = 2143\cdots n(n-1) \in \Ifpf_n$ is always fpf-dominant as $\shdomneq{z} = \varnothing$.

\begin{lem}\label{dom-fpf-lem}
If $z \in \Ifpf_n$ is fpf-dominant then $\FP(z) = \{ \shdomneq{z}\}.$ 
\end{lem}

\begin{proof}
This holds since
 $\shdomneq{z} \subseteq D$ for all
$D \in \FP(z)$ by Proposition~\ref{prop:dominant-independence}.
\end{proof}

\begin{prop}
If 
$\lambda$ is a strict partition with
$\lambda \subseteq (n-2,n-4,\dots,4,2)$ then there exists a unique 
fpf-dominant 
$z \in \Ifpf_n$
with $(\shdomneq{z})^{\uparrow T} := \{ (j, i-1) : (i,j) \in \shdomneq{z}\}  =  \SD_\lambda$.
\end{prop}

\begin{proof}
Uniqueness is clear from Lemma~\ref{dom-fpf-lem}.
If $\lambda \subseteq (n-2,n-4,\dots,2)$ is empty
then take $z=\idfpf_n$.
Otherwise, let $\mu\subset \lambda $ be a strict partition 
such that $ \SD_\lambda =  \SD_\mu\sqcup \{(i,j-1)\}$ where $i<j$.
By induction, there exists an fpf-dominant $y \in \Ifpf_n$
with $(\shdomneq{y})^{\uparrow T} =  \SD_\mu$.
Let $D \in \RP(y)$ be symmetric  with $D \cap \ltriangneq_n = \shdomneq{y}$.
Lemmas~\ref{diag-sym-lem} and \ref{outer-corner-lem}
imply that $ D \sqcup \{(j,i),(i,j)\}$ is a symmetric reduced pipe dream for some $z \in \Ifpf_n$,
which is the desired element.
\end{proof}

If $y,z \in \Ifpf_n$,
then $y \leq z$ in Bruhat order if and only if 
some (equivalently, every)
fpf-involution word for $z$ contains a subword that is an fpf-involution word for
$y$ \cite[Thm. 4.6]{HMP3}.
From this and Theorem~\ref{thm:fpf-almost-symmetric} we deduce the following:

\begin{lem}\label{fpf-bruhat-lem}
Suppose $y,z \in \Ifpf_n$.
Then $y \leq z$ if and only if some (equivalently, every) fpf-involution pipe dream for $z$ has a subset that
is an 
fpf-involution pipe dream for $y$.
\end{lem}

\begin{cor}\label{fpf-dom-subset-cor}
Let $y,z \in \Ifpf_n$ where $y$ is fpf-dominant. 
Then $y\leq z$ if and only if $\shdomneq{y} \subseteq D$ for some (equivalently, every) $D \in \FP(z)$.
\end{cor}

\begin{proof}
This is clear since if $y \in \Ifpf_n$ is fpf-dominant then $|\FP(y)|=1$.
\end{proof}

    For $y \in \Ifpf_n$ and $j \in [n]$,  
    define 
    $\fPhi(y,j)$ to be the set 
    of fixed-point-free involutions $z \in \Ifpf_{n+2}$
     with length $\ell(z) = \ell(y)+2$ that
    can be written as 
    $z = t_{js}\cdot ys_{n+1} \cdot t_{js}$ for an integer $s$ with $j < s \leq n+2$.
We have an analogue of Theorem~\ref{thm:inv-dominant-transition}:

\begin{thm}\label{thm:fpf-dominant-transition}
    Let $(j,i)$ be an outer corner of $y \in \Ifpf_n$ with $i<j$. 
    \ben
    \item[(a)] The map  $D \mapsto D \sqcup \{(j,i)\}$ is a bijection 
  $
        \FP(y) \to \bigsqcup_{z \in \fPhi(y,j)} \FP(z).
$ 
\item[(b)] If $i+j \leq n$ then $\fPhi(y,j)\subset \left\{ z s_{n+1} : z \in \Ifpf_{n}\right\}$.
\een
\end{thm}

\begin{proof} 
Our argument is similar to the proof of Theorem~\ref{thm:inv-dominant-transition}.
Choose $D \in \FP(y) = \FP(ys_{n+1})$. Suppose $D' \in \RP(ys_{n+1})$ is the symmetric reduced pipe dream with $D = D' \cap \ltriang_{n+2}$
and set $E = D\sqcup \{(j,i)\}$ and $E' = D' \sqcup \{(i,j),(j,i)\}$.
It follows from Lemmas~\ref{diag-sym-lem} and \ref{outer-corner-lem}
that $E'$ is a reduced pipe dream for some element $z \in \Ifpf_{n+2}$.
Since $E'$ is symmetric, one has
$E  = E' \cap \ltriangneq_n \in \FP(z)$.
Finally, by considering the pipes crossing at position $(j,i)$ in $E$ we deduce that $z \in \fPhi(y,j)$.

Thus $D \mapsto D \sqcup \{(j,i)\}$ is a well-defined
map $\FP(y) \to \bigsqcup_{z\in \fPhi(y,j)} \FP(z)$.
This map is clearly injective.
To show that it is also surjective, suppose $E \in \FP(z)$ for some $z \in \fPhi(y,j)$.
Lemma~\ref{fpf-bruhat-lem} implies that there exists a position $(l,k) \in E$ 
such that $E \setminus \{(l,k)\} \in \FP(y)$. If $(j,i) \neq (l,k)$ then it would follow
as in the proof Theorem~\ref{thm:inv-dominant-transition}
that $z(j) = y(j)=ys_{n+1}(j)=i$, which is impossible if $z = t_{js}\cdot ys_{n+1}\cdot t_{js}$ where  $i =y(j) < j < s \leq n + 2$.
Thus $(j,i) = (l,k)$ so the map in part (a) is also surjective.
Part (b) holds because $\left\{  zs_{n+1} : z \in \Ifpf_{n}\right\}$ contains all involutions in $\Ifpf_{n+2}$
with fpf-involution pipe dreams that are subsets of $\{ (j,i)  : i\leq j, i +j \leq n\}$.
\end{proof}

\begin{ex}\label{fpf-dom-ex}
If $y = 351624
 \in \Ifpf_6$
then
 \[
\FP(y) = 
\left\{
\begin{tikzpicture}[scale=0.3,baseline=(b.base)]
\node (b) at (0,3) {}; 
\filldraw (0.5, 5.5) circle[radius=.6mm]; 
\draw 
(0,5.5) to[bend right] (0.5,6) 
(0.5,5) to[bend left] (1,5.5); 
\filldraw (1.5, 5.5) circle[radius=.6mm]; 
\draw 
(1,5.5) to[bend right] (1.5,6) 
(1.5,5) to[bend left] (2,5.5); 
\filldraw (2.5, 5.5) circle[radius=.6mm]; 
\draw 
(2,5.5) to[bend right] (2.5,6) 
(2.5,5) to[bend left] (3,5.5); 
\filldraw (3.5, 5.5) circle[radius=.6mm]; 
\draw 
(3,5.5) to[bend right] (3.5,6) 
(3.5,5) to[bend left] (4,5.5); 
\filldraw (4.5, 5.5) circle[radius=.6mm]; 
\draw 
(4,5.5) to[bend right] (4.5,6) 
(4.5,5) to[bend left] (5,5.5); 
\filldraw (5.5, 5.5) circle[radius=.6mm]; 
\draw 
(5,5.5) to[bend right] (5.5,6); 
\filldraw (0.5, 4.5) circle[radius=.6mm]; 
\draw 
(0,4.5) to (1,4.5) 
(0.5,4) to (0.5,5); 
\filldraw (1.5, 4.5) circle[radius=.6mm]; 
\draw 
(1,4.5) to[bend right] (1.5,5) 
(1.5,4) to[bend left] (2,4.5); 
\filldraw (2.5, 4.5) circle[radius=.6mm]; 
\draw 
(2,4.5) to[bend right] (2.5,5) 
(2.5,4) to[bend left] (3,4.5); 
\filldraw (3.5, 4.5) circle[radius=.6mm]; 
\draw 
(3,4.5) to[bend right] (3.5,5) 
(3.5,4) to[bend left] (4,4.5); 
\filldraw (4.5, 4.5) circle[radius=.6mm]; 
\draw 
(4,4.5) to[bend right] (4.5,5); 
\filldraw (0.5, 3.5) circle[radius=.6mm]; 
\draw 
(0,3.5) to[bend right] (0.5,4) 
(0.5,3) to[bend left] (1,3.5); 
\filldraw (1.5, 3.5) circle[radius=.6mm]; 
\draw 
(1,3.5) to[bend right] (1.5,4) 
(1.5,3) to[bend left] (2,3.5); 
\filldraw (2.5, 3.5) circle[radius=.6mm]; 
\draw 
(2,3.5) to[bend right] (2.5,4) 
(2.5,3) to[bend left] (3,3.5); 
\filldraw (3.5, 3.5) circle[radius=.6mm]; 
\draw 
(3,3.5) to[bend right] (3.5,4); 
\filldraw (0.5, 2.5) circle[radius=.6mm]; 
\draw 
(0,2.5) to (1,2.5) 
(0.5,2) to (0.5,3); 
\filldraw (1.5, 2.5) circle[radius=.6mm]; 
\draw 
(1,2.5) to[bend right] (1.5,3) 
(1.5,2) to[bend left] (2,2.5); 
\filldraw (2.5, 2.5) circle[radius=.6mm]; 
\draw 
(2,2.5) to[bend right] (2.5,3); 
\filldraw (0.5, 1.5) circle[radius=.6mm]; 
\draw 
(0,1.5) to[bend right] (0.5,2) 
(0.5,1) to[bend left] (1,1.5); 
\filldraw (1.5, 1.5) circle[radius=.6mm]; 
\draw 
(1,1.5) to[bend right] (1.5,2); 
\filldraw (0.5, 0.5) circle[radius=.6mm]; 
\draw 
(0,0.5) to[bend right] (0.5,1); 
\end{tikzpicture}
,\quad
\begin{tikzpicture}[scale=0.3,baseline=(b.base)]
\node (b) at (0,3) {}; 
\filldraw (0.5, 5.5) circle[radius=.6mm]; 
\draw 
(0,5.5) to[bend right] (0.5,6) 
(0.5,5) to[bend left] (1,5.5); 
\filldraw (1.5, 5.5) circle[radius=.6mm]; 
\draw 
(1,5.5) to[bend right] (1.5,6) 
(1.5,5) to[bend left] (2,5.5); 
\filldraw (2.5, 5.5) circle[radius=.6mm]; 
\draw 
(2,5.5) to[bend right] (2.5,6) 
(2.5,5) to[bend left] (3,5.5); 
\filldraw (3.5, 5.5) circle[radius=.6mm]; 
\draw 
(3,5.5) to[bend right] (3.5,6) 
(3.5,5) to[bend left] (4,5.5); 
\filldraw (4.5, 5.5) circle[radius=.6mm]; 
\draw 
(4,5.5) to[bend right] (4.5,6) 
(4.5,5) to[bend left] (5,5.5); 
\filldraw (5.5, 5.5) circle[radius=.6mm]; 
\draw 
(5,5.5) to[bend right] (5.5,6); 
\filldraw (0.5, 4.5) circle[radius=.6mm]; 
\draw 
(0,4.5) to (1,4.5) 
(0.5,4) to (0.5,5); 
\filldraw (1.5, 4.5) circle[radius=.6mm]; 
\draw 
(1,4.5) to[bend right] (1.5,5) 
(1.5,4) to[bend left] (2,4.5); 
\filldraw (2.5, 4.5) circle[radius=.6mm]; 
\draw 
(2,4.5) to[bend right] (2.5,5) 
(2.5,4) to[bend left] (3,4.5); 
\filldraw (3.5, 4.5) circle[radius=.6mm]; 
\draw 
(3,4.5) to[bend right] (3.5,5) 
(3.5,4) to[bend left] (4,4.5); 
\filldraw (4.5, 4.5) circle[radius=.6mm]; 
\draw 
(4,4.5) to[bend right] (4.5,5); 
\filldraw (0.5, 3.5) circle[radius=.6mm]; 
\draw 
(0,3.5) to[bend right] (0.5,4) 
(0.5,3) to[bend left] (1,3.5); 
\filldraw (1.5, 3.5) circle[radius=.6mm]; 
\draw 
(1,3.5) to (2,3.5) 
(1.5,3) to (1.5,4); 
\filldraw (2.5, 3.5) circle[radius=.6mm]; 
\draw 
(2,3.5) to[bend right] (2.5,4) 
(2.5,3) to[bend left] (3,3.5); 
\filldraw (3.5, 3.5) circle[radius=.6mm]; 
\draw 
(3,3.5) to[bend right] (3.5,4); 
\filldraw (0.5, 2.5) circle[radius=.6mm]; 
\draw 
(0,2.5) to[bend right] (0.5,3) 
(0.5,2) to[bend left] (1,2.5); 
\filldraw (1.5, 2.5) circle[radius=.6mm]; 
\draw 
(1,2.5) to[bend right] (1.5,3) 
(1.5,2) to[bend left] (2,2.5); 
\filldraw (2.5, 2.5) circle[radius=.6mm]; 
\draw 
(2,2.5) to[bend right] (2.5,3); 
\filldraw (0.5, 1.5) circle[radius=.6mm]; 
\draw 
(0,1.5) to[bend right] (0.5,2) 
(0.5,1) to[bend left] (1,1.5); 
\filldraw (1.5, 1.5) circle[radius=.6mm]; 
\draw 
(1,1.5) to[bend right] (1.5,2); 
\filldraw (0.5, 0.5) circle[radius=.6mm]; 
\draw 
(0,0.5) to[bend right] (0.5,1); 
\end{tikzpicture}\
\right\}
\]
so $\shdom{y}=\{(2,1)\}$,
and  $(3,1)$ is an outer corner.
In this case, $\fPhi(y,3) = \{532614,456123\}$.
As predicted by the theorem with $(j,i) = (3,1)$,
both elements of $\fPhi(y,3)$ are fpf-dominant since 
\[
\FP(532614) = 
\left\{
\begin{tikzpicture}[scale=0.3,baseline=(b.base)]
\node (b) at (0,3) {}; 
\filldraw (0.5, 5.5) circle[radius=.6mm]; 
\draw 
(0,5.5) to[bend right] (0.5,6) 
(0.5,5) to[bend left] (1,5.5); 
\filldraw (1.5, 5.5) circle[radius=.6mm]; 
\draw 
(1,5.5) to[bend right] (1.5,6) 
(1.5,5) to[bend left] (2,5.5); 
\filldraw (2.5, 5.5) circle[radius=.6mm]; 
\draw 
(2,5.5) to[bend right] (2.5,6) 
(2.5,5) to[bend left] (3,5.5); 
\filldraw (3.5, 5.5) circle[radius=.6mm]; 
\draw 
(3,5.5) to[bend right] (3.5,6) 
(3.5,5) to[bend left] (4,5.5); 
\filldraw (4.5, 5.5) circle[radius=.6mm]; 
\draw 
(4,5.5) to[bend right] (4.5,6) 
(4.5,5) to[bend left] (5,5.5); 
\filldraw (5.5, 5.5) circle[radius=.6mm]; 
\draw 
(5,5.5) to[bend right] (5.5,6); 
\filldraw (0.5, 4.5) circle[radius=.6mm]; 
\draw 
(0,4.5) to (1,4.5) 
(0.5,4) to (0.5,5); 
\filldraw (1.5, 4.5) circle[radius=.6mm]; 
\draw 
(1,4.5) to[bend right] (1.5,5) 
(1.5,4) to[bend left] (2,4.5); 
\filldraw (2.5, 4.5) circle[radius=.6mm]; 
\draw 
(2,4.5) to[bend right] (2.5,5) 
(2.5,4) to[bend left] (3,4.5); 
\filldraw (3.5, 4.5) circle[radius=.6mm]; 
\draw 
(3,4.5) to[bend right] (3.5,5) 
(3.5,4) to[bend left] (4,4.5); 
\filldraw (4.5, 4.5) circle[radius=.6mm]; 
\draw 
(4,4.5) to[bend right] (4.5,5); 
\filldraw (0.5, 3.5) circle[radius=.6mm]; 
\draw 
(0,3.5) to (1,3.5) 
(0.5,3) to (0.5,4); 
\filldraw (1.5, 3.5) circle[radius=.6mm]; 
\draw 
(1,3.5) to[bend right] (1.5,4) 
(1.5,3) to[bend left] (2,3.5); 
\filldraw (2.5, 3.5) circle[radius=.6mm]; 
\draw 
(2,3.5) to[bend right] (2.5,4) 
(2.5,3) to[bend left] (3,3.5); 
\filldraw (3.5, 3.5) circle[radius=.6mm]; 
\draw 
(3,3.5) to[bend right] (3.5,4); 
\filldraw (0.5, 2.5) circle[radius=.6mm]; 
\draw 
(0,2.5) to (1,2.5) 
(0.5,2) to (0.5,3); 
\filldraw (1.5, 2.5) circle[radius=.6mm]; 
\draw 
(1,2.5) to[bend right] (1.5,3) 
(1.5,2) to[bend left] (2,2.5); 
\filldraw (2.5, 2.5) circle[radius=.6mm]; 
\draw 
(2,2.5) to[bend right] (2.5,3); 
\filldraw (0.5, 1.5) circle[radius=.6mm]; 
\draw 
(0,1.5) to[bend right] (0.5,2) 
(0.5,1) to[bend left] (1,1.5); 
\filldraw (1.5, 1.5) circle[radius=.6mm]; 
\draw 
(1,1.5) to[bend right] (1.5,2); 
\filldraw (0.5, 0.5) circle[radius=.6mm]; 
\draw 
(0,0.5) to[bend right] (0.5,1); 
\end{tikzpicture}\
\right\}
\quand
\FP(456123) = 
\left\{
\begin{tikzpicture}[scale=0.3,baseline=(b.base)]
\node (b) at (0,3) {}; 
\filldraw (0.5, 5.5) circle[radius=.6mm]; 
\draw 
(0,5.5) to[bend right] (0.5,6) 
(0.5,5) to[bend left] (1,5.5); 
\filldraw (1.5, 5.5) circle[radius=.6mm]; 
\draw 
(1,5.5) to[bend right] (1.5,6) 
(1.5,5) to[bend left] (2,5.5); 
\filldraw (2.5, 5.5) circle[radius=.6mm]; 
\draw 
(2,5.5) to[bend right] (2.5,6) 
(2.5,5) to[bend left] (3,5.5); 
\filldraw (3.5, 5.5) circle[radius=.6mm]; 
\draw 
(3,5.5) to[bend right] (3.5,6) 
(3.5,5) to[bend left] (4,5.5); 
\filldraw (4.5, 5.5) circle[radius=.6mm]; 
\draw 
(4,5.5) to[bend right] (4.5,6) 
(4.5,5) to[bend left] (5,5.5); 
\filldraw (5.5, 5.5) circle[radius=.6mm]; 
\draw 
(5,5.5) to[bend right] (5.5,6); 
\filldraw (0.5, 4.5) circle[radius=.6mm]; 
\draw 
(0,4.5) to (1,4.5) 
(0.5,4) to (0.5,5); 
\filldraw (1.5, 4.5) circle[radius=.6mm]; 
\draw 
(1,4.5) to[bend right] (1.5,5) 
(1.5,4) to[bend left] (2,4.5); 
\filldraw (2.5, 4.5) circle[radius=.6mm]; 
\draw 
(2,4.5) to[bend right] (2.5,5) 
(2.5,4) to[bend left] (3,4.5); 
\filldraw (3.5, 4.5) circle[radius=.6mm]; 
\draw 
(3,4.5) to[bend right] (3.5,5) 
(3.5,4) to[bend left] (4,4.5); 
\filldraw (4.5, 4.5) circle[radius=.6mm]; 
\draw 
(4,4.5) to[bend right] (4.5,5); 
\filldraw (0.5, 3.5) circle[radius=.6mm]; 
\draw 
(0,3.5) to (1,3.5) 
(0.5,3) to (0.5,4); 
\filldraw (1.5, 3.5) circle[radius=.6mm]; 
\draw 
(1,3.5) to (2,3.5) 
(1.5,3) to (1.5,4); 
\filldraw (2.5, 3.5) circle[radius=.6mm]; 
\draw 
(2,3.5) to[bend right] (2.5,4) 
(2.5,3) to[bend left] (3,3.5); 
\filldraw (3.5, 3.5) circle[radius=.6mm]; 
\draw 
(3,3.5) to[bend right] (3.5,4); 
\filldraw (0.5, 2.5) circle[radius=.6mm]; 
\draw 
(0,2.5) to[bend right] (0.5,3) 
(0.5,2) to[bend left] (1,2.5); 
\filldraw (1.5, 2.5) circle[radius=.6mm]; 
\draw 
(1,2.5) to[bend right] (1.5,3) 
(1.5,2) to[bend left] (2,2.5); 
\filldraw (2.5, 2.5) circle[radius=.6mm]; 
\draw 
(2,2.5) to[bend right] (2.5,3); 
\filldraw (0.5, 1.5) circle[radius=.6mm]; 
\draw 
(0,1.5) to[bend right] (0.5,2) 
(0.5,1) to[bend left] (1,1.5); 
\filldraw (1.5, 1.5) circle[radius=.6mm]; 
\draw 
(1,1.5) to[bend right] (1.5,2); 
\filldraw (0.5, 0.5) circle[radius=.6mm]; 
\draw 
(0,0.5) to[bend right] (0.5,1); 
\end{tikzpicture}\
\right\}.
\]
\end{ex}

We may now prove the second half of Theorem~\ref{thm:inv-pipe-dream-formula},
concerning the polynomials $\iSfpf_z$.

\begin{thm}\label{fpf-pipe-schubert-thm}
If $z \in \Ifpf_n$ then $\iSfpf_z = \sum_{D \in \FP(z)} \prod_{(i,j) \in D} (x_i + x_j)$.
\end{thm}

\begin{proof}
It follows from \cite[Thm. 4.17]{HMP3} that 
if $(j,i) \in \ltriangneq_n$ is an outer corner of $z \in \Ifpf_n$
then
\be\label{ftransition-eq}  (x_i+x_j) \iSfpf_z = \sum_{u \in \fPhi(z,j)} \iSfpf_u.\ee
Moreover, as $n$ is even, \cite[Thm. 1.3]{HMP1} implies that we have 
\be\label{fpf-rev-eq}
\iSfpf_{n\cdots 321} = \prod_{1\leq i < j \leq n-i} (x_i + x_j).\ee
Let $\fBJS_z = \sum_{D \in \FP(z)} \prod_{(i,j) \in D} (x_i+x_j)$.
If $\ell(z) = \max\left\{\ell(y): y \in \Ifpf_n\right\}$ then
$z=n\cdots 321$
and the identity
$\iSfpf_z = \fBJS_z$
is equivalent to \eqref{fpf-rev-eq}.
Otherwise, the transpose of $\shdomneq{z}$ shifted up one row is a proper subset of 
$\shdomneq{n\cdots 321}^{\uparrow T}=\SD_{(n-2,n-4,\dots,4,2)}$ by Corollary~\ref{fpf-dom-subset-cor},
so $z$ has an outer corner 
$(j,i) \in \ltriangneq_n$ with $i+j \leq n$.
In this case, 
$(x_i+x_j) \iSfpf_z = \sum_{u \in \fPhi(z,j)} \iSfpf_u= \sum_{u \in \fPhi(z,j)} \fBJS_u = (x_i+x_j) \fBJS_z$
by \eqref{ftransition-eq},
induction,
and Theorem~\ref{thm:fpf-dominant-transition}.
Dividing by $x_i+x_j$ completes the proof.
\end{proof}

\begin{ex}
Continuing Example~\ref{fpf-dom-ex}, we have 
\[
 \iSfpf_{532614} =(x_2 + x_1)(x_3+ x_1)(x_4+x_1) \quand
\iSfpf_{456123} =(x_2 + x_1)(x_3 + x_1)(x_3 + x_2), 
\]
so 
$\iSfpf_{351624} = \frac{1}{x_3 + x_1}\left( \iSfpf_{532614}+\iSfpf_{456123}\right) = (x_2  + x_1)(x_1 + x_2 + x_3 + x_4)$.
\end{ex}

\section{Generating pipe dreams} 
\label{sec:generating-pipe-dreams}

Bergeron and Billey \cite{bergeron-billey} proved
that the set $\RP(w)$ 
 is generated by applying simple transformations to 
a unique ``bottom'' pipe dream. Here, we derive versions of this result for
the sets of
involution pipe dreams $\IP(y)$ and $\FP(z)$.
This leads to algorithms for computing the sets $\IP(y)$ and $\FP(z)$ 
that are much more efficient than the 
naive methods suggested by our original definitions.

\subsection{Ladder moves}

Let $D$ and $E$ be subsets of $\PP\times \PP$,
depicted as positions marked by ``$+$'' in a matrix.
If
$E$ is obtained from $D$ by replacing a subset of the form
\[ 
\arraycolsep=1.5pt
 \barr{cc}
 \cdot & \cdot \\
   + & + \\
   \vdots & \vdots \\
   + & + \\
   + & \cdot
   \earr
\qquad\text{by}\qquad
\arraycolsep=1.5pt
 \barr{cc}
 \cdot & + \\
   + & + \\
   \vdots & \vdots \\
   + & + \\
   \cdot & \cdot
   \earr \]
   then we say that 
$E$ is obtained from $D$ by a \emph{ladder move} and write $D \lessdot_\RP E$.
More formally:
\begin{defn}
We write $D \lessdot_\RP E$ if 
for some
integers $i< j$ and $k$
the following holds:
\begin{itemize}
\item  One has $\{i+1,i+2,\dots,j-1\}\times\{k,k+1\} \subset D$.
\item It holds that $(j,k) \in D$ but $(i,k),(i,k+1),(j,k+1)\notin D$.
\item  One has $E = D\setminus \{(j,k)\} \cup \{(i,k+1)\}$.
\end{itemize}
One can have $i+1=j$ in this definition, in which case the first condition holds vacuously. 
Let $<_\RP$ be the transitive closure of $\lessdot_\RP$. 
This relation is a strict partial order. 
Let $\sim_\RP$ denote the symmetric closure of the partial order $\leq_\RP$.
\end{defn}

%


Recall the definition of the bottom pipe dream $\Dbot(w)$ from 
\eqref{eq:dbot}.

\begin{thm}[{\cite[Thm. 3.7]{bergeron-billey}}]\label{bb-thm}
Let $w \in S_n$. Then \[\RP(w)=
\left\{
E   : \Dbot(w) \leq_\RP E
\right\}
=
\left\{
E  : \Dbot(w) \sim_\RP E
\right\}
.\]
Thus $\RP(w)$ is an upper and lower set of $\leq_\RP$, with unique minimum $\Dbot(w)$.
\end{thm}

Define $\leq^\textsf{chute}_\RP$ to be the partial order with $D \leq^\textsf{chute}_\RP E$ if and only if 
$E^T \leq_\RP D^T$, and let  $\Dtop(w) = \Dbot(w^{-1})^T$ for $w \in S_n$.
Then 
 $\RP(w)=
\left\{
E   :   E \leq^\textsf{chute}_\RP \Dtop(w)
\right\}
$
by Corollary~\ref{transpose-cor} and Theorem~\ref{bb-thm}. Bergeron and Billey 
\cite{bergeron-billey} refer to the covering relation in $\leq^\textsf{chute}_\RP$ as a \emph{chute move}.
In the next sections, we will see that there are natural versions of  
$\leq_\RP$ and $\Dbot(w)$ for (fpf-)involution pipe dreams. There do not seem to be good involution analogues of  
$\leq^{\textsf{chute}}_\RP$ or $\Dtop(w)$, however.

\subsection{Involution ladder moves}
\label{ss:inv-ladder}

To prove an analogue of Theorem~\ref{bb-thm} for involution pipe dreams,
we need to introduce a more general partial order $<_\IP$ on subsets of $\PP\times \PP$.
Again let $D$ and $E$ be subsets of $\PP\times \PP$.
Informally, we define $<_\IP$ 
to be the transitive closure of $\lessdot_\RP$ and the relation that has
$D \lessdot_\IP E$ whenever $E$ is obtained from $D$ by replacing a subset of the form
\begin{equation}
    \def\arraystretch{0.8}
\label{eq:inv-ladder}
\arraycolsep=1.5pt
 \begin{array}{cccccc}
  & & \nearrow & \nearrow & \nearrow & \nearrow \\
   & \cdot & \cdot & \cdot & \cdot\\
 \cdot & \cdot  & \cdot & \cdot\\
   + & \cdot  & \cdot\\
   + & + \\
   \vdots & \vdots \\
   + & +   \\
   + & \cdot & {\phantom{+}} &  {\phantom +} &  {\phantom+}
   \end{array}
\qquad\text{by}\qquad 
\arraycolsep=1.5pt
 \begin{array}{cccccc}
 & & \nearrow & \nearrow & \nearrow & \nearrow \\
   & \cdot & \cdot & \cdot & \cdot\\
 \cdot & \cdot  & \cdot & \cdot\\
   + & +  & \cdot\\
   + & + \\
   \vdots & \vdots \\
   + & +   \\
   \cdot & \cdot & {\phantom+} &  {\phantom+} &  {\phantom+}
   \end{array}
\end{equation}
where the upper parts of the antidiagonals with $\nearrow$ are required to be empty.
For example,  
\[
\arraycolsep=1.5pt
 \begin{array}{cccccc}
 \cdot & \cdot  & \cdot & \cdot\\
   + & \cdot  & \cdot & +\\
   + & \cdot & \cdot& \cdot\\
   \cdot& + & \cdot & \cdot 
   \end{array}
\ \lessdot_\IP\ 
\arraycolsep=1.5pt
 \begin{array}{cccccc}
 \cdot & \cdot  & \cdot & \cdot\\
   + & +  & \cdot & +\\
   \cdot & \cdot & \cdot& \cdot\\
   \cdot & + & \cdot & \cdot 
   \end{array}
   \qquand
\arraycolsep=1.5pt
 \begin{array}{cccccc}
 \cdot & \cdot  & \cdot & \cdot\\
   + & \cdot  & \cdot & +\\
   + & + & \cdot& \cdot\\
   + & \cdot & \cdot & \cdot 
   \end{array}
\ \lessdot_\IP\ 
\arraycolsep=1.5pt
 \begin{array}{cccccc}
 \cdot & \cdot  & \cdot & \cdot\\
   + & +  & \cdot & +\\
   + & + & \cdot& \cdot\\
   \cdot & \cdot & \cdot & \cdot 
   \end{array}
\]
but  
\[
   \arraycolsep=1.5pt
 \begin{array}{cccccc}
 \cdot & \cdot  & + & \cdot\\
   + & \cdot  & \cdot & \cdot \\
   + & + & \cdot& \cdot\\
   + & \cdot & \cdot & \cdot 
   \end{array}
\ \not<_\IP\ 
\arraycolsep=1.5pt
 \begin{array}{cccccc}
 \cdot & \cdot  & + & \cdot\\
   + & +  & \cdot & \cdot\\
   + & + & \cdot& \cdot\\
   \cdot & \cdot & \cdot & \cdot 
   \end{array}
   \qquand
   \arraycolsep=1.5pt
 \begin{array}{cccccc}
 \cdot & \cdot  & \cdot & +\\
   + & \cdot  & \cdot & \cdot \\
   + & + & \cdot& \cdot\\
   + & \cdot & \cdot & \cdot 
   \end{array}
\ \not<_\IP\ 
\arraycolsep=1.5pt
 \begin{array}{cccccc}
 \cdot & \cdot  & \cdot & +\\
   + & +  & \cdot & \cdot\\
   + & + & \cdot& \cdot\\
   \cdot & \cdot & \cdot & \cdot 
   \end{array}
   \]
   since the relevant antidiagonals in \eqref{eq:inv-ladder} are not empty.
The precise definition of $\lessdot_{\IP} $ is below:

\begin{defn}\label{i-ladder-def}
We write $ D \lessdot_{\IP} E$ if for some
integers $i< j$ and $k$ the following holds:
\begin{itemize}
\item  One has $\{i+1,i+2,\dots,j-1\}\times\{k,k+1\} \subset D$.
\item It holds that $(i,k),(j,k) \in D$ but $(i,k+1),(i,k+2),(j,k+1)\notin D$.
\item  One has $E = D\setminus \{(j,k)\} \cup \{(i,k+1)\}$.
\item The set $D$ contains no positions strictly northeast of and in the same antidiagonal as $(i,k-1)$, $(i,k)$, $(i,k+1)$, or $(i,k+2)$.
\end{itemize}
One may again have $i+1=j$, in which case the first condition holds vacuously.
We define $<_\IP$ to be the transitive closure of $\lessdot_\RP$ and $\lessdot_{\IP}$,
and write $\sim_\IP$ for the symmetric closure of 
$\leq_\IP$.
\end{defn}

Our goal is to show that $<_\IP$ defines a partial order on $\IP(z)$;
for an example of this poset, see Figure~\ref{ip-fig}.
To proceed, we must recall a few nontrivial properties of the set $\A(z)$ from Section~\ref{schub-sect}.

\begin{lem}[{\cite[Lem. 6.3]{HMP2}}]
\label{l:321-atoms}
Let $z \in \I_n$ and $w \in \A(z)$.
Then no subword $w(a)w(b)w(c)$ of $w(1)w(2)\cdots w(n)$ for $1\leq a<b<c\leq n$ has the form $(i-1)i(i+1)$
for any integer $1<i<n$.
\end{lem}

Fix $z \in \I_n$.
The \emph{involution code} of  $z $ is $\ic(z) = (\ic_1(z),\ic_2(z),\dots,\ic_n(z))$
with $\ic_i(z)$ the number of integers $j>i$ with $z(i) > z(j)$ and $i\geq z(j)$.
 Note that we always have $\ic_i(z) \leq i$.

Suppose $a_1<a_2<\dots<a_l$ are the integers $a\in [n]$ with $a \leq z(a)$
and  set $b_i = z(a_i)$.
Define $\alpha_{\min}(z) \in S_n$ to be the permutation
whose inverse is given in one-line notation by removing all repeated letters from
$b_1 a_1 b_2 a_2\cdots b_l a_l$.
For example, if $z = 4231 \in \I_4$ then the latter word is $412233$ and 
$\alpha_{\min}(z) = (4123)^{-1} = 2341 \in S_4$.
Additionally, $\ic(z) = c(\alpha_{\min}(z))$~\cite[Lem. 3.8]{HMP1}.

Finally, let $\prec_\A$ be the transitive closure of the relation on $S_n$ that has $v \prec_\A w$ whenever the inverses of
$v,w \in S_n$ have the same one-line representations outside of three consecutive positions where 
$v^{-1} = \cdots cab\cdots $ and $ w^{-1}=\cdots bca\cdots$
for some integers $a<b<c$.
The relation $\prec_\A$ is a strict partial order.
Let $\sim_\A$ denote the symmetric closure of the partial order $\preceq_\A$.

\begin{thm}[{\cite[\S6.1]{HMP2}}]
\label{t:atoms} Let $z \in \I_n$. Then
\[
\A(z) = \{w \in S_n: \alpha_{\min}(z) \preceq_\A w\} = \{w \in S_n: \alpha_{\min}(z) \sim_\A w\}.
\]
Thus $\A(z)$ is an upper and lower set of $\preceq_\A$, with unique minimum $\alpha_{\min}(z)$.
\end{thm}

For $z \in \I_n$, let
$\IP^+(z) = \bigsqcup_{w \in \A(z)} \RP(w) $,
so that $\IP(z) = \{D \in \IP^+(z) : D \subseteq \ltriang_n\}$.

\begin{lem}\label{inv-ladder-lem}
Let $z \in \I_n$. Suppose $D$ and $E$ are subsets of $\PP\times \PP$
with $D <_{\IP} E $. Then $D \in \IP^+(z)$ if and only if $E \in \IP^+(z)$. 
\end{lem}

\begin{proof}
If $D \lessdot_\RP E$ then 
we have
$D \in \IP^+(z)$ if and only if $E \in \IP^+(z)$
by Theorem~\ref{bb-thm}.
Assume $D \lessdot_{\IP}E$ and 
let $i<j$ and $k$ be as in Definition~\ref{i-ladder-def}.

Consider the reading order $\omega$ that lists the positions $(i,j)\in [n]\times [n]$ such that $(-j,i)$ increases lexicographically, i.e., the order that goes down column $n$, then down column $n-1$, and so on. In view of Theorem~\ref{thm:almost-symmetric},
we may assume without loss of generality that columns $1,2,\dots,k-1$
of $D$ and $E$ are both empty, since omitting these positions has the effect of truncating the same final sequence of letters from $\word(D,\omega)$ and $\word(E,\omega)$.

Suppose  $E \in \RP(w)\subseteq \IP^+(z)$  for some permutation $w \in \A(z)$.
To show that $D \in \IP^+(z)$, it  
 suffices by Theorem~\ref{t:atoms}
to check that 
$D \in \RP(v)$ for a permutation $v \prec_\A w$.

Consider the wiring diagram of $E$ and let $m,m+1$ and $m+2$ be the top  indices of the wires in the antidiagonals containing
the cells
$
(i,k),$ $(i,k+1)$, and $(i+1,k+1)
$,
respectively.
Since the northeast parts of these antidiagonals are empty,
it follows that as one goes from northeast to southwest,
wire $m$ of $E$ enters the top of the $+$ in cell $(i,k)$,
wire $m+1$ enters the top of the $+$ in cell $(i,k+1)$, and 
wire $m+2$ enters the right of the $+$ in cell $(i,k+1)$.
Tracing these wires through the wiring diagram of $E$, we see that they exit column $k$ on the left in relative order $m+2,m,m+1$.
Since we assume columns $1,2,\dots,k-1$ are empty,
the wires must arrive at the far left in the same relative order.
This means that there are numbers $a<b<c$ such that $w^{-1}(m)w^{-1}(m+1)w^{-1}(m+2)=bca$.

Moving the $+$ in cell $(i,k+1)$ of $E$ to $(j,k)$ gives $D$ by assumption.
This transformation only alters the trajectories of wires $m$, $m+1$ and $m+2$
and causes no pair of wires to cross more than once, so $D$ is a
reduced pipe dream for some $v \in S_n$.
By examining the wiring diagram of $D$, we see that $v^{-1}(m)v^{-1}(m+1)v^{-1}(m+2)=cab$,
so $v \prec_\A w$ and $D \in \IP^+(z)$ as needed.
The same considerations show that if $D \in \RP(v)$ for some $v \in \A(z)$ then  $E \in \RP(w)$
for a permutation $w$ with $v \prec_\A w$. In this case, it follows that $w \in \A(z)$ by Theorem~\ref{t:atoms}
so $E \in \IP^+(z)$.  
\end{proof}

We define the \emph{bottom involution pipe dream} of $z \in \I_n$ to be the set
\be
\iDbot(z) = \{(i,j) \in [n] \times [n]: j \leq \ic_i(z)\} \subseteq \ltriang_n.
\ee
Since  $\ic(z) = c(\alpha_{\min}(z))$, it follows 
by Theorem~\ref{thm:almost-symmetric} that $\iDbot(z)   = \Dbot(\alpha_{\min}(z))\in \IP(z).$

\begin{thm}\label{pre-t:inv-ladder}
Let $z \in \I_n$. Then
$
\IP^+(z) = \left\{E : \iDbot(z) \leq_\IP E\right\}
= \left\{E  : \iDbot(z) \sim_\IP E\right\}. 
$
Thus $\IP^+(z)$ is an upper and lower set of $\leq_\IP$, with unique minimum $\iDbot(z)$.
\end{thm}

\begin{proof}
Both sets are contained in $\IP^+(z)$ by Lemma~\ref{inv-ladder-lem}.
Note that $\IP^+(z)$ is finite since $\A(z)$ is finite and each set $\RP(w)$ is finite.
Suppose $ \iDbot(z) \neq E = \Dbot(w)$ for some $w \in \A(z)$.
In view of Theorem~\ref{bb-thm},
we need only show that 
there exists a subset $D\subset \PP\times \PP$ with $D\lessdot_\IP E$.

As we assume $w \neq \alpha_{\min}(z)$, 
it follows from Theorem~\ref{t:atoms} 
that there exists some $p \in [n-2]$ with $w^{-1}(p{+}2) < w^{-1}(p) < w^{-1}(p{+}1)$.
Set $i = w^{-1}(p{+}2)$, and choose $p$ to minimize $i$.
We claim that if $h < i$ then $w(h) < p$.
To show this, we argue by contradiction.
Suppose there exists $1\leq h < i$ with $w(h) \geq p$.
Choose $h$ with this property so that $w(h)$ is as small as possible.
Then $w(h) > p+2 \geq 3$, and by the minimality of $w(h)$, the values $w(h){-}1$ and $w(h){-}2$ appear after position $h$ in the word $w(1)w(2)\cdots w(n)$. 
Therefore, by Lemma~\ref{l:321-atoms}, the one-line representation of $w$ must have the form
$\cdots w(h) \cdots w(h){-}2\cdots w(h){-}1\cdots$.
This contradicts the minimality of $i$, so no such $h$ can exist. 

Let $j > i$ be minimal with $w(j) < w(i)$ and define 
$k = c_i(w) - 1$. It is evident from the definition of $i$ that such an index $j$ exists and that $k$ is positive.
Now consider Definition~\ref{i-ladder-def} applied to these values of $i<j$ and $k$.
It follows from the claim in the previous paragraph if $h<i$ then $c_h(w) - c_i(w) \leq  i-h -3$.
Therefore, we see that the required antidiagonals are empty.
The minimality of $j$ implies that $c_m(w) \geq c_i(w)$ for all $i < m < j$,
and since we must have $j \leq w^{-1}(p)$,
it follows that  $c_j(w) < c_i(w) - 1$.
We conclude that replacing position $(i,k+1)$ in $E$ by $(j,k)$
produces a subset $D$ with $D \lessdot_\IP E$, as we needed to show.
\end{proof}

\begin{thm}
\label{t:inv-ladder}
If $z \in \I_n$ then 
$ \IP(z) = \left\{ E \subseteq \ltriang_n : \iDbot(z) \leq_\IP E \right\}
=
 \left\{ E \subseteq \ltriang_n : \iDbot(z) \sim_\IP E \right\}
.$
\end{thm}

\begin{proof}
This is clear from Theorem~\ref{pre-t:inv-ladder} since 
$\ltriang_n$ is a lower set under $\leq_\IP$.
\end{proof}

\begin{figure}[h]
\begin{center}
{\small
\begin{tikzpicture}[xscale=1.2,yscale=1.2]
\node at (0,0) (a) {
$\arraycolsep=1.5pt
\begin{array}{ccccc}
\cdot &  {\phantom+} &  {\phantom+}  & {\phantom+}  & {\phantom+}  \\
\cdot & \cdot &    &    \\
+ & \cdot   &\cdot   \\
+ & +    &\cdot   &\cdot\\
+ &\cdot  &\cdot &\cdot&\cdot
\end{array}$
};

\node at (-4,3) (b1) {
$\arraycolsep=1.5pt
\begin{array}{ccccc}
\cdot &  {\phantom+} &  {\phantom+}  & {\phantom+}  & {\phantom+}  \\
\cdot & + &    &    \\
\cdot & \cdot   &\cdot   \\
+ & +   &\cdot  & \cdot   \\
+ &\cdot  &\cdot & \cdot & \cdot 
\end{array}$
};

\node at (0,3) (b2) {
$\arraycolsep=1.5pt
\begin{array}{ccccc}
\cdot &  {\phantom+} &  {\phantom+}  & {\phantom+}  & {\phantom+}  \\
\cdot & \cdot &    &    \\
+ & \cdot   & +  \\
+ & \cdot   & \cdot & \cdot  \\
+ &\cdot  & \cdot& \cdot & \cdot 
\end{array}$
};

\node at (4,3) (b3) {
$\arraycolsep=1.5pt
\begin{array}{ccccc}
\cdot &  {\phantom+} &  {\phantom+}  & {\phantom+}  & {\phantom+}  \\
\cdot & \cdot  &  &    &    \\
+ & +  &\cdot    \\
+ & +      &\cdot& \cdot  \\
\cdot &\cdot  &\cdot & \cdot & \cdot 
\end{array}$
};

\node at (-4,6) (c1) {
$\arraycolsep=1.5pt
\begin{array}{ccccc}
\cdot &  {\phantom+} &  {\phantom+}  & {\phantom+}  & {\phantom+}  \\
\cdot & + &    &    \\
\cdot & +   &\cdot   \\
+ & +    &\cdot  & \cdot  \\
\cdot &\cdot  &\cdot & \cdot & \cdot 
\end{array}$
};

\node at (0,6) (c2) {
$\arraycolsep=1.5pt
\begin{array}{ccccc}
\cdot &  {\phantom+} &  {\phantom+}  & {\phantom+}  & {\phantom+}  \\
\cdot & + &    &    \\
\cdot & \cdot& +    \\
+ & \cdot & \cdot   & \cdot   \\
+ &\cdot  & \cdot& \cdot & \cdot 
\end{array}$
};

\node at (4,6) (c3) {
$\arraycolsep=1.5pt
\begin{array}{ccccc}
\cdot &  {\phantom+} &  {\phantom+}  & {\phantom+}  & {\phantom+}  \\
\cdot & \cdot &    &    \\
+ & +    & + \\
+ & \cdot    &\cdot& \cdot  \\
\cdot &\cdot  & \cdot& \cdot & \cdot 
\end{array}$
};

\node at (0,9) (d2) {
$\arraycolsep=1.5pt
\begin{array}{ccccc}
\cdot &  {\phantom+} &  {\phantom+}  & {\phantom+}  & {\phantom+}  \\
\cdot & + &    &    \\
\cdot & + & +    \\
\cdot & \cdot  & \cdot & \cdot     \\
+ &\cdot & \cdot & \cdot & \cdot 
\end{array}$
};

\node at (4,9) (d3) {
$\arraycolsep=1.5pt
\begin{array}{ccccc}
\cdot &  {\phantom+} &  {\phantom+}  & {\phantom+}  & {\phantom+}  \\
\cdot & + &    &    \\
+ & +    & + \\
\cdot & \cdot    &\cdot & \cdot \\
\cdot &\cdot  & \cdot& \cdot & \cdot 
\end{array}$
};

\node at (0,12) (e) {
$\arraycolsep=1.5pt
\begin{array}{ccccc}
\cdot &  {\phantom+} &  {\phantom+}  & {\phantom+}  & {\phantom+}  \\
\cdot & + &    &    \\
\cdot & + & +    \\
\cdot & +  & \cdot  & \cdot    \\
\cdot &\cdot & \cdot & \cdot & \cdot 
\end{array}$
};

\draw[->] (a) -- (b1);
\draw[->] (a) -- (b2);
\draw[red,->,dashed] (a) -- (b3);

\draw[->] (b2) -- (c2);
\draw[->] (b1) -- (c2);
\draw[->] (b1) -- (c1);
\draw[red,->,dashed] (b3) -- (c3);

\draw[->] (c2) -- (d2);
\draw[->] (c3) -- (d3);
\draw[->] (d2) -- (e);

\end{tikzpicture}}
\end{center}
\caption{Hasse diagram of $(\IP(z), <_\IP)$ for $z=(3,6)(4,5) \in \I_6$. The dashed red arrows indicate the covering relations of the form $D \lessdot_\IP E$.}
\label{ip-fig}
\end{figure}

\subsection{Fixed-point-free involution ladder moves}

In this subsection, we assume $n$ is a positive even integer.
Our goal is to replicate the results in Section~\ref{ss:inv-ladder} for fixed-point-free involutions.
To this end, 
we introduce a third partial order $<_\FP$.
Again let $D$ and $E$ be subsets of $\PP\times \PP$.
We define $<_\FP$ as the transitive closure of $\lessdot_\RP$ and 
the relation that has $D \lessdot_\FP E$ whenever $E$ is obtained from $D $
by replacing a subset of the form
\begin{equation}
\label{eq:fpf-ladder}
\arraycolsep=1.5pt
\def\arraystretch{0.8}
 \begin{array}{ccccccc}
 {\phantom+} & {\phantom+} & \nearrow & \nearrow & \nearrow & \nearrow & \nearrow \\
   {\phantom+} & \cdot & \cdot & \cdot &\cdot & \cdot \\
\cdot & \cdot & \cdot  & \cdot & \cdot\\
\cdot &   + & \cdot  & \cdot\\
  & + & + \\
  & \vdots & \vdots \\
 &  + & +   \\
  & + & \cdot &  {\phantom+} & {\phantom+} &  {\phantom+}&  {\phantom+}
      \end{array}
\qquad\text{by}\qquad
\arraycolsep=1.5pt
 \begin{array}{ccccccc}
 {\phantom+} & {\phantom+} & \nearrow & \nearrow & \nearrow & \nearrow & \nearrow \\
  {\phantom+}  & \cdot & \cdot & \cdot &\cdot & \cdot \\
\cdot & \cdot & \cdot  & \cdot & \cdot\\
+ &   + & \cdot  & \cdot\\
  & + & + \\
  & \vdots & \vdots \\
 &  + & +   \\
  & \cdot & \cdot &  {\phantom+} & {\phantom+} &  {\phantom+}&  {\phantom+}
      \end{array}
\end{equation}
Here, all positions containing ``$\ \cdot\ $'' should be empty, including 
the five antidiagonals extending upwards beyond each $\nearrow$.
For example,  
\[
\arraycolsep=1.5pt
 \begin{array}{cccccc}
 \cdot & \cdot & \cdot  & \cdot & \cdot\\
 \cdot &   + & \cdot  & \cdot & +\\
 + &   + & \cdot & +& \cdot\\
 \cdot &   \cdot& + & \cdot & \cdot 
   \end{array}
\ \lessdot_\FP\ 
\arraycolsep=1.5pt
 \begin{array}{cccccc}
 \cdot & \cdot & \cdot  & \cdot & \cdot\\
 + &   + & \cdot  & \cdot & +\\
 + &   \cdot & \cdot & +& \cdot\\
 \cdot &   \cdot & + & \cdot & \cdot 
   \end{array}
   \qquand
\arraycolsep=1.5pt
 \begin{array}{cccccc}
 \cdot & \cdot & \cdot  & \cdot & \cdot\\
 \cdot &   + & \cdot  & \cdot & +\\
 \cdot &   + & + & \cdot& \cdot\\
  + &  + & \cdot & + & \cdot 
   \end{array}
\ \lessdot_\FP\ 
\arraycolsep=1.5pt
 \begin{array}{cccccc}
 \cdot & \cdot & \cdot  & \cdot & \cdot\\
 + &   + & \cdot  & \cdot & +\\
 \cdot &   + & + & \cdot& \cdot\\
 + &   \cdot & \cdot & + & \cdot 
   \end{array}
\]
but  
\[
   \arraycolsep=1.5pt
 \begin{array}{cccccc}
\cdot &  \cdot & \cdot  & \cdot & +\\
 \cdot &   + & \cdot  & \cdot & \cdot\\
 \cdot &   + & \cdot & +& \cdot\\
 \cdot &   \cdot& + & \cdot & \cdot 
   \end{array}
\ \not<_\FP\ 
\arraycolsep=1.5pt
 \begin{array}{cccccc}
 \cdot & \cdot & \cdot  & \cdot & +\\
 + &   + & \cdot  & \cdot & \cdot \\
 \cdot &   \cdot & \cdot & +& \cdot\\
 \cdot &   \cdot & + & \cdot & \cdot 
   \end{array}
      \qquand
\arraycolsep=1.5pt
 \begin{array}{cccccc}
 + & \cdot & \cdot  & \cdot & \cdot\\
 \cdot &   + & \cdot  & \cdot & +\\
 \cdot &   + & + & \cdot& \cdot\\
  \cdot &  + & \cdot & + & \cdot 
   \end{array}
\ \not<_\FP\ 
\arraycolsep=1.5pt
 \begin{array}{cccccc}
 + & \cdot & \cdot  & \cdot & \cdot\\
 + &   + & \cdot  & \cdot & +\\
 \cdot &   + & + & \cdot& \cdot\\
 \cdot &   \cdot & \cdot & + & \cdot 
   \end{array}
   \]
   since the relevant antidiagonals in \eqref{eq:fpf-ladder} are not empty.
The precise definition of $\lessdot_\FP$ is as follows:

\begin{defn}
\label{fpf-ladder-def}
We write $ D \lessdot_{\FP} E$ if for some
integers  $0<i<j$ and $k\geq 2$ the following holds:
\begin{itemize}
\item  One has $\{i+1,i+2,\dots,j-1\}\times\{k,k+1\} \subset D$.
\item It holds that $(i,k),(j,k) \in D$ but $(i,k-1),(i,k+1),(i,k+2),(j,k+1)\notin D$.
\item  One has $E = D\setminus \{(j,k)\} \cup \{(i,k-1)\}$.
\item The set $D$ contains no positions strictly northeast of and in the same antidiagonal as $(i,k-2)$, $(i,k-1)$, $(i,k)$, $(i,k+1)$, or $(i,k+2)$.
\end{itemize}
When $i+1=j$, the first condition holds vacuously; see the lower dashed arrow in Figure~\ref{fp-fig}.
Define $<_\FP$ to be the transitive closure of $\lessdot_\RP$ and $\lessdot_{\FP}$.
Write $\sim_\FP$ for the symmetric closure of 
$\leq_\FP$.
\end{defn}

We will soon show that $<_\FP$ defines a partial order on $\FP(z)$, as one can see in the example
shown in Figure~\ref{fp-fig}.
For this, we will need a lemma from \cite{can-joyce-wyser} concerning the set $\Afpf(z)$.

\begin{lem}[{\cite[Cor. 2.16]{can-joyce-wyser}}]
\label{l:fpf-atoms}
Let $w \in S_{n}$ and $z \in \Ifpf_{n}$.
Then $w \in \Afpf(z)$ if and only if for all $a,b,c,d \in [n]$
with $a<b=z(a)$ and $c<d=z(c)$, the following holds:
\begin{enumerate}
\item[(1)] One has $w(a) = 2i-1$ and $w(b) = 2i$ for some $i \in [n/2]$.
\item[(2)] If $a<c$ and $b < d$, then $w(b) < w(c)$.
\end{enumerate}

\end{lem}

The involution code and partial order $\prec_\A$ both have fixed-point-free versions.
Fix $z \in \Ifpf_n$. 
The \emph{fpf-involution code} of $z$ is the integer sequence 
\[
\icfpf(z) = (\icfpf_1(z), \icfpf_2(z), \dots, \icfpf_n(z))
\] where $\icfpf_i(z)$ is the number of integers $j > i$ with $z(i) > z(j)$ and $i > z(j)$.
It always holds that $\icfpf_i(z) < i$.
If $a_1<a_2<\dots <a_{n/2}$ are the numbers $a\in [n]$ with $a< z(a)$ and $b_i = z(a_i)$,
then let
\[ \betamin(z) = (a_1b_1a_2b_2\dots a_{n/2} b_{n/2})^{-1} 
= s_1  s_3 s_5\cdots s_{n-1}  \alpha_{\min}(z) 
\in S_n.\]
For example, if $z = 632541 \in \Ifpf_6$ then 
we
have
$\betamin(z) = (162345)^{-1} = 134562 \in S_6$.
One can check that $\icfpf(z) = c(\betamin(z))$ \cite[Lem. 3.8]{HMP1}.

Define $\prec_{\Afpf}$ to be the transitive closure of the relation in $S_n$
that has
$v \prec_{\Afpf} w$ whenever the inverses of $v,w \in S_n$ have the same one-line representations outside of four consecutive positions where
$v^{-1} = \cdots adbc \cdots$ and $w^{-1} = \cdots bcad \cdots$ for some integers $a<b<c<d$.
This is a strict partial order on $S_n$.
Let $\sim_{\Afpf}$ denote the symmetric closure of the partial order $\preceq_{\cAfpf}$.

\begin{thm}[{\cite[\S6.2]{HMP2}}]
\label{t:fpf-order}
Let $z \in \Ifpf_n$. Then
\[
\Afpf(z) = \left\{w \in S_n: \betamin(z) \preceq_{\Afpf} w \right\}
=
 \left\{w \in S_n: \betamin(z) \sim_{\Afpf} w \right\}
.
\]
Thus $\Afpf(z)$ is an upper and lower set of $\preceq_{\Afpf}$, with unique minimum $\betamin(z)$.
\end{thm}

For $z \in \Ifpf_n$, let $\FP^+(z) = \bigsqcup_{w\in\Afpf(z)} \RP(w)$,
so $\FP(z) = \FP^+(z) \cap \ltriangneq_n$.

\begin{lem}\label{fpf-ladder-lem}
Let $z \in \Ifpf_n$. Suppose $D$ and $E$ are subsets of $\PP\times \PP$
with $D <_{\FP} E $. Then $D \in \FP^+(z)$ if and only if $E \in \FP^+(z)$. 
\end{lem}

 \begin{proof}
If $D \lessdot_\RP E$ then 
the result follows 
by Theorem~\ref{bb-thm}.
Assume $D \lessdot_{\FP}E$ and 
let $i<j$ and $k$ be as in Definition~\ref{fpf-ladder-def}.

As in the proof of Theorem~\ref{inv-ladder-lem}, consider the reading order $\omega$ that lists the positions $(i,j)\in [n]\times [n]$ such that $(-j,i)$ increases lexicographically. 
In view of Theorem~\ref{thm:fpf-almost-symmetric},
we may assume without loss of generality
that columns $1,2,\dots,k-2$, as well as all positions below row $i$ in column $k-1$, are empty in both
of $D$ and $E$.
This follows since  omitting these positions has the effect of truncating the same final sequence of letters from $\word(D,\omega)$ and $\word(E,\omega)$.

Assume $E \in \RP(v) \subseteq \FP^+(z)$ for some $w \in \Afpf(z)$.
To show that $D \in \FP^+(z)$, 
we will check that $D \in \RP(v)$ for some $v\in S_n$ with $v \prec_{\Afpf} w$.

Consider the wiring diagram of $E$ and let $m$, $m+1$, $m+2$, and $m+3$ be the top indices of the wires in the antidiagonals containing the cells $(i,k-1)$, $(i,k)$, $(i,k+1)$, and $(i,k+2)$, respectively.
Since the northeast parts of these antidiagonals are empty,
it follows that as one goes from northeast to southwest,
wire $m$ of $E$
enters the top of the $+$ in cell $(i,k-1)$, wire $m+1$ enters the top of the $+$ is cell $(i,k)$,
wire $m+2$ enters the right of the $+$ in cell $(i,k)$, and wire $m+3$ enters the 
top of cell $(i+1,k+1)$, which contains a $+$ if $i+1<j$.
Tracing these wires through the wiring diagram of $E$, we see that they exit column $k-1$ on the left in relative order $m+2$, $m$, $m+1$, $m+3$.
Since we assume that $D$ and $E$ contain no positions in the rectangle weakly southwest of $(i+1,k-1)$,
the wires must arrive at the far left in the same relative order.
This means that  $w^{-1}(m)w^{-1}(m+1) w^{-1}(m+2) w^{-1}(m+3) = bcad$
for some numbers $a<b<c<d$.

Moving the $+$ in cell $(i,k-1)$ of $E$ to $(j,k)$ gives $D$ by assumption. This transformation only alters the trajectories of wires $m$, $m+1$, $m+2$, and $m+3$
and causes no pair of wires to cross more than once, so $D$ is a 
reduced pipe dream for some $v \in S_n$.
By examining the wiring diagram of $D$, it is easy to check that $v^{-1}(m)v^{-1}(m+1) v^{-1}(m+2) v^{-1}(m+3) = adbc$
so $v \prec_{\Afpf} w$ as needed.

If  instead $D \in \RP(v)\subset \FP^+(z)$
for some $v \in \Afpf(z)$, then a similar argument shows that 
$E \in \RP(w)$ for some $w \in S_n$ with $v \prec_{\Afpf} w$, which implies that $E \in \FP^+(z)$ by Theorem~\ref{t:fpf-order}.
\end{proof}

We define the \emph{bottom fpf-involution pipe dream} of $z \in \Ifpf_n$ to be the set
\be
\iDbotfpf(z)  =  \left\{(i,j) \in [n]\times [n]: j \leq \icfpf_i(z)\right\} \subseteq \ltriangneq_n.
\ee
Since $\icfpf(z) = c(\betamin(z))$, 
Theorem~\ref{thm:fpf-almost-symmetric} implies that 
$
\iDbotfpf(z) = \Dbot(\betamin(z)) \in \FP(z).
$

\begin{thm}\label{pre-t:fpf-ladder}
Let $z \in \Ifpf_n$. Then
$
\FP^+(z) = \left\{E : \iDbotfpf(z) \leq_\FP E\right\}
= \left\{E  : \iDbotfpf(z) \sim_\FP E\right\}. 
$
Thus $\FP^+(z)$ is an upper and lower set of $\leq_\FP$, with unique minimum $\iDbotfpf(z)$.
\end{thm}

\begin{proof}
Both sets are contained in $\FP^+(z)$ by Lemma~\ref{fpf-ladder-lem},
and the set $\FP^+(z)$ is clearly finite.
Suppose $ \iDbotfpf(z) \neq E = \Dbot(w)$ for some $w \in \Afpf(z)$.
As in the proof of Theorem~\ref{pre-t:inv-ladder},
it suffices to show that
there exists a subset $D\subset \PP\times \PP$ with $D\lessdot_\FP E$.

Since  $w \neq \betamin(z)$, 
 Lemma~\ref{l:fpf-atoms} and Theorem~\ref{t:fpf-order} imply
that there exists an odd integer $p \in [n - 3]$ such that $
w^{-1}(p)w^{-1}(p+1)w^{-1}(p+2)w^{-1}(p+3) = bcad$ for some numbers $a<b<c<d$.
Choose $p$ such that $a$ is as small as possible.
We claim that $a<w^{-1}(q)$ for all  $q$ with $p+3 <q \leq n$.
To show this, let $a_0=a$ and $b_0  =d$ and suppose $a_i $ and $b_i$ are the integers
such that 
\[w^{-1}(p+2)w^{-1}(p+3)\cdots w^{-1}(n) = a_0b_0 a_1 b_1 \cdots a_k b_k.\]
Part (1) of Lemma~\ref{l:fpf-atoms} implies that $a_i < b_i =z(a_i)$ for all $i$, so it suffices to show 
 that $a_0<a_i$ for $i \in [k]$. This holds since 
if $i \in [k]$ were minimal with $a_i < a_0$,
then it would follow from part (2) of Lemma~\ref{l:fpf-atoms} that $a_i < a_{i-1} < b_{i-1}< b_i$, 
contradicting the minimality of $a$.
 
Now, to match Definition~\ref{fpf-ladder-def}, let $i =  a=w^{-1}(p+2)$, define $j > i$ to be minimal with $w(j) < w(i)$, and set $k = c_i(w)$.
 It is clear from the definition of $i$ that such an index $j$ exists and that $k\geq 2$.
The claim in the previous paragraph shows that if $1\leq h < i$ then $h$ must appear before position $p$ in the one-line representation
of $w^{-1}$, which means that   $w(h)<p$ and therefore $c_h(w) - c_i(w) \leq i-h- 4$.
The antidiagonals described in Definition~\ref{fpf-ladder-def} are thus empty as needed.
Since  $j \leq b= w^{-1}(p)$,
it follows that  $c_j(w) < c_i(w)$;
moreover, if $i < m < j$ then $w(m) > w(d) = p+3$ so $c_m(w) \geq c_i(w) + 1$.
Collecting these observations,  
we conclude that replacing $(i,k-1)$ in $E$ with $(j,k)$
gives a subset $D$ with $D \lessdot_\IP E$, as we needed to show.
\end{proof}

\begin{thm}
\label{t:fpf-ladder}
If $z \in \Ifpf_n$ then 
\[ \FP(z) = \left\{ E \subseteq \ltriangneq_n : \iDbotfpf(z) \leq_\FP E \right\}
=
 \left\{ E \subseteq \ltriangneq_n : \iDbotfpf(z) \sim_\FP E \right\}
.\]
\end{thm}

\begin{proof}
This is clear from Theorem~\ref{pre-t:fpf-ladder} since 
$\ltriangneq_n$ is a lower set under $\leq_\FP$.
\end{proof}

\begin{figure}[h] 
\begin{center}
{\scriptsize
\begin{tikzpicture}[xscale=1.1,yscale=0.725]

\node at (0,0) (a) {
$\arraycolsep=1.5pt
\begin{array}{cccccc}
\cdot &  {\phantom+} &  {\phantom+}  & {\phantom+}  & {\phantom+}  & {\phantom+}\\
\cdot & \cdot &    &    \\
\cdot & \cdot   &\cdot   \\
+ & \cdot    &\cdot   &\cdot\\
+ &+  &\cdot &\cdot&\cdot \\
+ &+  &\cdot &\cdot&\cdot & \cdot 
\end{array}$
};

\node at (-4,4) (b1) {
$\arraycolsep=1.5pt
\begin{array}{cccccc}
\cdot &  {\phantom+} &  {\phantom+}  & {\phantom+}  & {\phantom+}  & {\phantom+}\\
\cdot & \cdot &    &    \\
\cdot & +   &\cdot   \\
\cdot & \cdot    &\cdot   &\cdot\\
+ &+  &\cdot &\cdot&\cdot \\
+ &+  &\cdot &\cdot&\cdot & \cdot 
\end{array}$
};

\node at (0,4) (b2) {
$\arraycolsep=1.5pt
\begin{array}{cccccc}
\cdot &  {\phantom+} &  {\phantom+}  & {\phantom+}  & {\phantom+}  & {\phantom+}\\
\cdot & \cdot &    &    \\
\cdot & \cdot   &\cdot   \\
+ & \cdot    &+   &\cdot\\
+ &\cdot &\cdot &\cdot&\cdot \\
+ &+  &\cdot &\cdot&\cdot & \cdot 
\end{array}$
};

\node at (-4,8) (c1) {
$\arraycolsep=1.5pt
\begin{array}{cccccc}
\cdot &  {\phantom+} &  {\phantom+}  & {\phantom+}  & {\phantom+}  & {\phantom+}\\
\cdot & \cdot &    &    \\
\cdot & +   &\cdot   \\
\cdot & \cdot    &+   &\cdot\\
+ &\cdot  &\cdot &\cdot&\cdot \\
+ &+  &\cdot &\cdot&\cdot & \cdot 
\end{array}$
};

\node at (0,8) (c2) {
$\arraycolsep=1.5pt
\begin{array}{cccccc}
\cdot &  {\phantom+} &  {\phantom+}  & {\phantom+}  & {\phantom+}  & {\phantom+}\\
\cdot & \cdot &    &    \\
\cdot & +   &\cdot   \\
\cdot & \cdot    &+   &\cdot\\
+ &\cdot  &\cdot &\cdot&\cdot \\
+ &+  &\cdot &\cdot&\cdot & \cdot 
\end{array}$
};

\node at (4,8) (c3) {
$\arraycolsep=1.5pt
\begin{array}{cccccc}
\cdot &  {\phantom+} &  {\phantom+}  & {\phantom+}  & {\phantom+}  & {\phantom+}\\
\cdot & \cdot &    &    \\
\cdot & \cdot   &\cdot   \\
+ & \cdot    &+   &\cdot\\
+ &\cdot  &+ &\cdot&\cdot \\
+ &\cdot  &\cdot &\cdot&\cdot & \cdot 
\end{array}$
};

\node at (-4,12) (d1) {
$\arraycolsep=1.5pt
\begin{array}{cccccc}
\cdot &  {\phantom+} &  {\phantom+}  & {\phantom+}  & {\phantom+}  & {\phantom+}\\
\cdot & \cdot &    &    \\
\cdot & +   &\cdot   \\
\cdot & +    &+   &\cdot\\
\cdot &\cdot  &\cdot &\cdot&\cdot \\
+ &+  &\cdot &\cdot&\cdot & \cdot 
\end{array}$
};

\node at (0,12) (d2) {
$\arraycolsep=1.5pt
\begin{array}{cccccc}
\cdot &  {\phantom+} &  {\phantom+}  & {\phantom+}  & {\phantom+}  & {\phantom+}\\
\cdot & \cdot &    &    \\
\cdot & +   &\cdot   \\
\cdot & \cdot    &+   &\cdot\\
+ &\cdot  &+ &\cdot&\cdot \\
+ &\cdot  &\cdot &\cdot&\cdot & \cdot 
\end{array}$
};

\node at (4,12) (d3) {
$\arraycolsep=1.5pt
\begin{array}{cccccc}
\cdot &  {\phantom+} &  {\phantom+}  & {\phantom+}  & {\phantom+}  & {\phantom+}\\
\cdot & \cdot &    &    \\
\cdot & \cdot   &\cdot   \\
+ & +    &+   &\cdot\\
+ &\cdot  &\cdot &\cdot&\cdot \\
+ &\cdot  &\cdot &\cdot&\cdot & \cdot 
\end{array}$
};

\node at (0,16) (e1) {
$\arraycolsep=1.5pt
\begin{array}{cccccc}
\cdot &  {\phantom+} &  {\phantom+}  & {\phantom+}  & {\phantom+}  & {\phantom+}\\
\cdot & \cdot &    &    \\
\cdot & +   &\cdot   \\
\cdot & +    &+   &\cdot\\
\cdot &\cdot  &+ &\cdot&\cdot \\
+ &\cdot  &\cdot &\cdot&\cdot & \cdot 
\end{array}$
};

\node at (4,16) (e2) {
$\arraycolsep=1.5pt
\begin{array}{cccccc}
\cdot &  {\phantom+} &  {\phantom+}  & {\phantom+}  & {\phantom+}  & {\phantom+}\\
\cdot & \cdot &    &    \\
\cdot & +   &\cdot   \\
+ & +    &+   &\cdot\\
\cdot &\cdot  &\cdot &\cdot&\cdot \\
+ &\cdot  &\cdot &\cdot&\cdot & \cdot 
\end{array}$
};

\node at (0,20) (f1) {
$\arraycolsep=1.5pt
\begin{array}{cccccc}
\cdot &  {\phantom+} &  {\phantom+}  & {\phantom+}  & {\phantom+}  & {\phantom+}\\
\cdot & \cdot &    &    \\
\cdot & +   &\cdot   \\
\cdot & +    &+   &\cdot\\
\cdot &+  &+ &\cdot&\cdot \\
\cdot &\cdot  &\cdot &\cdot&\cdot & \cdot 
\end{array}$
};

\node at (4,20) (f2) {
$\arraycolsep=1.5pt
\begin{array}{cccccc}
\cdot &  {\phantom+} &  {\phantom+}  & {\phantom+}  & {\phantom+}  & {\phantom+}\\
\cdot & \cdot &    &    \\
\cdot & +   &\cdot   \\
+ & +    &+   &\cdot\\
\cdot &+  &\cdot &\cdot&\cdot \\
\cdot &\cdot  &\cdot &\cdot&\cdot & \cdot 
\end{array}$
};

\node at (4,24) (g1) {
$\arraycolsep=1.5pt
\begin{array}{cccccc}
\cdot &  {\phantom+} &  {\phantom+}  & {\phantom+}  & {\phantom+}  & {\phantom+}\\
\cdot & \cdot &    &    \\
+ & +   &\cdot   \\
+ & +    &+   &\cdot\\
\cdot &\cdot  &\cdot &\cdot&\cdot \\
\cdot &\cdot  &\cdot &\cdot&\cdot & \cdot 
\end{array}$
};

\draw[->] (a) -- (b1);
\draw[->] (a) -- (b2);

\draw[->] (b1) -- (c1);
\draw[->] (b2) -- (c2);
\draw[->] (b2) -- (c3);

\draw[->] (c1) -- (d1);
\draw[->] (c1) -- (d2);
\draw[->] (c2) -- (d1);
\draw[->] (c2) -- (d2);
\draw[->] (c3) -- (d2);
\draw[red,->,dashed] (c3) -- (d3);

\draw[->] (d1) -- (e1);
\draw[->] (d2) -- (e1);
\draw[->] (d3) -- (e2);

\draw[->] (e1) -- (f1);
\draw[->] (e2) -- (f2);

\draw[red,->,dashed] (f2) -- (g1);

\end{tikzpicture}}
\end{center}
\caption{Hasse diagram of $(\FP(z), <_\FP)$ for $z=(1,2)(3,7)(4,8)(5,6) \in \Ifpf_6$. The dashed red arrows indicate the covering relations of the form $D \lessdot_\FP E$.
}\label{fp-fig}
\end{figure}

\section{Future directions}
\label{sec:future}

In this final section we discuss some related identities and open problems.

\subsection{Enumerating involution pipe dreams}
\label{sec:enumeration}

Choose $w \in S_n$ and let $p =\ell(w)$.
Macdonald~\cite[(6.11)]{macdonald1991notes} proved that the following specialization of a Schubert 
polynomial gives an exact formula for the number of reduced
pipe dreams for $w$:
\be\label{bhy-eq}
|\RP(w)| = \fkS_w(1,1,\dots,1) = \frac{1}{p!}\sum_{(a_1,a_2,\dots,a_{p}) \in \R(w)} a_1 a_2\cdots a_{p}.
\ee
Recall that $\kappa(y)$ is the number of 2-cycles in $y \in\I_n$.
For $D \in \IP(y)$, define 
$\wt(D) = 2^{\kappa(y) -d_D}$ where $d_D $ is the number of diagonal positions in $D$.
For  $\cX \subseteq  \IP(y)$, define $\|\cX \| = \sum_{D \in \cX} \wt(D)$.
\begin{cor}\label{bhy-cor}
Suppose $y \in \I_n$, $z \in \Ifpf_{2n}$, and $p = \ellhat(y) =\ellfpf(z)$.
\ben
\item[(a)] $\|\IP(y)\| = 2^{\kappa(y)} \iS_y(\tfrac{1}{2},\tfrac{1}{2},\dots,\tfrac{1}{2}) = \frac{1}{2^pp!}\sum_{(a_1,a_2,\dots,a_{p}) \in \iR(y)} 2^{\kappa(y)}a_1 a_2\cdots a_{p}$.

\item[(b)] $ |\FP(z)| =  \iSfpf_z(\tfrac{1}{2},\tfrac{1}{2},\dots,\tfrac{1}{2}) = \frac{1}{2^pp!}\sum_{(a_1,a_2,\dots,a_{p}) \in \iRfpf(z)} a_1 a_2\cdots a_{p}$.
\een
\end{cor}

\begin{proof}
In both parts, 
the first equality is immediate from Theorem~\ref{thm:inv-pipe-dream-formula}
and the second equality is a consequence of \eqref{bhy-eq}, via Definitions~\ref{inv-sch-def} and \ref{fpf-inv-sch-def}.
\end{proof}

Billey, Holroyd, and Young gave the first bijective proof of \eqref{bhy-eq} 
(and of a more general $q$-analogue) in the recent paper 
 \cite{billey2019bijective}.
This follow-up problem is natural:

\begin{problem}
Find bijective proofs of the identities in Corollary~\ref{bhy-cor}.
\end{problem}

For some permutations, better formulas  than \ref{bhy-eq} are available.
A \emph{reverse plane partition} of shape $D \subset \PP\times \PP$
is a map $T: D \to \NN$ such that 
$ T(i,j) \leq T(i+1,j)$ and $T(i,j) \leq T(i,j+1)$
for all relevant $(i,j) \in D$.
If $\lambda$ is a partition, 
then let $\RPP_\lambda(k)$ be the set of reverse plane partitions of Ferrers shape 
$\D_\lambda  = \{ (i,j) \in \PP\times \PP : j \leq \lambda_i\}$
 with entries in $\{0,1,\dots,k\}$.

Given $w \in S_n$ write $1^k \times w$ for the permutation in $S_{n+k}$ 
that fixes $1,2,\dots,k$ while mapping $i+k \mapsto w(i) + k$ for $i \in [n]$.
Fomin and Kirillov \cite[Thm. 2.1]{fomin1997reduced} showed that if $w \in S_n$ is dominant then
$
|\RP(1^k \times w)| = |\RPP_{\lambda}(k)|
$
for the partition $\lambda$ with $\dom{w} = \D_\lambda$.
In particular:
\be
|\RP(1^k \times n \cdots 321)| =  |\RPP_{(n-1,\dots,3,2,1)}(k)| = \prod_{1 \leq i < j \leq n} \frac{ i + j+2k -1}{i + j -1}.
\ee
Serrano and Stump gave a  bijective proof of this identity in \cite{serrano2012maximal}.
There are similar formulas 
counting (weighted) involution pipe dreams. For example:
\begin{prop}
Let  $g_n = n{+1} \dots 2n\ 1 \dots n \in \I_{2n}$.
Then
	\[|\IP(1^k \times g_n)| = |\RPP_{(n,\dots,3,2,1)}(\lfloor k/2 \rfloor)|
\quad\text{for all $k \in \NN$.}\]
\end{prop}

\begin{proof}
It follows from Theorem~\ref{t:atoms}
that $\A(g_n) = \{w_n \}$ for $w_n = 246 \cdots (2n) 1 35 \cdots (2n-1) \in S_{2n}$, and that
$\A(1^k\times g_n) = \{1^k \times w_n\}$.
Moreover, we have $\iDbot(1^k\times g_n) = \{ (j+k,i) : 1 \leq i \leq j \leq n\}$.
From these facts, we see that 
$\IP(1^k \times g_n)$ is connected by ordinary ladder moves that are \emph{simple} in that they replace a single cell $(i,j)$ by $(i-1,j+1)$.

Now consider all ways of filling the cells $(i,j)\in \iDbot(1^k\times g_n)$ by numbers  $a \in \{0,1,\dots,\lfloor k/2\rfloor\}$
such that rows are weakly increasing and columns are weakly decreasing.
The set of such fillings is obviously in bijection with $\RPP_{(n,\dots,3,2,1)}(\lfloor k/2 \rfloor)$.
On the other hand, we can transform such a filling into 
a subset of $\ltriang_n$ by replacing each cell $(i,j)$ filled with $a$ by $(i-a,j+a)$.
It is easy to see that this operation is a bijection from our set of fillings to  $\IP(1^k\times g_n)$.
\end{proof}

Computations indicate that
if $k,n \in \NN$ and $\{p,q\}= \{\lfloor n/2 \rfloor, \lceil n/2 \rceil\}$
then
\begin{equation}
\label{eq:pp}
\ba \|\IP(1^k \times n \cdots 321)\| &= \prod_{i=1}^{p}
\prod_{j=1}^{q} \frac{i+j+k-1}{i+j-1}
\ea\ee
and
\be\label{eq:pp2}
\ba
|\FP(\idfpf_{2k} \times 2n \cdots 321)| 
	 &= \prod_{\substack{i,j \in [n], i\neq j }} \frac{i + j +2k-1}{i + j -1}.
	 \ea
\ee
The right-hand side of \eqref{eq:pp}
is the number of reverse plane partitions with entries at most $k$ of shifted shape $\SD_\lambda= \{ (i,i+j-1) : (i,j) \in \D_\lambda\}$ for
$\lambda= (p+q-1,p+q-3,p+q-5,\dots)$ \cite{proctor1983shifted}.
Similar formulas should  hold for
$\|\IP(1^k \times y)\|$ and $|\FP(\idfpf_{2k} \times z)|$ when $y \in\I_n$ and $z \in \Ifpf_{2n}$ are any (fpf-)dominant involutions.
We expect that one can prove such identities algebraically using the Pfaffian formulas  for $\iS_y$ and $\iSfpf_z$
in  \cite[\S5]{Pawlowski}.
A more interesting open problem is the following:

\begin{problem}
Find bijective proofs of \eqref{eq:pp} and \eqref{eq:pp2} and their dominant generalizations.
\end{problem}

\subsection{Ideals of matrix Schubert varieties} \label{sec:ideals}

Another open problem is to find a geometric explanation for the formulas in Theorem~\ref{thm:inv-pipe-dream-formula}. Such an explanation exists in the double Schubert case, as we briefly explain.

Recall that $A_{[i][j]}$ denotes the upper-left $i\times j$ submatrix of a matrix $A$.
Let $\mathcal{Z}$ be the matrix of indeterminates $(z_{ij})_{i,j \in [n]}$. For $w \in S_n$, let $I_w \subseteq \CC[z_{ij} : i,j \in [n]] = \CC[\Mat_n]$ be the ideal generated by all $(\rank w_{[i][j]}+1) \times (\rank w_{[i][j]}+1)$ minors of $\mathcal{Z}_{[i][j]}$ for $i,j \in [n]$.  The vanishing locus of $I_w$ in the space $\Mat_n$
of $n\times n$ complex matrices is exactly the matrix Schubert variety $\MX_w$.

Let $\init(I_w)$ be the \emph{initial ideal} of leading terms in $I_w$ with respect to any term order on $\CC[z_{ij}]$ with the property that the leading term of $\det(A)$ for any submatrix $A$ of $\mathcal{Z}$ is the product of the antidiagonal entries of $A$. For instance, lexicographic order with the variable ordering $z_{1n} < \cdots < z_{11} < z_{2n} < \cdots < z_{21} < \cdots$ has this property.

\begin{thm}[\cite{knutson-miller}]
For each permutation $w \in S_n$, the ideal $I_w$ is prime, 
and there is a prime decomposition
$
    \init(I_w) = \bigcap_{D \in \RP(w)} (z_{ij} : (i,j) \in D)
$. 
\end{thm}

Given the description of the class $[\MX_w]$ in \S \ref{subsec:eq-cohom}, this result implies the pipe dream formula \eqref{eq:double-schubert-def} for $[\MX_w] = \fkS_w(x,\yvar)$.

Now let
 $\hat{\mathcal{Z}}$ be the symmetric matrix 
 of indeterminates $[z_{\max(i,j), \min(i,j)}]_{i,j \in [n]}$.
 Define $\hat I_y \subseteq \CC[z_{ij} : 1 \leq j < i \leq n] = \CC[\SM_n]$ for $y \in \I_n$
 to be the ideal generated by all $(1+\rank y_{[i][j]}) \times (1+\rank y_{[i][j]})$ minors of $\hat{\mathcal{Z}}_{[i][j]}$ for $i,j \in [n]$. 
 The vanishing locus of $\hat I_y$ is $\iMX_y$.

\begin{conj} \label{conj:ideals} For $y \in \I_n$, the ideal $\hat I_y$ is prime, and there is a primary decomposition of $\init(\hat I_y)$ whose top-dimensional components are $\left(z_{ij}^{\cwt{i}{j}{D}} : (i,j) \in D\right)$ for $D \in \IP(y)$, where
$\cwt{i}{j}{D}=2$ if the pipes crossing at $(i,j)$ are labeled $p$ and $z(p)$ for some $p\in[n]$, and otherwise $\cwt{i}{j}{D}=1$.
 \end{conj} 

As per \S\ref{subsec:eq-cohom}, the conjecture would give a direct geometric proof of Theorem~\ref{thm:inv-pipe-dream-formula}.

\begin{ex}
Let $y = 1243 = (3,4) \in \I_4$.  Then $A \in \iMX_y$ if and only if $\rank A_{[i][j]} \leq m_{ij}$ for
\begin{equation*}
(m_{ij})_{1\leq i,j\leq 4} = 
\left(\begin{array}{cccc}
1 & 1 & 1 & 1\\
1 & 2 & 2 & 2\\
1 & 2 & 2 & 3\\
1 & 2 & 3 & 4
\end{array}\right).
\end{equation*} 
These rank conditions all follow from $\rank A_{[3][3]} \leq 2$, so $\hat I_y$ is generated by $\det \hat{\mathcal{Z}}_{[3][3]}$.
 The ideals in the primary decomposition $\init(\hat I_y) = (z_{31}^2 z_{22}) = (z_{31}^2) \cap (z_{22})$ correspond to the two involution pipe dreams
in the set
$
\IP(y) = \{ \{ (3,1)\}, \{(2,2)\}\}
$ for $y=(3,4)$.
\end{ex}

\begin{ex}
Let $y = 14523 = (2,4)(3,5) \in \I_5$. One computes that
\begin{equation*}
\init(\hat I_y) = (z_{21}^2, z_{31} z_{21}, z_{22}z_{31}, z_{31}^2, z_{32}z_{31}, z_{32}^2) = (z_{21}^2, z_{31}, z_{32}^2) \cap (z_{21}, z_{22}, z_{31}^2, z_{32}).
\end{equation*}
There is a single involution pipe dream for $y$ given by $\{(2,1),(3,1),(3,2)\}.$
This pipe dream corresponds to the codimension $3$ component $(z_{21}^2, z_{31}, z_{32}^2)$ of $\init(\hat I_y)$, while the codimension $4$ component $(z_{21}, z_{22}, z_{31}^2, z_{32})$ does not correspond to a pipe dream of $y$.
\end{ex}
For $z \in \Ifpf_n$, the ideal generated by the $(\rank z_{[i][j]}+1) \times (\rank z_{[i][j]}+1)$ of a skew-symmetric matrix of indeterminates need not be prime, and we do not have an analogue of Conjecture~\ref{conj:ideals}.

\subsection{Pipe dream formulas for $K$-theory}
\label{sec:k-theory}

A third source of open problems 
concerns pipe dreams for $K$-theory.
A subset  $D\subseteq [n]\times [n]$ is a \emph{$K$-theoretic pipe dream} for $w \in S_n$ if 
$\word(D) = (a_1,a_2,\dots,a_p)$
satisfies
$
s_{a_1} \circ s_{a_2} \circ \dots \circ s_{a_p} = w
$
where $\circ$ is
the Demazure product.
Let $\KRP(w)$ denote the set of $K$-theoretic pipe dreams of $w$.
Fomin and Kirillov \cite{FominKirillov94} introduce these objects in order 
to state this formula for
the \emph{(generalized) Grothendieck polynomial} $\fkG_w$ of 
a permutation $w\in S_n$:
\begin{equation}
\label{eq:groth}
\fkG_w = \sum_{D \in \KRP(w)} \beta^{|D| - \ell(w)} \prod_{(i,j) \in D} x_i \in \ZZ[\beta][x_1,x_2,\dots,x_n].
\end{equation}
This identity is nontrivial to derive from 
 Lascoux and Sch\"utzenberger's original definition of Grothendieck polynomials in terms of isobaric divided difference operators~\cite{lascoux1990anneau,lascoux1983symmetry}.
 
 Grothendieck polynomials becomes Schubert polynomials on setting $\beta=0$.
 In \cite{MP2019,MP2019a}, continuing work of Wyser and Yong \cite{wyser-yong-orthogonal-symplectic}, the second two authors studied \emph{orthogonal} and \emph{symplectic Grothendieck polynomials} $\fkG^{\O}_y$ and $\fkG^{\Sp}_z$ indexed by $y \in \I_n$ and $z \in \Ifpf_n$.
 These polynomials likewise recover the involution Schubert polynomials $\iS_y$ and $\iSfpf_z$ on setting $\beta=0$,
 and it would be interesting to know if they have analogous pipe dream formulas.

The symplectic case of this question is more tractable.
The polynomials $\fkG^{\Sp}_z$ have a formulation in terms of isobaric divided difference operators due to Wyser and Yong \cite{wyser-yong-orthogonal-symplectic},
which suggests a natural $K$-theoretic variant of the set $\FP(z)$.
A formula for $\fkG^{\Sp}_z$ involving these objects appears in \cite[\S4]{MP2020}.
 By contrast, no simple algebraic formula is known for the  polynomials $\fkG^{\O}_y$.
It is a nontrivial problem even to identify 
the correct $K$-theoretic generalization of $\IP(y)$.

 \begin{problem}
Find a pipe dream formula for 
the polynomials $\fkG^{\O}_y$
involving an appropriate ``$K$-theoretic'' generalization of the sets of involution pipe dreams 
$\IP(y)$.
\end{problem}

%
%




\end{document}

%% file: preamble.tex
\usepackage{xcolor}
\usepackage{latexsym}
\usepackage{amssymb}
\usepackage{amsthm}
\usepackage{amscd}
\usepackage{amsmath}
\usepackage{mathrsfs}
\usepackage{graphicx}
\usepackage{hyperref}
\usepackage{shuffle}
\usepackage{mathtools}
\usepackage{mathdots}
\usepackage[colorinlistoftodos]{todonotes}
\usetikzlibrary{chains,scopes}
\usetikzlibrary{shapes.geometric,positioning}

\numberwithin{equation}{section}
\theoremstyle{definition}\newtheorem{thm}{Theorem}[section]
\theoremstyle{definition}\newtheorem{lem}[thm]{Lemma}
\theoremstyle{definition}\newtheorem{prop}[thm]{Proposition}
\theoremstyle{definition}\newtheorem{conj}[thm]{Conjecture}
\theoremstyle{definition}\newtheorem{cor}[thm]{Corollary}
\theoremstyle{definition} \newtheorem{defn}[thm]{Definition}
\theoremstyle{definition} \newtheorem{ex}[thm]{Example}
\theoremstyle{definition} \newtheorem*{rmk}{Remark}
\theoremstyle{definition}
\theoremstyle{definition}\newtheorem{problem}[thm]{Problem}

\newcommand{\CC}{\mathbb{C}}
\newcommand{\ZZ}{\mathbb{Z}}
\newcommand{\NN}{\ZZ_{\geq 0}}

\renewcommand{\O}{\mathsf{O}}

\newcommand{\fpf}{\textsf{{FPF}}}

\newcommand{\I}{\mathcal{I}}
\newcommand{\A}{\mathcal{A}}
\newcommand{\Afpf}{\mathcal{A}^{\fpf}}
\newcommand{\cAfpf}{\Afpf}

\newcommand{\R}{\mathcal{R}}

\newcommand{\fkS}{\mathfrak{S}}
\newcommand{\fkG}{\mathfrak{G}}
\newcommand{\iS}{\hat{\mathfrak{S}}}
\newcommand{\iSfpf}{\hat{\mathfrak{S}}^{\fpf}}
\newcommand{\idfpf}{1^{\fpf}}

\newcommand{\Fl}[1]{\textsf{Fl}_{#1}}
\newcommand{\GL}{\mathsf{GL}}
\newcommand{\Borel}{\mathsf{B}}
\newcommand{\Torus}{\mathsf{T}}
\newcommand{\STorus}{\mathsf{S}}
\newcommand{\Sp}{\mathsf{Sp}}

\DeclareMathOperator{\rank}{rank}

\DeclareMathOperator{\wt}{wt}

\DeclareMathOperator{\init}{\textsf{init}}

\newcommand{\RP}{\mathcal{PD}}
\newcommand{\KRP}{\mathcal{KPD}}
\newcommand{\IRP}{\mathcal{ID}}
\newcommand{\IP}{\mathcal{ID}}
\newcommand{\FP}{\mathcal{FD}}
\newcommand{\iell}{\hat\ell}
\newcommand{\iellfpf}{\hat\ell^{\fpf}}
\newcommand{\ellfpf}{\iellfpf}
\newcommand{\Ifpf}{\I^{\fpf}}

\newcommand{\RPP}{\mathsf{RPP}}

\def\cX{\mathcal{X}}

\renewcommand{\R}{\mathcal{R}}
\newcommand{\iR}{\hat{\mathcal{R}}}
\newcommand{\iRfpf}{\hat{\mathcal{R}}^{\fpf}}

\newcommand{\cwt}[3]{m_{#1#2,#3}}

\newcommand{\Dbot}{D_{\text{bot}}}
\newcommand{\Dtop}{D_{\text{top}}}
\newcommand{\iDbot}{\hat D_{\text{bot}}}
\newcommand{\iDbotfpf}{\hat D_{\text{bot}}^{\fpf}}

\newcommand{\ic}{\hat c}
\newcommand{\icfpf}{\hat c^{\fpf}}

\newcommand{\iPhi}{\Psi}

\newcommand{\fPhi}{ \Psi^{\fpf}}

\newcommand{\rtriang}{\raisebox{-0.5pt}{\tikz{\draw (0,.25) -- (.25,.25) -- (0,0) -- (0,.25);}}\hspace{-0.5mm}}
\newcommand{\ltriang}{\raisebox{-0.5pt}{\tikz{\draw (0,0) -- (.25,0) -- (0,.25) -- (0,0);}}}
\newcommand{\ltriangeq}{\ltriang}
\newcommand{\ltriangneq}{\ltriang^{\!\!\neq}}

\newcommand{\dom}[1]{\operatorname{dom}(#1)}
\newcommand{\shdom}[1]{\operatorname{shdom}(#1)}
\newcommand{\shdomneq}[1]{\operatorname{shdom}^{\neq}(#1)}
\newcommand{\NWleq}{\leq_{\textsf{NW}}}

\newcommand{\iBJS}{\mathfrak{A}}
\newcommand{\fBJS}{\mathfrak{B}}

\newcommand{\word}{\textsf{word}}
\newcommand{\NEleq}{\leq_{\textsf{NE}}}

\newcommand{\maxdes}{\operatorname{maxdes}}

\def\ellhat{\iell}
\newcommand{\adiag}{\textsf{adiag}}
\def\PP{\ZZ_{>0}}

\def\ben{\begin{enumerate}}
\def\een{\end{enumerate}}
\def\[{\begin{equation*}}
\def\]{\end{equation*}}
\def\udiag{\textsf{udiag}}
\def\qquand{\qquad\text{and}\qquad}
\def\quand{\quad\text{and}\quad}
\def\be{\begin{equation}}
\def\ee{\end{equation}}

\def\ba{\begin{aligned}}
\def\ea{\end{aligned}}
\def\barr{\begin{array}}
\def\earr{\end{array}}

\def\D{\textsf{D}}
\def\SD{\textsf{SD}}

\def\yvar{y}

\def\Mat{\textsf{Mat}}
\def\SM{\textsf{SMat}}
\def\SSM{\textsf{SSMat}}

\newcommand{\MX}{M\hspace{-0.5mm}X}
\newcommand{\iX}{\hat X}
\newcommand{\iXfpf}{\hat X^{\fpf}}

\newcommand{\iMX}{M\hspace{-0.5mm}\hat X}
\newcommand{\iMXfpf}{M\hspace{-0.5mm}\hat X^{\fpf}}

\def\betamin{\alpha^\fpf_{\min}}

\newcommand{\Hom}{\operatorname{Hom}}
\newcommand{\Sym}{\operatorname{Sym}}
\newcommand{\mult}{\operatorname{mult}}

%% file: pipedreams-draft4.bbl
\begin{thebibliography}{99}


\bibitem{bergeron-billey} N. Bergeron and S. Billey. ``RC-graphs and Schubert polynomials''. Experiment. Math. 2.4 (1993), pp. 257--269.
\bibitem{billey2019bijective} S. C. Billey, A. E. Holroyd, and B. J. Young. ``A bijective proof of Macdonald's reduced word formula''. Algebr. Comb. 2.2 (2019), pp. 217--248.
\bibitem{BJK} S. C. Billey, W. Jockusch, and R. P. Stanley. ``Some Combinatorial Properties of Schubert Polynomials''. J. Algebraic Combin. 2 (1993), pp. 345--374.
\bibitem{brion} M. Brion. ``The behaviour at infinity of the Bruhat decomposition''. Comment. Math. Helv. 73 (1998), pp. 137--174.
\bibitem{can-joyce-wyser} M. B. Can, M. Joyce, and B. Wyser. ``Chains in Weak Order Posets Associated to Involutions''. J. Combin. Theory Ser. A 137 (2016), pp. 207--225.
\bibitem{fink2016matrix} A. Fink, J. Rajchgot, and S. Sullivant. ``Matrix Schubert varieties and Gaussian conditional independence models''. J. Algebr. Combin. 44.4 (2016), pp. 1009--1046.
\bibitem{FominKirillov94} S. Fomin and A. N. Kirillov. ``Grothendieck polynomials and the Yang-Baxter equation''. In:
Proceedings of the Sixth Conference in Formal Power Series and Algebraic Combinatorics, DIMACS (1994), pp. 184--190.
\bibitem{fomin1997reduced} S. Fomin and A. N. Kirillov. ``Reduced words and plane partitions''. J. Algebr. Combin. 6.4 (1997), pp. 311--319.
\bibitem{fomin-kirillov-yang-baxter} S. Fomin and A. N. Kirillov. ``The Yang-Baxter equation, symmetric functions, and Schubert polynomials''. Discrete Math. 153 (1996), pp. 123--143.
\bibitem{fomin1994schubert} S. Fomin and R. P. Stanley. ``Schubert polynomials and the nilCoxeter algebra''. Adv. Math. 103.2 (1994), pp. 196--207.
\bibitem{fulton1992flags} W. Fulton. ``Flags, Schubert polynomials, degeneracy loci, and determinantal formulas''. Duke Mathematical Journal 65.3 (1992), pp. 381--420.
\bibitem{HMP2} Z. Hamaker, E. Marberg, and B. Pawlowski. ``Involution words II: braid relations and atomic structures''. J. Algebr. Comb. 45 (2017), pp. 701--743.
\bibitem{HMP1} Z. Hamaker, E. Marberg, and B. Pawlowski. ``Involution words: counting problems and connections to Schubert calculus for symmetric orbit closures''. J. Combin. Theory Ser. A 160 (2018), pp. 217--260.
\bibitem{HMP3} Z. Hamaker, E. Marberg, and B. Pawlowski. ``Transition formulas for involution Schubert polynomials''. Selecta Math. (N.S.) 24.4 (2018), pp. 2991--3025.
\bibitem{hansson-hultman} M. Hansson and A. Hultman. ``A word property for twisted involutions in Coxeter groups''. J. Combin. Theory Ser. A 161 (2019), pp. 220--235.
\bibitem{harris-tu} J. Harris and L. Tu. ``On symmetric and skew-symmetric determinantal varieties''. Topology 23.1 (1984), pp. 71--84.
\bibitem{hu-zhang} J. Hu and J. Zhang. ``On involutions in symmetric groups and a conjecture of Lusztig''. Adv. Math. 287 (2016), pp. 1--30.
\bibitem{hultman-twisted-involutions} A. Hultman. ``The combinatorics of twisted involutions in Coxeter groups''. Trans. Amer. Math. Soc. 359.6 (2007), pp. 2787--2798.
\bibitem{hultman-twisted-involutions2} A. Hultman. ``Twisted identities in Coxeter groups''. J. Algebr. Combin. 28.2 (2008), pp. 313--332.
\bibitem{Humphreys} J. E. Humphreys. Reflection groups and Coxeter groups. Cambridge University Press, 1990.
\bibitem{incitti} F. Incitti. ``The Bruhat Order on the Involutions of the Symmetric Group''. J. Algebr.
Combin. 20.3 (2004), pp. 243--261.
\bibitem{Knutson} A. Knutson. ``Schubert polynomials, pipe dreams, equivariant classes, and a co-transition
formula''. Preprint (2019), arXiv:1909.13777.
\bibitem{knutson-miller} A. Knutson and E. Miller. ``Gr\"obner geometry of Schubert polynomials''. Ann. of Math.
(2) 161.3 (2005), pp. 1245--1318.
\bibitem{KM03} A. Knutson and E. Miller. ``Subword complexes in Coxeter groups''. Adv. Math. 184.1
(2004), pp. 161--176.
\bibitem{kohnert1997using} A. Kohnert and S. Veigneau. ``Using Schubert basis to compute with multivariate polynomials''. Adv. Appl. Math. 19.1 (1997), pp. 45--60.
\bibitem{lascoux1990anneau} A. Lascoux. ``Anneau de Grothendieck de la vari{\'e}t{\'e} de drapeaux''. The Grothendieck
Festschrift Volume III. Springer, 1990, pp. 1--34.
\bibitem{lascoux1982structure} A. Lascoux and M.-P. Sch\"utzenberger. ``Structure de Hopf de l'anneau de cohomologie et de l'anneau de Grothendieck d'une vari\'et\'e de drapeaux''. CR Acad. Sci. Paris S\'er. I Math 295.11 (1982), pp. 629--633.
\bibitem{lascoux1983symmetry} A. Lascoux and M.-P. Sch\"utzenberger. ``Symmetry and flag manifolds''. Invariant Theory. Springer. 1983, pp. 118--144.
\bibitem{macdonald1991notes} I. G. Macdonald. Notes on Schubert polynomials. Vol. 6. Laboratoire de combinatoire et d'informatique math\'ematique (LACIM), Universit\'e du Qu\'ebec \`a Montr\'eal, 1991.
\bibitem{M2018} E. Marberg. ``On some actions of the 0-Hecke monoids of affine symmetric groups''. J. Combin. Theory Ser. A 161 (2019), pp. 178--219.
\bibitem{MP2019} E. Marberg and B. Pawlowski. ``K-theory formulas for orthogonal and symplectic orbit closures''. Adv. Math. 372 (2020), 107299.
\bibitem{MP2019a} E. Marberg and B. Pawlowski. ``On some properties of symplectic Grothendieck polynomials''. J. Pure Appl. Algebra 225.1 (2021), 106463.
\bibitem{MP2020} E. Marberg and B. Pawlowski. ``Principal specializations of Schubert polynomials in classical types''. Algebr. Comb., to appear.
\bibitem{MZ2018} E. Marberg and Y. Zhang. ``Affine transitions for involution Stanley symmetric functions''. Preprint (2018), arXiv:1812.04880.
\bibitem{McGovern}  W. M. McGovern, The adjoint representation and the adjoint action, in Algebraic Quotients. Torus Actions and Cohomology. The Adjoint Representation and the Adjoint Action, Encyclopaedia Math. Sci. 131 (2002) 159, Springer, Germany (2002).
\bibitem{miller2004combinatorial} E. Miller and B. Sturmfels. Combinatorial commutative algebra. Vol. 227. Springer Science \& Business Media, 2004.
\bibitem{Pawlowski} B. Pawlowski. ``Universal graph Schubert varieties''. Transform. Groups, to appear.
\bibitem{proctor1983shifted} R. A. Proctor. ``Shifted plane partitions of trapezoidal shape''. Proc. Am. Math. Soc. 89.3
(1983), pp. 553--559.
\bibitem{RainsVazirani} E. M. Rains and M. J. Vazirani. ``Deformations of permutation representations of Coxeter
groups''. J. Algebr. Comb. 37 (2013), pp. 455--502.
\bibitem{richardson-springer} R. W. Richardson and T. A. Springer. ``The Bruhat order on symmetric varieties''. Geom.
Dedicata 35 (1990), pp. 389--436.
\bibitem{richardson-springer2} R. W. Richardson and T. A. Springer. ``Complements to: The Bruhat order on symmetric
varieties''. Geom. Dedicata 49 (1994), pp. 231--238.
\bibitem{serrano2012maximal} L. Serrano and C. Stump. ``Maximal Fillings of Moon Polyominoes, Simplicial Complexes,
and Schubert Polynomials''. Electron. J. Combin. 19.1 (2012), P16.
\bibitem{Viennot} X. G. Viennot. ``Heaps of pieces. I. Basic definitions and combinatorial lemmas''. Lecture
Notes in Math. 1234 (1986), pp. 321--350.
\bibitem{wyser-degeneracy-loci} B. J. Wyser. ``K-orbit closures on $G/B$ as universal degeneracy loci for flagged vector bundles with symmetric or skew-symmetric bilinear form''. Transform. Groups 18 (2013), pp. 557--594.
\bibitem{wyser-yong-orthogonal-symplectic} B. J. Wyser and A. Yong. ``Polynomials for symmetric orbit closures in the flag variety''. Transform. Groups 22 (2017), pp. 267--290.
\end{thebibliography}
